\documentclass[a4paper,DIV12,12pt]{article}
\usepackage[a4paper, top=20mm, bottom =20mm, textwidth=165mm]{geometry}
\usepackage{amsthm,amsmath,amssymb,amsfonts,tikz, amscd, physics}
\usepackage[matrix,arrow,curve]{xy}
\usepackage{enumitem}
\usepackage{color}
\usepackage[pdftex, bookmarksnumbered=true]{hyperref}
\usepackage{mathrsfs}
\usepackage[utf8]{inputenc}
\usepackage[normalem]{ulem}
\usepackage{tikz}
\usetikzlibrary{positioning, arrows.meta}
\usetikzlibrary{shapes.geometric,decorations.pathmorphing}

\usepackage{ytableau}

\graphicspath{{images/}}

\theoremstyle{plain}
\newtheorem{define}{Definition}[section]

\newtheorem{lemma}[define]{Lemma}%[section]

\newtheorem{proposition}[define]{Proposition}
\newtheorem*{proposition*}{Proposition}
\newtheorem{corollary}[define]{Corollary}
\newtheorem{theorem}[define]{Theorem}%[section]
\newtheorem{remark}[define]{Remark}%[section]

\newtheorem{claim}[define]{Claim}

\newcommand{\zh}{\textcolor{red}}

\numberwithin{equation}{section}

\newcommand{\re}{{\mathrm{e}}}

\title{Athinization of irreducible \(\widehat{\mathfrak{gl}}_n\)-modules with dominant highest weights
}

\usepackage[ backend=biber, maxnames=10, maxalphanames=5, minalphanames=5, style=alphabetic]{biblatex}
\DeclareFieldFormat[article]{title}{\mkbibemph{#1\isdot}}
\DeclareFieldFormat{journaltitle}{#1\isdot}
\renewbibmacro{in:}{%
  \ifentrytype{article}
    {}
    {\printtext{\bibstring{in}\intitlepunct}}}

\author{Mikhail Bershtein, Evgeny Mukhin, Leonid Rybnikov, and Aleksandr Trufanov}

\begin{document}

\maketitle

\begin{abstract}
We study the Gelfand--Tsetlin realization of generic Verma modules for the affine Lie algebra $\widehat{\mathfrak{gl}}_n$ by viewing them as {\it thin} modules over the affine Yangian $Y(\widehat{\mathfrak{sl}}_n)$. By results of \cite{feigin2011yangians}, these modules admit a basis indexed by periodic Gelfand--Tsetlin patterns with explicit formulas for the Yangian action, and we identify them with the evaluation modules introduced by Kodera~\cite{kodera2021guay}.

Our main result describes the specialization from generic highest weights to dominant highest weights (not necessarily integral). We call the resulting construction {\it athinization}: an irreducible \(\widehat{\mathfrak{gl}}_n\)-module, which is not thin as a module over the affine Kac--Moody algebra, is realized as a thin module over the larger (and ``more affine'') algebra \(Y(\widehat{\mathfrak{sl}}_n)\). Combinatorially, this realization is obtained by restricting the generic periodic Gelfand--Tsetlin basis to a distinguished subset of {\it permitted patterns}. We prove that the span of these patterns carries a well-defined affine Yangian action.

In particular, this construction yields explicit Gelfand--Tsetlin-type bases for admissible representations of $\widehat{\mathfrak{gl}}_n$ in the sense of Kac--Wakimoto, providing a new combinatorial realization of these modules. We compare the formulas for characters coming from this combinatorics with those for minimal models of  $W$-algebras of the type $A_n$ via the {\it principal specialization}.

Further, we obtain analogous results for representations of $U_q\widehat{\mathfrak{gl}}_n$ via their realization as thin modules over the quantum toroidal algebra of $\mathfrak{gl}_n$. 
\end{abstract}

\tableofcontents

\section{Introduction}

%{\lcom Preliminary version of the introduction, working on it. Feel free to add what you find necessary.}

\subsection{Affine Yangians and Gelfand--Tsetlin realizations of $\widehat{\mathfrak{gl}}_n$-modules}

The affine Yangian $Y(\widehat{\mathfrak{sl}}_n)$ and its trigonometric analog, the quantum toroidal algebra $ U_q(\ddot{\mathfrak{gl}}_n)$ of $\mathfrak{gl}_n$, %\zh{toroidal algebra should be gl not sl} 
play an important role in representation theory, algebraic geometry, and mathematical physics. One important class of representations of $Y(\widehat{\mathfrak{sl}}_n)$ consists of {\it thin} (or {\it tame}) modules, in which the Cartan subalgebra acts diagonally with simple joint spectrum. Such modules admit distinguished bases with explicit formulas for matrix elements of the generators.

In~\cite{feigin2011yangians}, Feigin, Finkelberg, Negut, and the third author showed that the generic Verma module over the affine Lie algebra $\widehat{\mathfrak{gl}}_n$ can be realized as a thin module over the affine Yangian $Y(\widehat{\mathfrak{sl}}_n)$. 
In this realization the module admits a natural basis indexed by periodic (or affine) Gelfand--Tsetlin patterns, and the action of the Yangian generators is given by explicit formulas for the matrix elements in this basis. 
These formulas are closely related to classical constructions of Gelfand--Tsetlin bases for representations of the Lie algebra $\mathfrak{gl}_n$. 
From the geometric perspective developed in~\cite{feigin2011yangians}, the generic Verma module arises from the equivariant cohomology of affine Laumon quasiflag spaces. 
In this picture the Gelfand--Tsetlin basis corresponds to the fixed points of a torus action, via Atiyah--Bott localization.
%, and the Shapovalov form is interpreted as the Poincar\'e pairing in equivariant cohomology.

The affine Yangian itself admits several equivalent presentations and is closely related to the quantum toroidal algebra. These structures were extensively studied in works of Tsymbaliuk~\cite{tsymbaliuk2010quantum}, Kodera~\cite{kodera2021guay}, and others. In particular, Kodera constructed evaluation homomorphisms from the affine Yangian $Y(\widehat{\mathfrak{sl}}_n)$ to completed enveloping algebras of $\widehat{\mathfrak{gl}}_n$, allowing one to regard affine Lie algebra modules as Yangian modules. The thin modules described in~\cite{feigin2011yangians} are naturally identified with pullbacks of highest weight modules under this evaluation map.

\subsection{Classical Gelfand--Tsetlin patterns.}
A (classical) Gelfand--Tsetlin pattern for an irreducible finite-dimensional $\mathfrak{gl}_n$-module with the highest weight $\mu$ is a triangular array
$\{\mu^{(i)}_j\}_{1\le j\le i\le n}$ with top row
$\mu=\mu^{(n)}=(\mu^{(n)}_1,\dots,\mu^{(n)}_n)$ and the interlacing inequalities
% \s{notation looks similar to notation \(n\)-tuples of Young diagrams in Section \(4.3\), but have different meaning. Do we care?} \zh{Maybe we can write $\lambda_{ij}$ instead of $\lambda_j^{(i)}$? This notation is often used in this context and it is similar to $d_{ij}$.} \s{It is possible. Also Lenya suggested to use \(\nu\) in introduction. I like both options.} 
% \s{Maybe \(\mu_{ij}\)? we always used \(\mu\) for \(\mathfrak{gl}(n)\) weights, aso it is good. } \zh{Yes, $\mu_{ij}$ looks like a best option.} 
% \s{I decided to preserve higher bracket because otherwise equality \(\mu=\mu_n\)  looks weird to me.} 
\begin{equation}
  \mu^{(i)}_j \;\ge\; \mu^{(i-1)}_j \;\ge\; \mu^{(i)}_{j+1}\qquad
(2\le i\le n,\; 1\le j\le i-1).
\end{equation}
It is convenient to display the pattern as
\[
\begin{tikzpicture}[baseline=(current bounding box.center),scale=1,
  every node/.style={inner sep=1.2pt}]
  % top row i=n
  \node (n1) at (0,0) {$\mu^{(n)}_1$};
  \node (n2) at (1.6,0) {$\mu^{(n)}_2$};
  \node (n3) at (3.2,0) {$\cdots$};
  \node (n4) at (4.8,0) {$\mu^{(n)}_{n-1}$};
  \node (n5) at (6.4,0) {$\mu^{(n)}_n$};

  % row n-1
  \node (m1) at (0.8,-1.0) {$\mu^{(n-1)}_1$};
  \node (m2) at (2.4,-1.0) {$\mu^{(n-1)}_2$};
  \node (m3) at (4.0,-1.0) {$\cdots$};
  \node (m4) at (5.6,-1.0) {$\mu^{(n-1)}_{n-1}$};

  % dots
  \node (d) at (3.2,-2.0) {$\vdots$};

  % row i=2 (new)
  \node (c1) at (2.4,-2.7) {$\mu^{(2)}_1$};
  \node (c2) at (4.0,-2.7) {$\mu^{(2)}_2$};

  % bottom row i=1
  \node (b1) at (3.2,-3.6) {$\mu^{(1)}_1$};

  % interlacing "arrows"
  \draw[->,thin] (n1) -- (m1);
  \draw[->,thin] (m1) -- (n2);
  \draw[->,thin] (n2) -- (m2);
  \draw[->,thin] (m2) -- (n3);
  \draw[->,thin] (m3) -- (n4);
  \draw[->,thin] (n4) -- (m4);
  \draw[->,thin] (m4) -- (n5);

  % arrows for the last rows
  \draw[->,thin] (c1) -- (b1);
  \draw[->,thin] (b1) -- (c2);
\end{tikzpicture}
\]
where each entry $\mu^{(i-1)}_j$ is squeezed between the two entries above it. Such arrays naturally index the {\it Gelfand-Tsetlin basis} in the irreducible $\mathfrak{gl}_n$-module with the highest weight $\mu=\mu^{(n)}$ with the action of Chevalley generators in this basis given by explicit formulas.

Following \cite{FFFR,feigin2011yangians}, we will 
%it is often more convenient to 
encode a pattern
$\{\mu^{(i)}_j\}$ by a triangular pattern of integers $\{d_{ij}\}_{1\le j\le i\le n}$
defined by
\begin{equation}
  \mu^{(i)}_j \;=\; \mu^{(n)}_j \;-\; d_{ij}, \qquad 1\le j\le i\le n,
\end{equation}
in particular $d_{nj}=0$ for all $j$.
(In other words, $d_{ij}$ measures how much the $j$-th coordinate drops when passing
from the top row to the $i$-th row.)
The interlacing inequalities translate into inequalities for the integers $d_{ij}$.
For example, $\mu^{(i)}_j\ge \mu^{(i-1)}_j$ becomes
\begin{equation}\label{eq:d-vertical}
  d_{ij}\le d_{i-1,j},
\end{equation}
while $\mu^{(i-1)}_j\ge \mu^{(i)}_{j+1}$ becomes
\begin{equation}\label{eq:d-diagonal}
  \mu^{(n)}_j-d_{i-1,j}\;\ge\;\mu^{(n)}_{j+1}-d_{i,j+1}
  \quad\Longleftrightarrow\quad
  d_{i-1,j}\;\le\; d_{i,j+1} + \bigl(\mu^{(n)}_j-\mu^{(n)}_{j+1}\bigr).
\end{equation}
Thus the classical interlacing conditions can be viewed as a system of linear inequalities
for the integers $d_{ij}$ with constants determined by the dominant top row
$\mu^{(n)}_1\ge\cdots\ge\mu^{(n)}_n$. 
We can picture these relations as follows, with the black arrows representing relations~\eqref{eq:d-vertical} and the blue arrows representing  relations~\eqref{eq:d-diagonal}:

\[
\begin{tikzpicture}[every node/.style={font=\small}]
    % left dots

    % row i=1 (keep only j>=1)
    \node (a11) at (0,0) {\(d_{1,1}\)};

    % row i=2
    \node (a21) at (0,-1) {\(d_{2,1}\)};
    \node (a22) at (1,-1) {\(d_{2,2}\)};

    % row i=3
    
    \node (a31) at (0,-2) {\(d_{3,1}\)};
    \node (a32) at (1,-2) {\(d_{3,2}\)};
    \node (a33) at (2,-2) {\(d_{3,3}\)};

    % bottom dots
    \node (a41) at (0,-3) {\(\dots\)};
    \node (a42) at (1,-3) {\(\dots\)};
    \node (a43) at (2,-3) {\(\dots\)};
    \node (a44) at (3,-3) {\(\dots\)};

    % black arrows (standard GT inequalities d_{i,j} >= d_{i+1,j})

    \draw[->] (a11) -- (a21);

    \draw[->] (a21) -- (a31);
    \draw[->] (a22) -- (a32);

    \draw[->] (a31) -- (a41);
    \draw[->] (a32) -- (a42);
    \draw[->] (a33) -- (a43);

    % blue diagonal arrows (permitted-pattern inequalities depending on lambda)
    % Example corresponding to the pair (Lambda_dom, alpha_0+alpha_1):
    \draw[blue,->] (a32) -- (a21);
      %node[midway, right, xshift=2pt, yshift=2pt]
      %{\scriptsize\(\scriptstyle \mathtt{y}_0-\mathtt{y}_2-1\)};
    \draw[blue,->] (a22) -- (a11);
      %node[midway, right, xshift=2pt, yshift=2pt]
      %{\scriptsize\(\scriptstyle \mathtt{y}_0-\mathtt{y}_2-1\)};
      \draw[blue,->] (a33) -- (a22);
\end{tikzpicture}
\]

For what follows, it is convenient to replace  inequalities~\eqref{eq:d-diagonal} with 
the strict 
% inequalities $d_{i-1,j}\;<\; d_{i,j+1} + \bigl(\mu^{(n)}_j-\mu^{(n)}_{j+1}+1\bigr)$ that are equivalent to the following 
(redundant, in general) system of inequalities for all positive coroots $\alpha^\vee_{jk}$ for $1\le j<k\le n$:
\begin{equation}\label{eq:d-diagonal-strict}
d_{i-1,j}\;<\; d_{i,k} + \bigl(\mu^{(n)}_j-\mu^{(n)}_{k}-j+k\bigr)=d_{i,k} + \langle\mu^{(n)}+\rho,\alpha_{jk}^\vee\rangle.
\end{equation}

% \subsection{Relaxing inequalities for nonintegral dominant highest weights.}

\bigskip

Several generalizations of this construction exist for infinite-dimensional representations. 
Namely, when $\mu^{(n)}$ is dominant but not necessarily integral, the constants
$\mu^{(n)}_j-\mu^{(n)}_{k}$, \(j<k\) 
% (and more generally differences between entries that are forced to be compared by interlacing) 
need not be integers (but have to be nonnegative if integers).
In the $d$-description this leads to a natural ``relaxation'':
one keeps exactly those inequalities~\eqref{eq:d-diagonal-strict} whose right-hand side shift is an integer, and drops
the inequalities involving nonintegral shifts.
% Equivalently, inequalities comparing two entries are imposed only when these entries are {\it comparable}, i.e.\ when their difference is an integer, and the inequalities between incomparable entries (those whose difference is not an integer) are removed. 
This produces a family of integer patterns $d_{ij}$ associated
with a dominant (possibly nonintegral) $\mu^{(n)}$. 
In \cite{futorny2019combinatorial} (also stated in \cite{Popov}), it is shown that, for {\it any} dominant highest weight $\mu^{(n)}$, such patterns index a basis in the corresponding simple module.
% , (this was also stated in the master's thesis of Pavel Popov \cite{Popov}). 
% \zh{Sometimes authors of the references are listed without given names, this one has both given and last name, in other cases the authors not listed.}
%In the present paper, we get this as a special case of a similar statement for the affine Lie algebra $\widehat{\mathfrak{gl}}_n$ where the patterns $\{d_{ij}\}$ generalize to {\it periodic} patterns. 

%For $\mu^{(n)}$ being the highest weight of a parabolically induced module (i.e. such $\mu^{(n)}$ that for any $j$, all the $k$ such that $\mu_k^{(n)}-\mu_j^{(n)}$ is an integer, form an interval in $\{1,2,\ldots,n\}$). Such relaxed Gelfand--Tsetlin patterns index a basis in the simple $\mathfrak{gl}_n$-module with the highest weight $\mu^{(n)}$ (that is, the parabolically induced from a simple finite-dimensional module), and the formulas for the action of the generators stay the same as in the classical case. This relaxed affineGelfand--Tsetlin combinatorics admits a direct generalization to the affine/periodic setting in the sense of \cite{FFNR}.

\subsection{Periodic (affine) generalization}

%{\lcom continue by summarizing the combinatorics of FFNR11}

A similar combinatorial picture arises in the affine setting (see  \cite{feigin2011yangians} for the details). We assume \(n>2\) unless stated otherwise.  Here one considers representations of the affine Lie algebra $\widehat{\mathfrak{gl}}_n$ and of the affine Yangian $Y(\widehat{\mathfrak{sl}}_n)$. The appropriate combinatorial objects are {\it periodic} Gelfand--Tsetlin patterns, that is, infinite arrays of nonnegative integers $d_{ij}$ for all $i\ge j$ satisfying periodicity conditions $d_{i+n,j+n}=d_{ij}$ as well as the finitistic condition $d_{ij}=0$ for large enough values of $i-j$. Such patterns may be viewed as affine analogues of the ``most relaxed'' Gelfand--Tsetlin patterns describing bases of Verma modules: 
\[
\begin{tikzpicture}!%[scale=0.8]
            \node (a00) at (-1,1) {\(\dots\)};
            
            \node (a10) at (-1,0) {\(d_{1,0}\)};        \node (a11) at (0,0) {\(d_{1,1}\)};
    
            \node (a20) at (-1,-1) {\(d_{2,0}\)};
            \node (a21) at (0,-1) {\(d_{2,1}\)};
            \node (a22) at (1,-1) {\(d_{2,2}\)};
    
            \node (a30) at (-1,-2) {\(\dots\)};
            \node (a31) at (0,-2) {\(d_{3,1}\)};
            \node (a32) at (1,-2) {\(d_{3,2}\)};
            \node (a33) at (2,-2) {\(d_{3,3}\)};
    
            \node (a41) at (0,-3) {\(\dots\)};
            \node (a42) at (1,-3) {\(\dots\)};
            \node (a43) at (2,-3) {\(\dots\)};
            \node (a44) at (3,-3) {\(\dots\)};
    
            \draw[->] (a00) -- (a10);
    
            \draw[->] (a10) -- (a20);
            \draw[->] (a11) -- (a21);
            
            \draw[->] (a20) -- (a30);
            \draw[->] (a21) -- (a31);
            \draw[->] (a22) -- (a32);
    
            \draw[->] (a31) -- (a41);
            \draw[->] (a32) -- (a42);
            \draw[->] (a33) -- (a43);
        \end{tikzpicture}
\]       

In \cite{feigin2011yangians} it was shown that these patterns index a basis in a thin representation of the affine Yangian $Y(\widehat{\mathfrak{sl}}_n)$ whose character coincides with that of the universal Verma module over $\widehat{\mathfrak{gl}}_n$. In this basis the generators of the Yangian act by explicit formulas for matrix elements which generalize the classical Gelfand--Tsetlin formulas.

Just as in the finite-dimensional case, we are interested in extending this description to smaller irreducible modules. The main problem is that in this situation the formulas for matrix elements of generators may develop poles, and the generic Gelfand--Tsetlin realization no longer directly produces an irreducible module. Our main observation is that, for dominant highest weights, the specialization can nevertheless be described combinatorially by restricting to a distinguished subset of Gelfand--Tsetlin patterns, which we call {\it permitted patterns}. We prove that the linear span of basis vectors corresponding to such patterns carries a well-defined action of the affine Yangian and is naturally identified with the irreducible highest weight $\widehat{\mathfrak{gl}}_n$-module (the case of {\it integer} dominant highest weights was discussed in \cite{feigin2011yangians}). In the present paper, we prove the following

\bigskip

\noindent{\bf Main Theorem.} (Theorem~\ref{th:Main} in the text) {\it The vector space spanned by permitted Gelfand--Tsetlin patterns carries a well-defined structure of a thin $Y(\widehat{\mathfrak{sl}}_n)$-module. This module is irreducible and naturally identified with the pullback of the irreducible highest weight $\widehat{\mathfrak{gl}}_n$-module under the evaluation homomorphism of \cite{kodera2021guay}
\(
Y(\widehat{\mathfrak{sl}}_n) \to U\widehat{\mathfrak{gl}}_n.
\)
}

\bigskip

This is the sense in which our construction gives an {\it athinization} of irreducible \(\widehat{\mathfrak{gl}}_n\)-modules with dominant highest weights. As modules over the affine Kac--Moody algebra, these representations need not be thin: their weight spaces may have nontrivial multiplicities. The larger algebra \(Y(\widehat{\mathfrak{sl}}_n)\), however, has a larger commutative Cartan-type subalgebra, and the additional commuting operators separate these multiplicities. The irreducible affine module is therefore realized as a thin Yangian module, with one-dimensional joint eigenspaces indexed by permitted affine Gelfand--Tsetlin patterns.

This construction leads to explicit combinatorial bases and character formulas for a large class of $\widehat{\mathfrak{gl}}_n$-modules with dominant highest weights. 
In particular, it provides a new realization of the {\it admissible representations} introduced by Kac and Wakimoto~\cite{kac1988modular}. 
These representations form a distinguished class of highest weight modules whose characters enjoy modular properties. 
Despite their importance, explicit combinatorial bases for admissible modules have been largely unknown. 
Our results provide such bases in the form of affine Gelfand--Tsetlin patterns, see Proposition~\ref{pr:admissible}.

For the case of integral dominant highest weight representations, the conditions that specify permitted Gelfand--Tsetlin patterns are equivalent to cylindric plane partitions introduced in \cite{gessel1997cylindric}. 
In the general case, these combinatorial conditions have appeared before, restated in terms of partitions. For the case \(n=2\), they were introduced in \cite{burge1993restricted} for combinatorial reasons. 
For general \(n\), these conditions appear in the representation theory of quantum deformed \(W_n\) algebras in \cite{feigin2013representations}.

\subsection{$q$-deformation}  
% We work in the setting of representations of the affine Lie algebra $\widehat{\mathfrak{gl}}_n$ regarded as representations of the affine Yangian $Y(\widehat{\mathfrak{sl}}_n)$, but 
All the above results can be easily repeated for \(q\)-deformed algebras as well. 
We do this in Section~5 of the present paper. 
In this setting the affine Yangian is replaced by the quantum toroidal algebra of \(\mathfrak{gl}_n\), while the affine Lie algebra \(\widehat{\mathfrak{gl}}_n\) is replaced by its quantum affine counterpart \(U_q\widehat{\mathfrak{gl}}_n\).  
Using the explicit Gelfand--Tsetlin formulas for the action of the quantum toroidal algebra, due to Tsymbaliuk \cite{tsymbaliuk2010quantum}, one obtains a \(q\)-deformed version of the thin modules considered above. 
For generic highest weights these modules are again indexed by periodic Gelfand--Tsetlin patterns, with matrix coefficients given by explicit rational \(q\)-difference analogues of the Yangian formulas.

For dominant highest weights, the same specialization procedure remains valid. Namely, the span of permitted periodic Gelfand--Tsetlin patterns is stable under the quantum toroidal action, and the resulting module is naturally identified, via the quantum toroidal evaluation homomorphism, with the pullback of the irreducible highest weight \(U_q\widehat{\mathfrak{gl}}_n\)-module. Thus the main results of the paper extend from the affine Lie algebra/affine Yangian setting to the quantum affine/toroidal algebra setting. In particular, the same combinatorics of permitted patterns gives explicit Gelfand--Tsetlin-type bases and character formulas for the corresponding \(q\)-deformed dominant modules, including the quantum analogues of the admissible representations considered earlier. 

\subsection{Geometric perspective}

The results of this paper admit a natural geometric interpretation within the framework developed in~\cite{feigin2011yangians}. In that work the generic Verma module is realized as the equivariant cohomology of affine Laumon quasiflag spaces. The Gelfand--Tsetlin basis arises from Atiyah--Bott localization to the discrete fixed-point set of the maximal torus action.

The specialization of the highest weight corresponds geometrically to reducing the equivariance, that is, passing from the full torus to an appropriate subtorus. In this situation the fixed-point set is no longer discrete but still splits into connected components.

From the geometric interpretation of the Shapovalov form as the Poincar\'e pairing, one expects that the irreducible quotient of the Verma module is spanned by the cohomology classes of compact components of the fixed-point set. The existence of a Gelfand--Tsetlin basis corresponds to the situation in which all such compact components are isolated points. The results of the present paper suggest that this phenomenon occurs precisely when the highest weight is dominant. We outline this argument in Section~\ref{sec:geometry} and plan to present more details elsewhere.  

\subsection{Possible crystal structure}

%The combinatorics of permitted Gelfand--Tsetlin patterns also suggests a natural crystal-theoretic structure. For each simple root direction one considers the set of permitted patterns obtained by varying only the corresponding row of the pattern. The defining inequalities for permitted patterns then identify this set with a tensor product of elementary \(\mathfrak{sl}_2\)-crystals, possibly infinite in one direction. Applying this construction for all simple roots one obtains a natural \(\widehat{\mathfrak{sl}}_n\)-crystal structure on the set of permitted periodic Gelfand--Tsetlin patterns.

%In the case of an integral dominant highest weight, this crystal is normal and can be identified with the product of the crystal of the corresponding integrable highest weight \(\widehat{\mathfrak{sl}}_n\)-module and the set of integer partitions. For generic highest weights, the resulting crystal degenerates to the product of the crystal \(B(\infty)\) with the same partition factor. This provides further evidence that the permitted Gelfand--Tsetlin patterns should be viewed not only as a basis of the corresponding thin module, but also as a natural nonintegral extension of the usual crystal combinatorics for highest weight representations. We discuss it in Section~\ref{ssec:crystal}.

The combinatorics of permitted Gelfand--Tsetlin patterns naturally suggests a crystal-theoretic question. In the finite-dimensional case, Gelfand--Tsetlin patterns carry the crystal structure studied by Littelmann \cite{Littelmann1998}: for each simple root one fixes all rows except the corresponding one, and the remaining degrees of freedom form a tensor product of elementary \(\mathfrak{sl}_2\)-crystals. The inequalities defining permitted periodic patterns admit a formally analogous row-wise construction in the affine setting. Thus, one obtains natural candidates for Kashiwara operators on the set of permitted patterns, and in the integrable dominant case, this construction recovers the crystals already considered by Tingley \cite{Tingley2008}.

However, for general dominant, nonintegral highest weights the situation is more subtle. The straightforward extension of the above row-wise construction to permitted patterns does not, in general, satisfy Stembridge's local axioms for simply-laced crystals. Thus, the resulting object should not be regarded as an \(\widehat{\mathfrak{sl}}_n\)-crystal in the usual sense without further modification. We therefore do not pursue the crystal-theoretic interpretation in the present paper and postpone the analysis of possible crystal structures on permitted patterns to future work.

\subsection{Organization of the paper}

In Section~\ref{sec:basic} we recall the necessary background on affine Lie algebras and affine Yangians and review the evaluation homomorphisms. 
In Section~\ref{sec:thin} we study the specialization of the Gelfand--Tsetlin formulas to dominant highest weights and introduce the notion of permitted patterns. 
We prove that the span of these patterns carries a Yangian module structure and identify it with the irreducible highest weight module. 
In Section~\ref{sec:admissible} we apply these results to admissible representations of $\widehat{\mathfrak{sl}}_n$ and derive explicit formulas for their characters.
In Section~\ref{sec:qcase}, we establish similar results in the setting of affine quantum and toroidal algebras. 
Finally, in Section~\ref{sec:geometry}, we discuss the geometric approach via affine Laumon spaces.

\iffalse 
\zh{While I would say many things differently, I like the introduction. There is one part which is not discussed  - the Verma module, integrable modules can be lifted to toroidal (or affine Yangian) case for generic parameters $q_1,q_2$. For evaluation module we will a relation between $q_1,q_2$. But for admissible modules this is not the case. For generic $q_1,q_2$ the module is larger and it has to be reduced in the the resonance. This is quiet similar to what happens when you go to roots of unity. Well, maybe this is important only to me and we can forget it.}
\fi

\subsection*{Acknowledgements}

We are grateful to Boris Feigin for useful discussions and suggestions. The work of L.R. was supported by the Fondation Courtois. The work of E.M. is partially supported by Simons Foundation grant $\sharp$ 709444. A.T. was partially funded by a excellence fellowship (“365478”) from the Fonds de recherche du Québec. M.B. is grateful to the Perimeter Institute and the Universit{\'e} de Montr{\'e}al, where part of this work was done, for their hospitality.

%write intro

\section{Preliminaries}\label{sec:basic}

  \subsection{Affine algebra \(\widehat{\mathfrak{gl}}_n\)}
    Let \( \widehat{\mathfrak{gl}}_n = \mathfrak{gl}_n \otimes \mathbb{C}[t,t^{-1}] \oplus \mathbb{C}K \) be the central extension of the loop algebra \(\mathfrak{gl}_n \otimes \mathbb{C}[t,t^{-1}]\), with commutation relations
    \begin{equation}
        [X \otimes t^r, Y \otimes t^s] = [X, Y] \otimes t^{r+s} + \delta_{r+s,0} r \, \operatorname{Tr}(XY) K,
    \end{equation}   
    where \(X,Y\in\mathfrak{gl}_n\) and \(r,s\in\mathbb{Z}\). Let \( \widetilde{\mathfrak{gl}}_n = \widehat{\mathfrak{gl}}_n \oplus \mathbb{C}D\), where
    \begin{equation}
        [D, X \otimes t^r] = r X \otimes t^r, \quad [D, K] = 0.
    \end{equation}    

    Similarly,
    \begin{equation}
        \widehat{\mathfrak{sl}}_n = \mathfrak{sl}_n \otimes \mathbb{C}[t,t^{-1}] \oplus \mathbb{C}K \subset \widehat{\mathfrak{gl}}_n,
    \end{equation}
    Let \(\mathtt{Heis}\) be the Heisenberg algebra generated by \(K', \mathtt{a}_r\) for \(r \in \mathbb{Z}\), with central \(K'\) and relations:
    \begin{equation}
        [\mathtt{a}_r, \mathtt{a}_s] = r \delta_{r+s,0} K', 
    \end{equation}
    and \(\widetilde{\mathtt{Heis}} = \mathtt{Heis}\oplus \mathbb{C}D'\) with relations
    \begin{equation}
        [D', \mathtt{a}_r] = r \mathtt{a}_r, \quad  [D', K'] = 0. 
    \end{equation}    
    \begin{equation}
        \widetilde{\mathfrak{sl}}_n = \widehat{\mathfrak{sl}}_n \oplus \mathbb{C}D \subset \widetilde{\mathfrak{gl}}_n.
    \end{equation}    
    %\zh{Is $d$ same is $D$?}
    We fix  an isomorphism
    \begin{equation}\label{eq:glnthroughsln}
        \widetilde{\mathfrak{gl}}_n \cong (\widetilde{\mathtt{Heis}}\oplus\widetilde{\mathfrak{sl}}_n)/(K' - n K, D'-D).
    \end{equation}
    We use the triangular decompositions \(\mathfrak{gl}_n = \mathfrak{n}_{-} \oplus 
    \mathfrak{h}_{\mathfrak{gl}}\oplus \mathfrak{n}_{+}\) and \(\mathfrak{sl}_n = \mathfrak{n}_{-} \oplus \mathfrak{h} \oplus \mathfrak{n}_{+}\), where \(\mathfrak{h}\) is the space of traceless diagonal matrices.
    Let \( \widehat{\mathfrak{h}} = \mathfrak{h} \oplus \mathbb{C}K \), \( \widetilde{\mathfrak{h}} = \widehat{\mathfrak{h}} \oplus \mathbb{C}D \), and decompose \(        \widetilde{\mathfrak{sl}}_n = \widehat{\mathfrak{n}}_{-} \oplus \widetilde{\mathfrak{h}} \oplus \widehat{\mathfrak{n}}_{+}\),
    where \(\widehat{\mathfrak{n}}_{\pm} = \mathfrak{n}_{\pm} \oplus \mathfrak{sl}_n \otimes t^{\pm1}\mathbb{C}[t^{\pm 1}]\), and \(\widetilde{\mathfrak{b}} = \widehat{\mathfrak{n}}_+ \oplus \widetilde{\mathfrak{h}}\subset \widetilde{\mathfrak{sl}}_n\). 
    
    The \textit{Cartan matrix} \(C=(C_{ij})_{i,j\in\mathbb{Z}/n\mathbb{Z}} \) has entries
    \begin{equation}
        C_{ij}  = 2\delta_{ij}^n - \delta_{i,j+1}^n - \delta_{i,j-1}^n,
    \end{equation}
    where
    \begin{equation}
        \delta_{i,j}^n = \begin{cases}1, & i \equiv j \mod n \\ 0, & \text{otherwise} \end{cases}.
    \end{equation}
        \begin{theorem}[\cite{kac:1990}] The Lie algebra
    \(\widehat{\mathfrak{sl}}_n\) is isomorphic to the Lie algebra generated by \(e_i, h_i, f_i\), \(i \in \mathbb{Z}/n\mathbb{Z}\), with relations
    \begin{equation}
        [h_i,h_j] = 0, \quad [h_i,e_j] = C_{ij}e_j, \quad [h_i,f_j] = -C_{ij}f_j, \quad [e_i,f_j] = \delta_{ij}h_i,
    \end{equation}
    \begin{equation}
        \mathrm{ad}_{e_i}^{1 - C_{ij}}(e_j) = \mathrm{ad}_{f_i}^{1 - C_{ij}}(f_j) = 0 \quad \text{for } i \not= j.\quad \qed      
    \end{equation}
    \end{theorem}
    \subsection{Root systems}
    Define a  scalar product $(\ ,\ )$ on \(\mathfrak{h}_{\mathfrak{gl}}\) by \ \((X,Y) = \operatorname{Tr}(XY)\). Extend it to \(\widetilde{\mathfrak{h}}\) by
    \begin{equation}
        (K, \mathfrak{h}) = 0, \quad (D, \mathfrak{h}) = 0, \quad (K, D) = 1, \quad (K,K)=(D,D)=0.
    \end{equation}    
         Let \(\{\epsilon_1, \dots, \epsilon_n\}\) be the  basis in \(\mathfrak{h}_{\mathfrak{gl}}^*\) dual to the basis of the matrix units \(E_{ii}\) and let \(\delta = (K,\cdot)\), \(\omega_0 = (D,\cdot) \) be elements of \(\widetilde{\mathfrak{h}}^*\). Clearly \(\delta(D) = \omega_0(K) = 1\) and \(\delta(K) = \omega_0(D) = 0\). 

    %\zh{Maybe first extend the scalar product, and then describe a dual basis and then define the scalar product on the dual space?}
    
    Let \(\alpha_i = \epsilon_i - \epsilon_{i+1}\) for \(i \in \{1,\dots,n-1\}\), \(\alpha_0 = \delta +\epsilon_n -\epsilon_{1} \) . Extend \(\epsilon_i\) and \(\alpha_i\) to all \(i \in \mathbb{Z}\) by periodicity: \(\epsilon_i =\epsilon_{i+n}\), \(\alpha_i = \alpha_{i+n}\). %Note that 
    %\begin{equation}\label{eq:DefDelta}
    %    \delta = \sum_{i=0}^{n-1} \alpha_i.
    %\end{equation}
    Let  \( \Delta = \{\alpha_i\}_{i=1}^{n-1}\)
        and  \(\widehat{\Delta} = \{\alpha_0\}\cup\Delta\) be the \textit{simple roots} of types  \(A_{n-1}\) and \(\widehat{A}_{n-1}\) correspondingly, forming bases in \(\mathfrak{h}^*\) and \(\widehat{\mathfrak{h}}^*\). Let
    \begin{equation}
        \Phi = \{\epsilon_i - \epsilon_j \mid 1\le i \neq j\le n\} \subset \mathfrak{h}^*, \quad
        \Phi_+ = \{\epsilon_i - \epsilon_j  \mid 1\le i<j\le n\}.       
    \end{equation}
    The \textit{affine root system} \(\widehat{\Phi}\) of type \(\widehat{A}_{n-1}\) is
    \begin{equation}
        \{\alpha + m\delta \mid \alpha \in \Phi,\; m \in \mathbb{Z}\} \cup \{m\delta \mid m \neq 0\},        
    \end{equation}
    with \textit{positive roots}
    \begin{equation}
        \widehat{\Phi}_+ = \{\alpha + m\delta \mid \alpha \in \Phi,\; m > 0\} \cup \Phi^+ \cup \{m\delta \mid m > 0\}
    \end{equation}
    and \textit{real roots}
    \begin{equation}
        \widehat{\Phi}^{\text{re}} = \{\alpha + m\delta \mid \alpha \in \Phi,\; m \in\mathbb{Z}\}.
    \end{equation}    

    \begin{proposition}\label{th:RepOfRoot}
     For any element \(\alpha\in\widehat{\Phi}_+\) there exists a pair of integers \(l<r\) such that
     \begin{equation}
         \alpha = \alpha_{l,r},
     \end{equation}
     where \(\alpha_{l,r}=\sum_{s=l}^{r-1}\alpha_s\).
     For \(\alpha\in\widehat{\Phi}_+^{\text{re}}\)  the pair \((l,r)\) is unique up to the shift \((l+ t n, r+ tn)\) for \(t\in\mathbb{Z}\).
    \end{proposition}
    %Extend \(\epsilon_i\) for all \(i\in \mathbb{Z}\) by peroidicity \(\epsilon_{i+n}=\epsilon_i\). \zh{Can we move it to the place where $\alpha_i$ were extended? Why do we have spelling issues? Need to run a spelling check!} 
    Note that 
    \begin{equation}\label{eq:gammag-delta}
        \alpha_{l,r} = \epsilon_{l}-\epsilon_{r}+\left(\Big\lceil \frac{r}{n} \Big\rceil -\Big\lceil \frac{l}{n} \Big\rceil\right)\delta, \quad \text{for}\quad l<r.
    \end{equation} 
    \begin{proof}
        Any \(\alpha\in\widehat{\Phi}_+\) is represented as \(\alpha = \epsilon_i - \epsilon_j + m\delta\) where either \(1\le i<j\le n\) and \(m\ge 0\)  or  \(1\le j\le i\le n\) and \(m \ge 1\). In both cases \(i< mn+j\) and we have \( \alpha = \alpha_i + \alpha_{i+1} + \dots \alpha_{mn+j-1}\).
    \end{proof}    
    Let 
    \begin{equation}
        \widehat{\Phi}_+^{\text{re}} =  \widehat{\Phi}^{\text{re}}\cap \widehat{\Phi}_{+}.
    \end{equation}
    Note that
        \begin{equation}
            \widehat{\Phi}_+^{\text{re}}  = \{\alpha_{l,r}|\ l<r\in\mathbb{Z},\quad l-r\not\equiv 0 \quad \mod n\}.
        \end{equation}

    Let \(\{\omega_0,\omega_1,\dots,\omega_{n-1},\delta\} \subset \widetilde{\mathfrak{h}}^*\) be the basis dual to \(\{\alpha_0,\alpha_1,\dots,\alpha_{n-1},\omega_0\}\) with respect to the scalar product \((\;,\;)\). The \textit{fundamental weights} \(\omega_i\) form a basis of \(\omega_0^{\perp} \subset \widetilde{\mathfrak{h}}^*\). %\zh{\sout{Extend \(\alpha_i, \omega_i\) to all \(i \in \mathbb{Z}\) by periodicity \( \omega_i = \omega_{i+n}\). }}
    
    Let  %\zh{Why not to call them the root and weight lattices?}
    \begin{align}
        Q &= \mathrm{Span}_\mathbb{Z}\{\alpha_1,\dots,\alpha_{n-1}\}, \\
        P &= \mathrm{Span}_\mathbb{Z}\{\omega_1 - \omega_0,\dots,\omega_{n-1} - \omega_0\}
    \end{align}
     be the \textit{root} and \textit{weight lattices} correspondingly. 
    %Similarly \zh{Why similarly? Why it is in this place?}
    % \begin{equation}
    %     \widehat{Q} = \mathrm{Span}_{\mathbb{Z}}\{\alpha_0,\alpha_1,\dots,\alpha_{n-1}\}
    % \end{equation}
    % is the \textit{affine root lattice}.
    We have \(Q \subset P\). The lattice \(P\) acts on \(\widetilde{\mathfrak{h}}^*\) by the following operators
    \begin{equation}\label{eq:Paction}    
        \begin{aligned}
            t_\Lambda(v) &= v - (v,\Lambda)\,\delta, ~~\text{for}~ v\in\widehat{\mathfrak{h}}^*,\\
            t_\Lambda(\omega_0) &= \omega_0 + \Lambda - \tfrac{(\Lambda,\Lambda)}{2}\,\delta,
        \end{aligned}
    \end{equation}
    for all \(\Lambda\in P\). %\zh{Can we either use commas (prefered) or semicolumns but not both?}

    For $\alpha\in\widehat{\Delta}$, let \(s_\alpha(v) = v - 2\tfrac{(v,\alpha)}{(\alpha,\alpha)}\alpha\) be the corresponding reflection in \(\widetilde{\mathfrak{h}}^*\). Let \(W\) be the \textit{finite Weyl group} generated by \(s_\alpha\), \(\alpha \in \Delta\). Remark that  \(W\simeq \mathfrak{S}_n\) is a group of permutations on \(n\) elements. Similarly, let \(\widehat{W}\) be the \textit{affine Weyl group} generated by \(s_\alpha\), \(\alpha \in \widehat{\Delta}\). The affine Weyl group is known to be the semidirect product \(\widehat{W} = W \ltimes Q\) where \(Q\) acts on $W$ by formulas \eqref{eq:Paction}. Extending $Q$ to \(P\), we get the \textit{extended affine Weyl group} \(\widehat{W}^e = W \ltimes P\).   
    \subsection{Highest weight representations}
%\zh{I would write simply $i\neq j$. Otherwise you have to write $\delta^n_{ij}$}
    
    Let \(\Lambda \in \omega_0^{\perp}\). The \textit{Verma module} \(\mathcal{M}_\Lambda\) \textit{over} \(\widetilde{\mathfrak{sl}}_n\) is
    \begin{equation}
        \mathcal{M}_\Lambda = \mathrm{Ind}_{\widetilde{\mathfrak{b}}}^{\widetilde{\mathfrak{sl}}_n} \mathbb{C}\zeta_\Lambda,
    \end{equation}
    where \(\mathbb{C}\zeta_\Lambda\) is a 1-dimensional module with
    \begin{equation}
        \widehat{\mathfrak{n}}_+ \zeta_\Lambda = 0, \quad h\,\zeta_\Lambda = \Lambda(h)\,\zeta_\Lambda, \quad \forall h \in \widetilde{\mathfrak{h}}.
    \end{equation}
    
    \begin{theorem}[\cite{kac:1990}]
    There exists a unique symmetric bilinear form (Shapovalov form) on \(\mathcal{M}_\Lambda\) such that
    \begin{equation}
        (\zeta_\Lambda, \zeta_\Lambda) = 1, \quad (e_i \zeta, \xi) = (\zeta, f_i \xi), \quad (h_i \zeta, \xi) = (\zeta, h_i \xi),\quad \text{for all} \quad \zeta, \xi\in \mathcal{M}_\Lambda.        
    \end{equation}
        \qed
    \end{theorem}
    
    \begin{theorem}[\cite{kac:1990}]
    The Verma module $\mathcal{M}_\Lambda$ has a unique irreducible quotient $\mathcal{L}_\Lambda$. The module $\mathcal{L}_\Lambda$ is the quotient of $\mathcal{M}_\Lambda$ by the kernel of the Shapovalov form. \qed
    \end{theorem}
%\zh{I think the wording is ugly. I would say: The Verma module $\mathcal{M}_\Lambda$ has a unique irreducible quotient $\mathcal{L}_\Lambda$. The module $\mathcal{L}_\Lambda$ is the quotient of $\mathcal{M}_\Lambda$ by the kernel of the Shapovalov form. }

    \begin{theorem}[\cite{kac:1990}]
    The Verma module \(\mathcal{M}_\Lambda\) is irreducible if and only if
    \begin{align}
        (\Lambda + \rho,\alpha) &\notin \mathbb{Z}_{>0}, \quad \text{for all } \alpha \in \widehat{\Phi}^{\mathrm{re}}_+;\\       
        (\Lambda + \rho, \delta) &\neq 0;
    \end{align}
    where \(\rho = \sum_{i=0}^{n-1} \omega_i\).
    \qed
    \end{theorem}

    We say that  \(\Lambda \in \omega_0^{\perp}\) is of \textit{level} \(k = (\delta, \Lambda)\), where \((\ ,\ )\) is the scalar product in \(\widetilde{\mathfrak{h}}^*\). Set \(\kappa = k + n\). A weight \(\Lambda\in \omega_0^\perp\) is called \textit{generic} if \( (\Lambda + \rho,\alpha)\notin \mathbb{Z}\) for all \(\alpha\in \widehat{\Phi}_+%^{\text{re}}
    \).
    In particular, \( \kappa= (\Lambda + \rho,\delta)\not\in \mathbb{Q}\). 

    Next, we consider \(\widetilde{\mathfrak{gl}}_n\)-modules. Let \(\Lambda\) be a weight for \(\widetilde{\mathfrak{sl}}_n\)  of level \(k\). Let \(\mathcal{F}^k_{u}\) be a Fock module of \(\widetilde{\mathtt{Heis}}\) freely generated by \(\mathtt{a}_r\), \(r<0\), acting on a highest weight vector \(\zeta_u^k\) satisfying
    \begin{equation}
        K'\zeta_u^k = n k \zeta_u^k, \quad D'\zeta_u^k = 0,\quad \mathtt{a}_0 \zeta_u^k = u \zeta_u^k, \quad \mathtt{a}_r \zeta_u^k = 0,~ \text{for}~ r>0.
    \end{equation}
    Define \(\mathcal{M}_{\Lambda,u}=\mathcal{M}_{\Lambda}\otimes\mathcal{F}_u^k\) and \(\mathcal{L}_{\Lambda,u}=\mathcal{L}_{\Lambda}\otimes\mathcal{F}_u^k\) as \(\widetilde{\mathfrak{gl}}_n\)-modules with  the action induced from \eqref{eq:glnthroughsln}.
    \begin{remark}
            There is an automorphism \(F_{\widetilde{\mathfrak{gl}}_n}\) of \(\widetilde{\mathfrak{gl}}_n\) defined by  
            \begin{equation}\label{eq:Fgl}
                F_{\widetilde{\mathfrak{gl}}_n}(X\otimes t^r)= -X^t\otimes t^{-r},\quad\quad F_{\widetilde{\mathfrak{gl}}_n}(K)=-K,\quad\quad F_{\widetilde{\mathfrak{gl}}_n}(D)= - D .
            \end{equation}          
            Restricting \(F_{\widetilde{\mathfrak{gl}}_n}\) to \(\widetilde{\mathfrak{sl}}_n\) we obtain an automorphism 
            \(F_{\widetilde{\mathfrak{sl}}_n}\) of $\widetilde{\mathfrak{sl}}_n$ such that  
            \begin{equation}\label{eq:Fsl}
                F_{\widetilde{\mathfrak{sl}}_n}(f_i)= -e_{i},\quad\quad F_{\widetilde{\mathfrak{sl}}_n}(e_i)=-f_{i}, \quad\quad F_{\widetilde{\mathfrak{sl}}_n}(h_i )= - h_{i},\quad\quad F_{\widetilde{\mathfrak{sl}}_n}(K)=-K.
            \end{equation}
            Since these automorphisms interchange \(\widehat{\mathfrak{n}}_{\pm}\) they allow us to go from lowest weight modules to highest weight modules and back.
    \end{remark}

    \subsection{Affine Yangian}

   We follow the convenitions for the affine Yangian in \cite{feigin2011yangians}.
        \begin{define}
        For \(n>2\) and \(\hbar,\hbar'\in\mathbb{C}\) we define \(\widehat{Y}_\circ (\hbar,\hbar')\)  as an associative algebra with unit with generators \(\mathbf{h}_{i,r},\mathbf{x}_{i,r}^{\pm}\) for \(i\in\{1,\dots,n\}\) and \(r\in\mathbb{Z}_{\ge 0}\) and relations
        \begin{align}
               [\mathbf{h}_{i,0},\mathbf{x}^{\pm}_{j,s}] &= \pm C_{ij}\mathbf{x}^{\pm}_{j,s}
               ,  \\
                [\mathbf{x}^{+}_{i,r},\mathbf{x}^{-}_{j,s}] &=\delta_{ij}^n \mathbf{h}_{i,r+s},\\\label{eq:Y0-quadratic}
                [\mathbf{h}_{i,r}, \mathbf{h}_{j,s}] &=0,
            \end{align}
        for \( i,j\in\{1,\dots,n\}, ~ r,s \in\mathbb{Z}_{\ge0}\),
        \begin{align}
                2[\mathbf{h}_{i,r+1},\mathbf{x}^{\pm}_{j,s}]-2[\mathbf{h}_{i,r},\mathbf{x}^{\pm}_{j,s+1}]&=\pm\hbar C_{ij}(\mathbf{h}_{i,r}\mathbf{x}^{\pm}_{j,s}+\mathbf{x}^{\pm}_{j,s}\mathbf{h}_{i,r}),\\ \label{eq:Ypm-quadratic}
                2[\mathbf{x}^{\pm}_{i,r+1},\mathbf{x}^{\pm}_{j,s}]-2[\mathbf{x}^{\pm}_{i,r},\mathbf{x}^{\pm}_{j,s+1}]&=\pm\hbar C_{ij}(\mathbf{x}^{\pm}_{i,r}\mathbf{x}^{\pm}_{j,s}+\mathbf{x}^{\pm}_{j,s}\mathbf{x}^{\pm}_{i,r}),
        \end{align}
        for \((i,j)\in\{1,\dots,n\}^2\setminus\{(1,n),(n,1)\}\),
        and
        \begin{align}
            2['\mathbf{h}_{n,r+1},\mathbf{x}^{\pm}_{1,s}]-2['\mathbf{h}_{n,r},\mathbf{x}^{\pm}_{1,s+1}] &= \mp\hbar C_{ij}('\mathbf{h}_{n,r}\mathbf{x}^{\pm}_{1,s}+\mathbf{x}^{\pm}_{1,s} ~'\mathbf{h}_{n,r}),\\
                2[\mathbf{h}_{1,r+1},~'\mathbf{x}^{\pm}_{n,s}]-2[\mathbf{h}_{1,r},~'\mathbf{x}^{\pm}_{n,s+1}]&=\mp\hbar C_{ij}(\mathbf{h}_{1,r}~'\mathbf{x}^{\pm}_{n,s}+~'\mathbf{x}^{\pm}_{n,s}\mathbf{h}_{1,r}),\\
                2['\mathbf{x}^{\pm}_{n,r+1},\mathbf{x}^{\pm}_{1,s}]-2['\mathbf{x}^{\pm}_{n,r},\mathbf{x}^{\pm}_{1,s+1}]&=\mp\hbar C_{ij}('\mathbf{x}^{\pm}_{n,r}\mathbf{x}^{\pm}_{1,s}+\mathbf{x}^{\pm}_{1,s}~'\mathbf{x}^{\pm}_{n,r}),
        \end{align}
        where to define \('\mathbf{h}_{n,s}\) and \('\mathbf{x}^{\pm}_{n,s}\) we use generating series
        \begin{equation}
            \begin{aligned}
                            \mathbf{h}_n(v) &= 1+\sum_{s=0}^\infty \mathbf{h}_{n,s}\hbar^{-s}v^{-s-1},\\
                            '\mathbf{h}_n(v) &= 1+\sum_{s=0}^\infty\!~'\mathbf{h}_{n,s}\hbar^{-s}v^{-s-1}:=  \mathbf{h}_n\left(v- \frac{\hbar'}{\hbar}-\frac{n}{2}\right),  
            \end{aligned}
        \end{equation}
        and
        \begin{equation}
            \begin{aligned}
                            \mathbf{x}^\pm_n(v) &= \sum_{s=0}^\infty \mathbf{x}^\pm_{n,s}\hbar^{-s}v^{-s-1},\\
                            '\mathbf{x}^\pm_n(v)  &= \sum_{s=0}^\infty\!~'\mathbf{x}^\pm_{n,s}\hbar^{-s}v^{-s-1}:=  \mathbf{x}^\pm_n\left(v- \frac{\hbar'}{\hbar}-\frac{n}{2}\right).  
            \end{aligned}
        \end{equation}
        \end{define}
        \begin{remark}
            We extend \(\mathbf{x}^\pm_i(v),\mathbf{h}_i(v)\) to all \(i\in\mathbb{Z}\) by %formulas
            \begin{equation}
                \mathbf{x}^\pm_{i}(v) =\mathbf{x}^\pm_{i-n}(v+\frac{\hbar'}{\hbar}+\frac{n}{2}), \quad \mathbf{h}_{i}(v) =\mathbf{h}_{i-n}(v+\frac{\hbar'}{\hbar}+\frac{n}{2}).
            \end{equation}
            Then 
            \begin{equation}
                '\mathbf{x}^\pm_n(v) = \mathbf{x}^\pm_0(v), \quad '\mathbf{h}_n(v) = \mathbf{h}_0(v).
            \end{equation}
        \end{remark}

        \begin{define}
            Affine Yangian is an algebra \(\widehat{Y}(\hbar,\hbar')= \widehat{Y}_\circ(\hbar,\hbar')/I_{\text{Serre}}\) where \(I_{\text{Serre}}\) is a two-sided ideal in \(\widehat{Y}_\circ(\hbar,\hbar')\) generated by relations
        \begin{align}\label{eq:Serrepm} 
            [[\mathbf{x}^{\pm}_{i,r},[\mathbf{x}^{\pm}_{i,p},\mathbf{x}^{\pm}_{j,s}]] + [[\mathbf{x}^{\pm}_{i,p},\mathbf{x}^{\pm}_{i,r}],\mathbf{x}^{\pm}_{j,s}] &= 0
        \end{align}
        for \( i,j\in\{1,\dots,n\}, i=j\pm1 \operatorname{mod} n, ~ r,s,p \in\mathbb{Z}_{\ge0}\).
        \end{define} 

        We define 
        \(\widetilde{Y}(\hbar,\hbar') =\widehat{Y}(\hbar,\hbar')\oplus \mathbb{C}D, ~~ \widetilde{Y}_\circ(\hbar,\hbar')=\widehat{Y}_\circ(\hbar,\hbar')\oplus \mathbb{C}D\) with relations
        \begin{equation}
            [D, \mathbf{x}^{\pm}_{i,r}] =  \pm \mathbf{x}^{\pm}_{i,r}, \quad [D, \mathbf{h}_{i,r}] = 0.
        \end{equation}

    Note that for \(\hbar\neq 0\), algebras \(\widehat{Y}(\hbar,\hbar'), \widetilde{Y}(\hbar,\hbar'), \widehat{Y}_\circ(\hbar,\hbar'), \widetilde{Y}_\circ(\hbar,\hbar'), \) depend only on \(\frac{\hbar'}{\hbar}\) (up to isomorphism).
    
    \begin{remark}
        We will use results of \cite{kodera2021guay}, we compair %\zh{spelling is still off} 
        notations. There is an isomorphism
        \begin{equation}
            \phi:\widehat{Y}(\hbar,\hbar')\rightarrow Y_{\epsilon_1,\epsilon_2}(\widehat{\mathfrak{sl}}_n)
        \end{equation}
        where \(Y_{\epsilon_1,\epsilon_2}(\widehat{\mathfrak{sl}}_n)\) is from  \cite{kodera2021guay}  with
        \begin{equation}\label{eq:e&h}
            \epsilon_1+\epsilon_2 = \hbar , \quad  \quad \quad \frac{\epsilon_1 - \epsilon_2}{2} =-\frac{\hbar'}{n} - \frac{\hbar}{2},
        \end{equation}
        defined by %formulas 
        \begin{equation}
            H_{j}(v ) = \phi\left(\mathbf{h}_{j}\left(v +j\frac{\epsilon_1 - \epsilon_2}{2}
            \right)
            \right), \quad X_{j}^{\pm}(v ) = \phi\left(\mathbf{x}_{j}^{\pm}\left(v +j\frac{\epsilon_1 - \epsilon_2}{2}\right)\right), \quad j\in\{1,\dots,n\}.
        \end{equation}
    \end{remark}

    Denote by \(\widehat{Y}^{\pm}(\hbar,\hbar')\) and \(\widetilde{Y}^{0}(\hbar,\hbar')\) the algebras generated by \(\mathbf{x}_{j,r}^{\pm}\) and \(\mathbf{h}_{j,r},D\)  with relations \eqref{eq:Ypm-quadratic}, \eqref{eq:Serrepm} and \eqref{eq:Y0-quadratic} correspondingly. 
    \begin{theorem}[{\cite[Thm. 4.1]{yang2020pbw}}]
    Algebras \(\widehat{Y}^{\pm}(\hbar,\hbar'),
            \widetilde{Y}^{0}(\hbar,\hbar')\) are embedded into \(\widetilde{Y}(\hbar,\hbar')\) and the multiplication map induces vector space isomorphism
        \begin{equation}
            \widetilde{Y}(\hbar,\hbar')
            \simeq
            \widehat{Y}^{-}(\hbar,\hbar')
            \otimes
            \widetilde{Y}^{0}(\hbar,\hbar')
            \otimes
            \widehat{Y}^{+}(\hbar,\hbar').          
        \end{equation}
        \qed
    \end{theorem}
    We denote by \(\widetilde{Y}^{0,+}(\hbar,\hbar')\) and \( \widetilde{Y}^{0,+}_\circ(\hbar,\hbar')\) the subalgebras generated by \(\mathbf{x}_{j,r}^{+},\mathbf{h}_{j,r},D\) of \(\widetilde{Y}(\hbar,\hbar')\)   and \(\widetilde{Y}_\circ(\hbar,\hbar')\) correspondingly.
    %Similarly, let \(\widehat{Y}^{0,+}(\hbar,\hbar')\) subalgebra of \(\widehat{Y}(\hbar,\hbar')\) generated by \(\mathbf{x}_{j,r}^{+},\mathbf{h}_{j,r}\). 
    \begin{corollary}
         The multiplication map induces a vector space isomorphism
        \begin{equation}
            \widetilde{Y}^{0,+}(\hbar,\hbar')
            \simeq
            \widetilde{Y}^{0}(\hbar,\hbar')
            \otimes
            \widehat{Y}^{+}(\hbar,\hbar').          
        \end{equation}
    \end{corollary}    
    For a tuple \(\vec{\mathbf{Q}}(v) =\{\mathbf{Q}_i(v)\}_{i=1}^n\subset\mathbb{C}[[v]]^{\times n}\) there is a well-defined \(1\)-dimensional module \(\mathbb{C}\zeta_{\vec{\mathbf{Q}}}\) over \(\widetilde{Y}^{0,+}(\hbar,\hbar')\)  such that
    \begin{equation}
        \mathbf{x}_i^{+}(v)\zeta_{\vec{\mathbf{Q}}} = 0,\quad \mathbf{h}_i(v)\zeta_{\vec{\mathbf{Q}}} = \mathbf{Q}_i(v)\zeta_{\vec{\mathbf{Q}}}, \quad D \zeta_{\vec{\mathbf{Q}}} = 0,\quad \text{for} \quad i\in\{1,\dots,n\}.
    \end{equation}
    We have a surjective map 
    \begin{equation}
        \widetilde{Y}^{0,+}_\circ(\hbar,\hbar')\rightarrow \widetilde{Y}^{0,+}(\hbar,\hbar'),
    \end{equation} 
    so \(\mathbb{C}\zeta_{\vec{\mathbf{Q}}}\)  acquires a \(\widetilde{Y}^{0,+}_\circ(\hbar,\hbar')\)-module structure. A Verma module over \(\widetilde{Y}_\circ(\hbar,\hbar')\) is defined by formula 
    \begin{equation}
        \mathtt{M}_{\vec{\mathbf{Q}}}^\circ = \operatorname{Ind}_{\widetilde{Y}^{0,+}_\circ(\hbar,\hbar')}^{\widetilde{Y}_\circ(\hbar,\hbar')} \mathbb{C}\zeta_{\vec{\mathbf{Q}}}.
    \end{equation} 
    The Verma module over \(\widetilde{Y}(\hbar,\hbar')\) is defined by a similar formula
    \begin{equation}
        \mathtt{M}_{\vec{\mathbf{Q}}} = \operatorname{Ind}_{\widetilde{Y}^{0,+}(\hbar,\hbar')}^{\widetilde{Y}(\hbar,\hbar')} \mathbb{C}\zeta_{\vec{\mathbf{Q}}}.
    \end{equation}
    \begin{proposition}\label{th:IrrModYangian}
        There exists a unique irreducible quotient \(\mathtt{L}_{\vec{\mathbf{Q}}}\) of \(\widetilde{Y}(\hbar,\hbar')\)-module \(\mathtt{M}_{\vec{\mathbf{Q}}}\) and a unique irreducible quotient of \(\mathtt{L}^\circ_{\vec{\mathbf{Q}}}\) of \(\widetilde{Y}(\hbar,\hbar')\)-module \(\mathtt{M}^\circ_{\vec{\mathbf{Q}}}\).    
    \end{proposition}
    \begin{proof}
        Denote \( \mathcal{N} = \sum N\) where the sum is over all proper submodules of  \(\mathtt{M}^\circ_{\vec{\mathbf{Q}}}\). Note that \(\mathtt{M}^\circ_{\vec{\mathbf{Q}}}\) is graded by \(D\). So, any submodule \(N\) in \(\mathtt{M}^\circ_{\vec{\mathbf{Q}}}\) is also graded. Note that  \(\mathtt{M}^\circ_{\vec{\mathbf{Q}}}\) is generated by \(\zeta_{\vec{\mathbf{Q}}}\). Moreover, the  eigenspace of \(D\) with eigenvalue \(0\) is \(\mathbb{C}\zeta_{\vec{\mathbf{Q}}}\).  Thus \(\mathcal{N}\) does not contain \(\mathbb{C}\zeta_{\vec{\mathbf{Q}}}\). So, \(\mathtt{M}^\circ_{\vec{\mathbf{Q}}}/\mathcal{N}\) is a unique irreducible quotient of \(\mathtt{M}^\circ_{\vec{\mathbf{Q}}}\). The same argument works for \(\mathtt{M}_{\vec{\mathbf{Q}}}\).
    \end{proof}
   % \zh{I would make it a lemma or proposition at most. It is well-known for sure. The proof is very standard too.}
    \begin{corollary}\label{th:irrquad=irrwithserre}
        We have \(\mathtt{L}^\circ_{\vec{\mathbf{Q}}} \simeq \mathtt{L}_{\vec{\mathbf{Q}}}\), in particular,  relations \eqref{eq:Serrepm} are satisfied in $\mathtt{L}^\circ_{\vec{\mathbf{Q}}}$.
    \end{corollary}
    
\subsection{Evaluation map}
    \begin{theorem}[\cite{kodera2021guay}]
    For \(\kappa \notin \{0,n\}\), there is a surjective evaluation homomorphism
    \begin{equation}
        \operatorname{ev}_{-}:\widetilde{Y}(1,-\kappa)\rightarrow (U\widetilde{\mathfrak{gl}}_n/(K-\kappa+n))_{\text{comp}}.
    \end{equation}
    % and
    % \begin{equation}
    %     \operatorname{ev}_{+}:\widetilde{Y}(1,k)\rightarrow \mathcal{U}(\widetilde{\mathfrak{gl}}_n/(K-k))_{\text{comp}},
    % \end{equation}
    Here, \((U\widetilde{\mathfrak{gl}}_n)_{\text{comp}}\) is a suitable completion of \(U\widetilde{\mathfrak{gl}}_n\) whose action is well-defined in all highest weight modules over \(\widetilde{\mathfrak{gl}}_n\).
    \qed
    \end{theorem}
    \begin{remark}\label{th:FY}
        In \cite{kodera2021guay} \(\operatorname{ev}_{-}\) defined in the completion compatible with  lowest weight modules. We  modify the evaluation map of  \cite{kodera2021guay} by formula \(F_{\widetilde{\mathfrak{gl}}_n}\circ \operatorname{ev}_{-}\circ F_{Y}\) where \(F_{\widetilde{\mathfrak{gl}}_n}\) is defined by formula \eqref{eq:Fgl} and \(F_Y\) is an automorphism  of \(\widetilde{Y}(\hbar,\hbar')\) switching \(\widehat{Y}^{\pm}(\hbar,\hbar')\)
                \begin{equation}
                    F_Y(\mathbf{h}_{i}(v))=\mathbf{h}_{i}(-v),\quad \quad  F_Y(\mathbf{x}_{i}^{\pm}(v))=\mathbf{x}_{i}^{\mp}(-v ),  \quad F_Y(D) = -D.
                \end{equation}
%                In terms of 
%                \(Y_{\epsilon_1,\epsilon_2}(\widehat{\mathfrak{sl}}_n)\) we have 
%                \begin{equation}
%                    F_Y(H_{i}(v)) = H_{i}(-v - i(\epsilon_1 - \epsilon_2)),\quad\quad F_Y(X_{i}^{\pm}(v)) = X_{i}^{\mp}(-v - i(\epsilon_1 - \epsilon_2)).
%                \end{equation}
%                In particular
%                \begin{equation}
%                    F_Y(H_{i,0}) = - H_{i,0},\quad\quad F_Y(H_{i,1}) = H_{i,1}+ i(\epsilon_1 - \epsilon_2) H_{i,0}.
%                \end{equation}
%                \begin{equation}
%                    F_Y(X^{\pm}_{i,0}) = - X^{\mp}_{i,0},\quad\quad F_Y(X^{\pm}_{i,1}) = X^{\mp}_{i,1}+ i(\epsilon_1 - \epsilon_2) X^{\mp}_{i,0}.
%                \end{equation}
        \end{remark}
    \begin{theorem}[{\cite[Thm. 4.18]{kodera2019braid}}]\label{th:heisenberg}
     %\zh{ Is $k=\kappa$?}s{no, \(\kappa = k + n\)} \zh{Why Theorem 2.15 has $\kappa\neq 0$ and this one $k\neq 0$? I would suggest to write for $\kappa\neq n$}
    For \(\kappa \notin \{0,n\}\), 
    the image of \(\operatorname{ev}_{-}\widehat{Y}(1,-\kappa)\) contains \(U\widehat{\mathfrak{gl}}_n\).
            \qed
    \end{theorem}
    \begin{proposition}
        We have \(\operatorname{ev}_{-}^{*}\mathcal{L}_{\Lambda,u}\simeq \mathtt{L}_{\vec{\mathbf{Q}}}\) for 
        \begin{equation}\label{eq:highestweight}
            \mathbf{Q}_i(v) = \frac{- v + \mu_{i+1}-\frac{i}{2}-1}{- v + \mu_{i}-\frac{i}{2}-1}, \quad \text{for} \quad i\in\{1,\dots, n\},
        \end{equation}
        where \(\mu_{i}\) are determined by 
        \begin{equation}\label{eq:lambdai}
        \begin{cases}
            \mu_{i}- \mu_{i+1} = (\Lambda,\alpha_i), \quad  \text{for} \quad i\in\{1,\dots, n-1\},\\
            \sum_{i=1}^n \mu_i = u.
        \end{cases}           
        \end{equation}     
    \end{proposition}
    Note that \(\mu_i\) are eigenvalues of \(E_{ii}\otimes t^0\in \widehat{\mathfrak{gl}}_n\)  on the highest weight vector in \(\mathcal{L}_{\Lambda,u}\).
    \begin{proof}
    By analogy with \cite[Thm.4.1]{kodera2021guay} for \(\operatorname{ev}_{-}\) we have
    \begin{equation}
        \mathbf{x}_i^{+}(v)\zeta_{\Lambda}\otimes \zeta^{k}_u = 0,\quad \mathbf{h}_i(v)\zeta_{\Lambda}\otimes \zeta^{k}_u = \mathbf{Q}_i(v)\zeta_{\Lambda}\otimes \zeta^{k}_u, \quad \text{for} \quad i\in\{1,\dots,n\},
    \end{equation}
    where \(\mathbf{Q}_i(v)\) are determined by formula \eqref{eq:highestweight}. By Theorem \ref{th:heisenberg}  \(\operatorname{ev}_{-}^{*}\mathcal{L}_{\Lambda,u}\) is generated by \(\zeta_{\Lambda}\otimes \zeta^{k}_u\). 
    So, there is a surjective homomorphism \(j:\mathtt{M}_{\vec{\mathbf{Q}}}\rightarrow\operatorname{ev}_{-}^{*}\mathcal{L}_{\Lambda,u}\). The module \(\operatorname{ev}_{-}^{*}\mathcal{L}_{\Lambda,u}\) is irreducible. So, by Proposition \ref{th:IrrModYangian}  we have  \(
        \mathtt{L}_{\vec{\mathbf{Q}}}\simeq \operatorname{ev}_{-}^{*}\mathcal{L}_{\Lambda,u}\).
    \end{proof}

\section{Yangian action on thin modules}\label{sec:thin}
    \subsection{Yangian action on Gelfand-Tsetlin patterns}

    Let \(\underline{d}\) be pattern  of integer nonnegative numbers \(\{d_{i,j}\}_{i
    \ge j\in\mathbb{Z}}\). We call \(\underline{d}\) an \textit{affine Gelfand-Tsetlin pattern}  if
    \begin{enumerate}
        \item \label{def:dperidicity} \(d_{i,j}\ge d_{i+1,j}\) for all \(i\ge j\), %\zh{I would write: for all $i\geq j$}
        \item \(d_{i+n,j+n} = d_{i,j}\) for all \(i\ge j\), %\zh{I would write: for all $i\geq j$}
        \item \label{def:finitness} 
        for any $i\in\mathbb{Z}$ only finitely many numbers $d_{ij}$ are non-zero. %\zh{I would write: for any $i\in\mathbb{Z}$ only finitely many numbers $d_{ij}$ are non-zero}
    \end{enumerate}
    We denote the set of affine Gelfand-Tsetlin patterns by \(\widehat{\mathtt{GT}}\).
%    \begin{align}
%        d_i(\underline{d}) &= \sum_{j\le i} d_{ij},\\ \label{eq:lengthd}
%        L(\underline{d}) &= \max(i - j| d_{i,j}>0).
%    \end{align}
 We denote the affine Gelfand-Tsetlin pattern with all \(d_{ij} = 0\)  by \(\underline{0}\).
    
    In figures, we denote the inequality \(a\le b\) by \(a \leftarrow b\). % We use black arrows to denote inequalities that are satisfied by any pattern from \(\widehat{\mathtt{GT}}\). 
     See the conditions for affine Gelfand-Tsetlin patterns on Figure \ref{pic:gt-pattern}.
    \begin{figure}[h]
        \centering
        \begin{tikzpicture}!%[scale=0.8]
            \node (a00) at (-1,1) {\(\dots\)};
            
            \node (a10) at (-1,0) {\(d_{1,0}\)};        \node (a11) at (0,0) {\(d_{1,1}\)};
    
            \node (a20) at (-1,-1) {\(d_{2,0}\)};
            \node (a21) at (0,-1) {\(d_{2,1}\)};
            \node (a22) at (1,-1) {\(d_{2,2}\)};
    
            \node (a30) at (-1,-2) {\(\dots\)};
            \node (a31) at (0,-2) {\(d_{3,1}\)};
            \node (a32) at (1,-2) {\(d_{3,2}\)};
            \node (a33) at (2,-2) {\(d_{3,3}\)};
    
            \node (a41) at (0,-3) {\(\dots\)};
            \node (a42) at (1,-3) {\(\dots\)};
            \node (a43) at (2,-3) {\(\dots\)};
            \node (a44) at (3,-3) {\(\dots\)};
    
            \draw[->] (a00) -- (a10);
    
            \draw[->] (a10) -- (a20);
            \draw[->] (a11) -- (a21);
            
            \draw[->] (a20) -- (a30);
            \draw[->] (a21) -- (a31);
            \draw[->] (a22) -- (a32);
    
            \draw[->] (a31) -- (a41);
            \draw[->] (a32) -- (a42);
            \draw[->] (a33) -- (a43);
        \end{tikzpicture}
        \caption{Inequalities for \(\underline{d}\in\widehat{\mathtt{GT}}\)}
        \label{pic:gt-pattern}
    \end{figure}
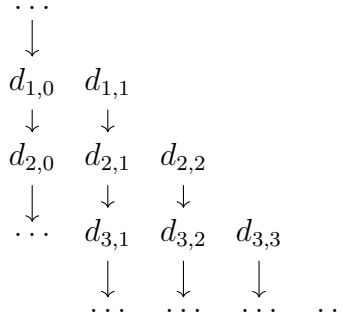

    \begin{remark}
        Every pattern \(\underline d\in\widehat{\mathtt{GT}}\) corresponds to an \(n\)-tuple of Young diagrams \(\{\lambda^{(a)}\}_{a=1}^n\) via
        \begin{equation}\label{eq:nYoungDiagrams}
            \lambda^{(a)}_j=d_{a+j-1,a}.
        \end{equation}
    \end{remark}
     For a given  \((\Lambda,u)\in\omega_0^\perp\oplus\mathbb{C}\),  define a sequence \(\{\mathtt{y}_i(\Lambda,u)\}_{i\in\mathbb{Z}}\)  by the system of equations 
    \begin{align}\label{eq:yi}
        \begin{cases}
            \sum_{s = 1}^n \mathtt{y}_{s}(\Lambda,u)=  u -\frac{n(n+2)}{2},\\
            \mathtt{y}_{l}(\Lambda,u) - \mathtt{y}_{r}(\Lambda,u) = \big(\sum_{s=l}^{r - 1}\alpha_s, \Lambda+\rho\big) \quad \text{for any } l<r\in\mathbb{Z} .               
        \end{cases}
    \end{align}
    Recall that \(\mu_i\) are defined by formulas \eqref{eq:lambdai}. Then,
        \begin{equation}
            \mathtt{y}_{i} = \mu_i - i - \frac{1}{2}, \quad \text{for} \quad i\in\{1,\dots, n\}.
        \end{equation}
    Note, that
    \begin{equation}
           \mathtt{y}_{i+n}(\Lambda,u) = \mathtt{y}_{i}(\Lambda,u)-\kappa, ~ \kappa = (\delta,\Lambda+\rho).
    \end{equation}
    Let
    \begin{equation}
        p_{ij}(\Lambda,u,\underline{d}) = -\mathtt{y}_j(\Lambda,u)+d_{ij}.
    \end{equation}
     Denote \(d_i(\underline{d}) = \sum_{j\le i} d_{ij}\). To simplify notations we do not write the arguments of \(\mathtt{y}_i, p_{ij}, d_i\) further.

For a pattern $\underline d$ denote \(\underline{d}^{\pm}_{ij}\) the patterns obtained from $\underline d$  by adding $\pm 1$ to element $d_{ij}$ . The same change is applied simultaneously to all of its periodic copies.
    Let \begin{equation}\label{eq:GLu}
        \mathbb{M}_{\Lambda,u}= \bigoplus_{\underline{d}\in\widehat{\mathtt{GT}}}\mathbb{C}\xi_{\underline{d}}.
    \end{equation}
    \begin{theorem}[{\cite{feigin2011yangians}[Thm.3.11]}]\label{th:slactsonGTVerma}
    For a generic weight \(\Lambda\)  there is an action of \(\widehat{\mathfrak{sl}}_n\) on  \(\mathbb{M}_{\Lambda,u}\) given by %formulas
        \begin{subequations}\label{eq:ActionInGTBasis}
            \begin{align}
                h_{i}^{\Lambda,u}\xi_{\underline{d}} &= \xi_{\underline{d}}((\mathtt{y}_i - \mathtt{y}_{i+1}-1) + d_{i-1} - 2d_i + d_{i+1}),\\
                e_{i}^{\Lambda,u} \xi_{\underline{d}} &=-\sum_{j\le i}\xi_{\underline{d}^{-}_{ij}}(p_{i+1,j}-p_{ij})(p_{i+1,i+1}-p_{ij})\prod_{s\leq i, s\ne j}\frac{p_{i+1,s}-p_{ij}}{p_{is}-p_{ij}},                                 
               \\            f_{i}^{\Lambda,u}\xi_{\underline{d}}&= 
               \xi_{\underline{d}^{+}_{ii}}\prod_{s\leq i-1} \frac{p_{i-1,s}-p_{ii}}{p_{is}-p_{ii}} + \sum_{j< i}\xi_{\underline{d}^{+}_{ij}}\frac{p_{i-1,j}-p_{ij}}{p_{ii}-p_{ij}}\prod_{s\leq i-1, s\neq j} \frac{p_{i-1,s}-p_{ij}}{p_{is}-p_{ij}}.
            \end{align}
        \end{subequations}
%       \zh{\sout{where \(\underline{d}^{\pm}_{ij}\) is a pattern with entries the same as \(\underline{d}\) except that \(d_{ij}\) is changed to \(d_{ij}\pm 1\) simultaneously with  all \(d_{i+mn,j+mn}\) for \(m\in\mathbb{Z}\).}}
        \qed
    \end{theorem}
    \begin{remark}
        This action does not make  \(\mathbb{M}_{\Lambda,u}\) an irreducible \(\widehat{\mathfrak{sl}}_n\) module. In fact, there is an action of the Heisenberg algebra on \(\mathbb{M}_{\Lambda,u}\) commuting with \(\widehat{\mathfrak{sl}}_n\), we will see it in the Corollary \ref{th:GLambdauHeisenbergAction}.
    \end{remark}
        
    \begin{theorem}[{\cite{feigin2011yangians}[Thm.3.20]}]\label{th:YactsonGTVerma}
        If \(\Lambda\) is generic then there is an  action of Yangian \(\widetilde{Y}(1,-\kappa)\) on \(\mathbb{M}_{\Lambda,u}\) given by %formulas 
        \begin{subequations}\label{eq:ActionInGTBasisYangian}
            \begin{align}
                \mathbf{x}^{+, \Lambda,u}_{i}(v) \xi_{\underline{d}} &= \sum_j \left(-v -(p_{ij}-\frac{1+i}{2})\right)^{-1} (-e_{i,\underline{d},\underline{d}^{-}_{ij}})\xi_{\underline{d}^{-}_{ij}},
                \\   
                \mathbf{x}^{-, \Lambda,u}_{i}(v) \xi_{\underline{d}} &= \sum_j\left(-v - (p_{ij}+\frac{1-i}{2})\right)^{-1}(-f_{i,\underline{d},\underline{d}^{+}_{ij}})\xi_{\underline{d}^{+}_{ij}},\\ \label{eq:ActionInGTBasisYangianH} 
                \mathbf{h}_{i}^{\Lambda,u}(v)\xi_{\underline{d}} &= \frac{(-v+\frac{i+1}{2}-p_{i+1,i+1})}{(-v+\frac{i-1}{2}-p_{ii})}\prod_{j\le i} \frac{(-v+\frac{i+1}{2}-p_{i+1,j})(-v+\frac{i-1}{2}-p_{i-1,j-1})}{(-v+\frac{i+1}{2}-p_{ij})(-v+\frac{i-1}{2}-p_{i,j-1})}\xi_{\underline{d}},\\
                D \xi_{\underline{d}} &= - d_n \xi_{\underline{d}}.
            \end{align} 
        \end{subequations}
        \qed
    \end{theorem}
     Note that in formulas \eqref{eq:ActionInGTBasis} and \eqref{eq:ActionInGTBasisYangianH} for any \(\underline{d}\in\widehat{\mathtt{GT}}\) only finite number of multipliers in product that are not equal to \(1\).
     
     We use notations \(e_{i,\underline{d},\underline{d}'}^{\Lambda,u}, f_{i,\underline{d},\underline{d}'}^{\Lambda,u},h_{i,\underline{d},\underline{d}'}^{\Lambda,u}\) and \(\mathbf{x}_{i,r,\underline{d},\underline{d}'}^{\pm,\Lambda,u},\mathbf{h}_{i,r,\underline{d},\underline{d}'}^{\Lambda,u}\) for the matrix elements between \(\xi_{\underline{d}}\) and \(\xi_{\underline{d}'}\).
     \begin{remark}
         These matrix elements are rational functions in \(\Lambda\). For generic \(\Lambda\) they do not vanish and are well-defined. However, for some weights $\Lambda$ the formulas are not well-defined.
     \end{remark}
    \begin{remark}
         The action of \(\widehat{\mathfrak{sl}}_n\) and \(\widehat{Y}(1,-\kappa)\) on \(\bigoplus_{\underline{d}\in\widehat{\mathtt{GT}}}\mathbb{C}\xi_{\underline{d}}\) given in \cite{feigin2011yangians} defines the structure of a lowest weight \(\widehat{\mathfrak{sl}}_n\)-module and a lowest weight \(\widehat{Y}(1,-\kappa)\)-module, respectively. We use \(F_{\widehat{\mathfrak{sl}}_n}\) and \(F_{Y}\) to make it a highest weight module.

        Note that in \cite{feigin2011yangians} there is a  misprint in the Yangian action, more precisely \(\mathbf{x}^{+}_{i}(v)\) are mixed with  \(\mathbf{x}^{-}_{i}(v)\).
    \end{remark}
    \begin{remark}
        The eigenvalues of \(\mathbf{h}_{i}(v)\) on the highest weight vector  \(\xi_{\underline{0}}\in \mathbb{M}_{\Lambda,u}\) are given by \eqref{eq:highestweight}.
        
        %\zh{I would remove the word "formulas" from the text everywhere. Then at least I do not need to think if it is plural or single.} \s{Here I agree: I will try to get rid of it where is possible. But not everywhere, e.g. in Remark \(3.5\) it seems impossible for me now} 
        %\zh{You can leave it in  Remark 3.5 but why not to say: The action given in [DDNR11] on $\xi_d$ defines a structure.... ?} 
%        \begin{equation}
%            \mathbf{Q}_i(v) = \frac{- v + \lambda_{i+1}-\frac{i}{2}-1}{- v + \lambda_{i}-\frac{i}{2}-1}, \quad \text{for} \quad i\in\{1,\dots, n\}.
%        \end{equation}
    \end{remark}
    \begin{theorem}
        For a generic \(\Lambda\) we have an isomorphism of \(\widehat{Y}(1,-\kappa)\)-modules
        \begin{equation}
             \mathbb{M}_{\Lambda,u}\simeq \operatorname{ev}_{-}^{*}\mathcal{L}_{\Lambda,u}\simeq \mathtt{L}_{\vec{\mathbf{Q}}},
        \end{equation}
        for \(\mathbf{Q}_i(v)\) determined by formulas \eqref{eq:highestweight}.    
    \end{theorem}
    \begin{proof}          
        %The action of \(\mathbf{h}_{i}(v)\) on highest weight vector in \(\mathbb{M}_{\Lambda,u}\) and  \(\operatorname{ev}_{-}^{*}\mathcal{M}_{\Lambda,u}\) is the same. It is clear that \(\operatorname{ev}_{-}^{*}\mathcal{M}_{\Lambda,u}\) is irreducible  module over \(\widehat{Y}(1,-\kappa)\). It remains to show that \(\mathbb{M}_{\Lambda,u}\) is also irreducible \(\widehat{Y}(1,-\kappa)\)-module. 
        We will prove this statement for a wider class of weights in  Theorem \ref{th:Main}.        
        %we have an epimorphism \(\mathtt{f}:\mathbb{M}_{\Lambda,u}\rightarrow \operatorname{ev}_{-}^{*}\mathcal{M}_{\Lambda,u}\). It is clear that \(\mathbb{M}_{\Lambda,u}\) and  \(\operatorname{ev}_{-}^{*}\mathcal{M}_{\Lambda,u}\) have the same character, so \(\mathtt{f}\) is forced to bijection.
    \end{proof}
    \begin{corollary}\label{th:GLambdauHeisenbergAction}
        There is an action of \(\mathtt{Heis}\) on \(\mathbb{M}_{\Lambda,u}\)commuting with \(\widehat{\mathfrak{sl}}_n\). Together, these actions endow  \(\mathbb{M}_{\Lambda,u}\) with a structure of \(\widehat{\mathfrak{gl}}_n\)-module isomorphic to \(\mathcal{M}_{\Lambda,u}\).    \qed
    \end{corollary}
%\zh{Any infinite-dimensional space can be made to a $\widehat{\mathfrak{gl}}_n$-module. What do you want to say by this corollary? }

    \subsection{Dominant weights and permitted Gelfand-Tsetlin patterns}\label{sec:permitted}
        
    We say \(\Lambda_{\text{dom}} \in \omega_0^{\bot} \) is \textit{dominant} if
    \begin{equation}
        \bigl(\Lambda_{\text{dom}} + \rho,\;\alpha\bigr) \notin \mathbb{Z}_{\le 0}
        \qquad
        \text{for every } \alpha\in\widehat{\Phi}_{+}.%^{\mathrm{re}}.
    \end{equation}
    In particular, \( \kappa = (\Lambda_{\text{dom}} + \rho,\delta)\not\in \mathbb{Q}_{\le 0}\). Denote 
    \begin{equation}
        \widehat{\Phi}_{\Lambda_{\text{dom}}} =\{\beta\mid \beta\in \widehat{\Phi}, ~  (\beta,\Lambda_{\text{dom}}+\rho)\in\mathbb{Z} \}
    \end{equation}
    and let \(\widehat{W}_{\Lambda_{\text{dom}}}\) be the group generated by \(s_\alpha\) for \(\alpha \in \widehat{\Phi}_{+,\Lambda_{\text{dom}}}^{\mathrm{re}} = \widehat{\Phi}_{\Lambda_{\text{dom}}}\cap \widehat{\Phi}^{\mathrm{re}}_+\).   
    \begin{remark}
        A  dominant weight is not necessarily integral. For example, a generic weight \(\Lambda\) is dominant. Further, we will consider admissible weights. The admissible weights are dominant.
    \end{remark}
    Consider a root \(\alpha_{l,r}\in\widehat{\Phi}^{\text{re}}_{+,\Lambda_{\text{dom}}}\) for \(l<r\in\mathbb{Z}\).
    Then
    \begin{equation}
        (\alpha_{l,r},\Lambda_{\text{dom}} + \rho) = \mathtt{y}_l -\mathtt{y}_r\in\mathbb{Z}_{> 0}. 
    \end{equation}    
    We say that \(\underline{d}\in\widehat{\mathtt{GT}}\) is \textit{permitted with respect to a pair} \(\Lambda_{\text{dom}},\alpha_{l,r}\) if for any \(i\ge r-1\) we have
    \begin{equation}\label{eq:DashedIneq}
        d_{i,l}\le d_{i+1,r} + \mathtt{y}_l -\mathtt{y}_r-1. 
    \end{equation}
    \begin{remark}
        Let \(m\delta\in \widehat{\Phi}_{+,\Lambda_{\mathrm{dom}}}\). Then, for any pair
        \((l,r)\) satisfying \(r-l=mn\), we have
        \(m\delta=\alpha_{l,r}\).
        Moreover, condition~\eqref{eq:DashedIneq} is automatically satisfied for every
        \(\underline d\in\widehat{\mathtt{GT}}\).
        Thus, every
        \(\underline d\in\widehat{\mathtt{GT}}\)
        is said to be permitted with respect to pair
        \(\Lambda_{\mathrm{dom}}, m\delta\).
    \end{remark}
    
    For convenience on figures we denote this inequality by blue arrow with index as follows
    \(
        d_{i,l}{\color{blue}\xleftarrow{\mathtt{y}_l -\mathtt{y}_r-1}} d_{i+1,r}
    \).  
    See Figure \ref{pic:gt-pattern-perm1}, which illustrates a pattern \(\underline{d}\)  permitted with respect to the pair \(\Lambda_{\text{dom}},\alpha_0 +\alpha_1\). 
    \begin{figure}
        \centering
        \begin{tikzpicture}[every node/.style={font=\small}]
    
        \node (a00) at (-1,1) {\(\dots\)};
        
        \node (a10) at (-1,0) {\(d_{1,0}\)};        \node (a11) at (0,0) {\(d_{1,1}\)};
    
        \node (a20) at (-1,-1) {\(d_{2,0}\)};
        \node (a21) at (0,-1) {\(d_{2,1}\)};
        \node (a22) at (1,-1) {\(d_{2,2}\)};
    
        \node (a30) at (-1,-2) {\(\dots\)};
        \node (a31) at (0,-2) {\(d_{3,1}\)};
        \node (a32) at (1,-2) {\(d_{3,2}\)};
        \node (a33) at (2,-2) {\(d_{3,3}\)};
    
        \node (a41) at (0,-3) {\(\dots\)};
        \node (a42) at (1,-3) {\(\dots\)};
        \node (a43) at (2,-3) {\(\dots\)};
        \node (a44) at (3,-3) {\(\dots\)};
    
        \draw[->] (a00) -- (a10);
    
        \draw[->] (a10) -- (a20);
        \draw[->] (a11) -- (a21);
        
        \draw[->] (a20) -- (a30);
        \draw[->] (a21) -- (a31);
        \draw[->] (a22) -- (a32);
    
        \draw[->] (a31) -- (a41);
        \draw[->] (a32) -- (a42);
        \draw[->] (a33) -- (a43);
    
        % Пунктирные красные стрелки с метками
        \draw[blue,   ->] (a32) -- (a20) node[midway, right, xshift=2pt, yshift=2pt] {\scriptsize\(\scriptstyle \mathtt{y}_0-\mathtt{y}_2-1\)};
        \draw[blue,   ->] (a22) -- (a10) node[midway, right, xshift=2pt, yshift=2pt] {\scriptsize\(\scriptstyle \mathtt{y}_0-\mathtt{y}_2-1\)};
        \end{tikzpicture}
        \caption{Inequalities for \(\underline{d}\in\widehat{\mathtt{GT}}\) permitted with respect to pair \(\Lambda_{\text{dom}},\alpha_0 +\alpha_1\).}
        \label{pic:gt-pattern-perm1}
    \end{figure}
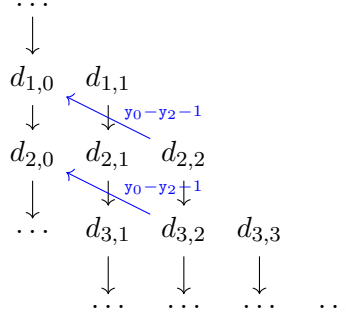
    The inequalities \eqref{eq:DashedIneq} imply
    \begin{equation}\label{eq:IneqOnPij}
        p_{i,l}<p_{i+1,r}, \quad\quad p_{i,l}<p_{i,r}.
    \end{equation}
    We say that \(\underline{d}\in\widehat{\mathtt{GT}}\) is \textit{permitted with respect to} \(\Lambda_{\text{dom}}\) if \(\underline{d}\) is permitted with respect to the pair \(\Lambda_{\text{dom}},\beta\) for any \(\beta\in\widehat{\Phi}_{+,\Lambda_{\text{dom}}}\). Denote the set of all permitted patterns by \(\widehat{\mathtt{GT}}_{\text{perm}}(\Lambda_{\text{dom}})\).  Consider a vector space
     \begin{equation}\label{eq:PLu}
            \mathbb{L}_{\Lambda_{\text{dom}},u} = \bigoplus_{\underline{d}\in\widehat{\mathtt{GT}}_{\text{perm}}(\Lambda_{\text{dom}})}\mathbb{C}\xi_{\underline{d}}.
    \end{equation}  
    
    %The goal of the remaining part of the subsection is to show that a structure of \(\widehat{Y}_\circ(1,-\kappa)\)-module is correctly defined  on \(\mathbb{L}_{\Lambda_{\text{dom}},u}\) and prove the main Theorem \ref{th:Main} of the paper.
%\zh{Can we write instead: 
In the remaining part of this section we show that \(\mathbb{L}_{\Lambda_{\text{dom}},u}\) inherits a structure of  \(\widehat{Y}_\circ(1,-\kappa)\)-module, see Theorem \ref{th:Main} below.%}

    \begin{remark}
        A generic weight \(\Lambda\in \omega_0^\perp\) is dominant, then the conditions \eqref{eq:DashedIneq} are void and we have 
        \begin{equation}
            \mathbb{L}_{\Lambda,u} = \mathbb{M}_{\Lambda,u}.
        \end{equation}
    \end{remark}
     %that  \(\mathbb{L}_{\Lambda_{\text{dom}},u}=\operatorname{ev}_{-}^{*}\mathcal{L}_{\Lambda_{\text{dom}},u}\). 
    \begin{lemma}\label{th:pijnotequal}
        For a dominant weight  \(\Lambda_{\text{dom}}\) and \(\underline{d}\in \widehat{\mathtt{GT}}_{\text{perm}}(\Lambda_{\text{dom}})\) we have
        \begin{enumerate}
            \item If \(p_{i,l} = p_{i+1,r}\) for \(l\le i\) and \(r\le i+1\),  then \(l = r\) and \(d_{i,l}=d_{i+1,l}\),
            \item 
             If \(p_{i,l} = p_{i,r}\) for \(l,r\le i\), then \(l=r\).
        \end{enumerate}
    \end{lemma}
    \begin{proof}
        \begin{enumerate}
            \item  Let \(p_{i,l}=p_{i+1,r}\). Assume that \(l<r\). Then \(\alpha_{l,r} \in \widehat{\Phi}_{+,\Lambda}\). So  for \(\underline{d}\in \widehat{\mathtt{GT}}_{\text{perm}}(\Lambda_{\text{dom}})\)  we have \(p_{i,l}<p_{i+1,r}\) and which gives a contradiction.  If \(l>r\), then \(\alpha_{r,l} \in \widehat{\Phi}_{+,\Lambda}\) and we have 
            \begin{equation}
                p_{i,r}\ge p_{i+1,r}=p_{i,l}\ge p_{i+1,l}.
            \end{equation} 
            This contradicts the inequality for permitted patterns \(p_{i,r}<p_{i+1,l}\). So,  \(l=r\) and we have \(d_{i,l}=d_{i+1,l}\).
            \item Let \(p_{i,l} = p_{i,r}\). Assume that \(l\neq r\) and without loss of generality \(l<r\). Then \(\alpha_{l,r} \in \widehat{\Phi}_{+,\Lambda}\) and
            \begin{equation}
                p_{i,l}< p_{i+1,r}\le p_{i,r}.
            \end{equation}
            So, we have a contradiction.
        \end{enumerate}
    \end{proof}

     \begin{corollary}\label{th:WellDefMatEl}
        We have the following properties of matrix elements as functions of \(\Lambda\in\omega_0^{\perp}\)
        \begin{enumerate}
            \item \(\mathbf{x}_{i,r,\underline{d},\underline{d}'}^{\pm,\Lambda,u},\mathbf{h}_{i,r,\underline{d},\underline{d}}^{\Lambda,u}\) do not have poles at \(\Lambda =\Lambda_{\text{dom}}\) for \(\underline{d}\in \widehat{\mathtt{GT}}_{\text{perm}}(\Lambda_{\text{dom}})\)
            ,            
            \item  \(\mathbf{x}_{i,r,\underline{d},\underline{d}'}^{\pm,\Lambda,u}\) have zeroes and do not have poles at \(\Lambda=\Lambda_{\text{dom}}\)  for \(\underline{d}\in \widehat{\mathtt{GT}}\setminus\widehat{\mathtt{GT}}_{\text{perm}}(\Lambda_{\text{dom}})\) and \(\underline{d}'\in \widehat{\mathtt{GT}}_{\text{perm}}(\Lambda_{\text{dom}})\);  
            \item \(\mathbf{x}_{i,0,\underline{d},\underline{d}^{\mp}_{ij}}^{\pm,\Lambda,u}\)  do not have zeros in \(\Lambda =\Lambda_{\text{dom}}\)  for \(\underline{d},\underline{d}^{\mp}_{ij}\in \widehat{\mathtt{GT}}_{\text{perm}}(\Lambda_{\text{dom}})\).            
        \end{enumerate}           
     \end{corollary}
     \begin{proof}
         It follows from \eqref{eq:ActionInGTBasisYangian} and Lemma \ref{th:pijnotequal}.
     \end{proof}
    \begin{corollary}
        Operators \(\tilde{\mathbf{x}}_{i,r}^{\pm},\tilde{\mathbf{h}}_{i,r}:\mathbb{L}_{\Lambda_{\text{dom}},u}\rightarrow \mathbb{L}_{\Lambda_{\text{dom}},u}\) given in the basis \(\{\xi_{\underline{d}}\}_{\underline{d}\in\widehat{\mathtt{GT}}_{\text{perm}}(\Lambda_{\text{dom}}) }\)  by %formulas
        \begin{equation}\label{eq:restrictedaction}
        \begin{aligned}
            \tilde{\mathbf{x}}_{i,r}^{\pm} \xi_{\underline{d}} &= \sum_{\underline{d}'\in\widehat{\mathtt{GT}}_{\text{perm}}(\Lambda_{\text{dom}}) }\mathbf{x}_{i,\underline{d},\underline{d}'}^{\pm,\Lambda_{\text{dom}},u}[r]\xi_{\underline{d}'},\\
            \tilde{\mathbf{h}}_{i}(v) \xi_{\underline{d}} &= \mathbf{h}_{i}^{\Lambda_{\text{dom}},u}(v) \xi_{\underline{d}} .         
        \end{aligned}
        \end{equation}
        are well-defined.
    \end{corollary} 
    \begin{proof}
        All corresponding matrix elements do not have poles for \(\Lambda = \Lambda_{\text{dom}}\).
    \end{proof}

    \begin{proposition}
        Operators \(\tilde{\mathbf{x}}_{i,r}^{\pm},\tilde{\mathbf{h}}_{i,r}:\mathbb{L}_{\Lambda_{\text{dom}},u}\rightarrow \mathbb{L}_{\Lambda_{\text{dom}},u}\) 
        define a structure of \(\widehat{Y}_\circ(1,-\kappa)\)-module on \(\mathbb{L}_{\Lambda_{\text{dom}},u}\) .
    \end{proposition}
    \begin{proof}
        It suffices to check that \(\tilde{\mathbf{x}}_{i,r}^{\pm},\tilde{\mathbf{h}}_{i,r}\) satisfy commutation relations of \(\widehat{Y}_\circ(1,-\kappa)\). Consider a projection \(\mathtt{P}^{\Lambda}_{\Lambda_{\text{dom}}}:\mathbb{M}_{\Lambda,u}\rightarrow \mathbb{L}_{\Lambda_{\text{dom}},u}\) such that
        \begin{equation}
            \mathtt{P}^{\Lambda}_{\Lambda_{\text{dom}}}(\xi_{\underline{d}}) = \begin{cases}
                \xi_{\underline{d}}, \quad \text{for} \quad \underline{d}\in \widehat{\mathtt{GT}}_{\text{perm}}(\Lambda_{\text{dom}}),\\
                0, \quad \text{otherwise.}
            \end{cases}
        \end{equation}
        
        Consider any quadratic expression \(A^\Lambda B^\Lambda\) where \(A^\Lambda,B^\Lambda\in\{\mathbf{x}_{i,r}^{\pm,\Lambda,u}, \mathbf{h}_{i,r}^{\Lambda,u}\}\) with matrix elements \(A^\Lambda_{\underline{d},\underline{d}'}, B^\Lambda_{\underline{d},\underline{d}'}\) for \(\underline{d},\underline{d}'\in\widehat{\mathtt{GT}}\). Fix \(\underline{d}\in \widehat{\mathtt{GT}}_{\text{perm}}(\Lambda_{\text{dom}})\). Denote by \(\tilde A,\tilde B\)  corresponding element in \(\{\tilde{\mathbf{x}}_{i,r}^{\pm}, \tilde{\mathbf{h}}_{i,r}\}\).  We have
        \begin{equation}
            \mathtt{P}^{\Lambda}_{\Lambda_{\text{dom}}}A^\Lambda B^\Lambda\xi_{\underline{d}} = \sum_{\underline{d}''\in\widehat{\mathtt{GT}}_{\text{perm}}(\Lambda_{\text{dom}})}\sum_{\underline{d}'\in\widehat{\mathtt{GT}}}A^\Lambda_{\underline{d}',\underline{d}''}B^\Lambda_{\underline{d},\underline{d}'}\xi_{\underline{d}''}.
        \end{equation}
        Then by Corollary \ref{th:WellDefMatEl}
        \begin{equation}
            (\mathtt{P}^{\Lambda}_{\Lambda_{\text{dom}}}A^\Lambda B^\Lambda\xi_{\underline{d}})|_{\Lambda=\Lambda_{\text{dom}} }= \sum_{\underline{d}''\in\widehat{\mathtt{GT}}_{\text{perm}}(\Lambda_{\text{dom}})}\sum_{\underline{d}'\in\widehat{\mathtt{GT}}_{\text{perm}}(\Lambda_{\text{dom}})}A^{\Lambda_{\text{dom}}}_{\underline{d}',\underline{d}''}B^{\Lambda_{\text{dom}}}_{\underline{d},\underline{d}'}\xi_{\underline{d}''} = \tilde A \tilde B \xi_{\underline{d}}.
        \end{equation}
        So, if a quadratic expression of \(\mathbf{x}_{i,r}^{\pm,\Lambda,u}, \mathbf{h}_{i,r}^{\Lambda,u}\) is equal to zero, then the same expression in operators  \(\tilde{\mathbf{x}}_{i,r}^{\pm}, \tilde{\mathbf{h}}_{i,r}\) also vanishes.
    \end{proof}

    %We call  \(\widehat{Y}_\circ(1,-\kappa)\)-module \textit{thin} %(or  \textit{tame}) 
    %if \(\widehat{Y}^0(1,-\kappa)\) acts semi simply and for any two eigenvectors \(\xi_1, \xi_2\) there exists  \(h\in\widehat{Y}^0(1,-\kappa)\) such that its eigenvalues on \(\xi_1, \xi_2\) are different.    
    We call a $\widehat{Y}_\circ(1,-\kappa)$-module \textit{thin} if the action of $\widehat{Y}^0(1,-\kappa)$ is semisimple and all eigenspaces are one-dimensional.
    \begin{lemma}\label{th:thin}
        For  \(\kappa \neq n\)  the module \(\mathbb{L}_{\Lambda_{\text{dom}},u}\) is a thin \(\widehat{Y}_\circ(1,-\kappa)\)-module.        
    \end{lemma}
    \begin{proof}
        One has to check that for any \(\underline{d},\underline{d}'\in\widehat{\mathtt{GT}}_{\text{perm}}(\Lambda_{\text{dom}})\) there exists an element \(h\in\widehat{Y}^0(1,-\kappa)\) such that eigenvalues of \(h\) on \(\underline{d},\underline{d}'\) are different. We divide the proof into \(3\) steps.
        
        \textbf{Step \(1\).} 
        % Consider series \( \mathbf{a}_{i}(v) = v^{-1}+\sum_{r=1}^{\infty}\mathbf{a}_{i,r} v^{-r-1}\in \widehat{Y}^0(1,-\kappa)[[v^{-1}]]\) for \(i\in\mathbb{Z}\) given by formulas
        % \begin{subequations}
        %     \begin{align}\label{eq:a&h}
        %     \mathbf{h}_i(-v) &= \frac{\mathbf{a}_{i-1}\left(v+\frac{i-1}{2}\right)}{\mathbf{a}_i\left(v+\frac{i+1}{2}\right)}, \\ \label{eq:a_period}  
        %     \mathbf{a}_{i+n}(v) &= \mathbf{a}_{i}(v - \kappa).
        %     \end{align}
        % \end{subequations}
        % Elements \(\mathbf{a}_{i}(v)\) are well-defined. Indeed, equation \eqref{eq:a&h} implies
        % \begin{equation}
        %     \frac{\mathbf{a}_{i}(v)}{\mathbf{a}_{i+n}(v+n)}\! =  \mathbf{h}_{i+1}\!\!\left(-v+\frac{i}{2}\right) \mathbf{h}_{i+2}\!\!\left(-v+\frac{i-1}{2}\right) \dots \mathbf{h}_{i+n}\!\!\left(-v-n+1+\frac{i-n+1}{2}\right)\!,\!
        % \end{equation}  
        % The right handside is in \(\widehat{Y}^0(1,-\kappa)\), denote it by \(\mathbf{H}_i(v)\). By formula \eqref{eq:a_period} we have
        % \begin{equation}\label{eq:FunEqOnAi}
        %     \frac{\mathbf{a}_{i}(v)}{\mathbf{a}_{i}(v+n-\kappa)} = \mathbf{H}_i(v).
        % \end{equation}
        % Expanding the series, we obtain
        % \begin{equation}
        %     \frac{\mathbf{a}_{i}(v)}{\mathbf{a}_{i}(v+n-\kappa)} = 1 + (n-\kappa)v^{-1} + \sum_{r=2}^{\infty} \Big(r(n-\kappa) \mathbf{a}_{i,r-1}+P_r(\{\mathbf{a}_{i,s}\}_{s=0}^{r-2})\Big)v^{-r-1},
        % \end{equation}
        % where \(P_r\) are polynomials. Hence, the elements \(\mathbf{a}_{i,r}\) are determined recursively by formula \eqref{eq:FunEqOnAi}.  
        Consider series \( \mathbf{a}_{i}(v) = v^{-1}+\sum_{r=1}^{\infty}\mathbf{a}_{i,r} v^{-r-1}\in \widehat{Y}^0(1,-\kappa)[[v^{-1}]]\) for \(i\in\mathbb{Z}\) given by %formulas
        \begin{equation}\label{eq:FunEqOnAi}
        \frac{\mathbf{a}_{i}(v)}{\mathbf{a}_{i}(v+n-\kappa)} = 
        \mathbf{h}_{i+1}\left(-v+\frac{i}{2}\right) \mathbf{h}_{i+2}\left(-v+\frac{i-1}{2}\right) \dots \mathbf{h}_{i+n}\left(-v+\frac{i-n+1}{2}\right).
        \end{equation} 
        Expanding the left-hand side, we obtain
        \begin{equation}
            \frac{\mathbf{a}_{i}(v)}{\mathbf{a}_{i}(v+n-\kappa)} = 1 + (n-\kappa)v^{-1} + \sum_{r=2}^{\infty} \Big(r(n-\kappa) \mathbf{a}_{i,r-1}+P_r(\{\mathbf{a}_{i,s}\}_{s=0}^{r-2})\Big)v^{-r-1},
        \end{equation}
        where \(P_r\) are polynomials. Hence, for \(k\neq 0\) the elements \(\mathbf{a}_{i,r}\) are determined recursively by formula \eqref{eq:FunEqOnAi}.
        
        Let us show that the eigenvalue \( a_{i,\underline{d}}(v)\) of \(\mathbf{a}_{i}(v) \) on \(\xi_{\underline{d}}\in\mathbb{L}_{\Lambda_{\text{dom}},u}\) is given by
        \begin{equation}\label{eq:EigenvaluesAi}
           a_{i,\underline{d}}(v) =\frac{1}{(v-p_{i+1,i+1})}\prod_{j\le i}\frac{v-p_{i,j}}{v-p_{i+1,j}}.
        \end{equation}
        Indeed, applying  \eqref{eq:FunEqOnAi} to  \(\xi_{\underline{d}}\) and using \eqref{eq:ActionInGTBasisYangian} we get 
        \begin{multline}\label{eq:AiEigenFunctEq}
            \frac{a_{i,\underline{d}}(v)}{a_{i,\underline{d}}(v+n-\kappa)}=\frac{(v+n-p_{i+n+1,i+n+1})}{(v-p_{i+1,i+1})}\prod_{j\le i}\frac{v-p_{i,j}}{v-p_{i+1,j}}\frac{v+n-p_{i+n+1,j+n}}{v+n-p_{i+n,j+n}}   \\
            =\!\left(\!\frac{1}{(v-p_{i+1,i+1})}\prod_{j\le i}\frac{v-p_{i,j}}{v-p_{i+1,j}} \!\right)  \! \left(\! \frac{1}{(v+n-\kappa-p_{i+1,i+1})}\prod_{j\le i}\frac{v+n-\kappa-p_{i,j}}{v+n-\kappa-p_{i+1,j}} \! \right)^{-1}.  
        \end{multline} 
        Functional equation \eqref{eq:AiEigenFunctEq} on \(a_{i,\underline{d}}(v)\) has a unique solution by the argument as above. Moreover, \eqref{eq:EigenvaluesAi} satisfies the \eqref{eq:AiEigenFunctEq}.
        
        \textbf{Step \(2\).}      
        Consider a set 
        \begin{equation}
            Z^{(i)}_{\underline{d}} = \{ p_{i,j}(\underline{d})| j\le i,  d_{i,j}> 0\}.
        \end{equation}
        By the second part of Lemma \ref{th:pijnotequal} it has no multiplicities. 
        %Note that   by the first part of Lemma \ref{th:pijnotequal} a rational function \(a_{i,\underline{d}}(v)\) has poles and zeroes in \(p_{i+1,j}\) and \(p_{i,j}\) correspondingly for all \(j\le i\) such that \(p_{i,j}\neq p_{i+1,j}\).
        % Note that \(p_{i,j}(\underline{d})\not\in Z^{(i)}_{\underline{d}}\) if and only if 
        % \begin{equation}
        %     p_{i,j}(\underline{d}) = p_{i+1,j}(\underline{d}) =\dots = p_{i+L,j}(\underline{d}).
        % \end{equation}
        Let us show that 
        \begin{equation}
            Z^{(i)}_{\underline{d}}  =\bigcup_{l=0}^{\infty} V_{i,\underline{d}}^{(l)},
        \end{equation}
        where \(V_{i,\underline{d}}^{(l)}\) is a set of zeros of \begin{equation}
            A_{i,\underline{d}}^{(l)}(v) = \prod_{m=0}^{l} a_{i+m,\underline{d}}(v) =\prod_{s=0}^l \frac{1}{(v-p_{i+l+1,i+s+1})}\prod_{j\le i}\frac{v-p_{i,j}}{v-p_{i+l+1,j}}.
        \end{equation}
        If \(d_{i,j}=0\) then \(d_{i+s,j}=0\) for any \(s\ge 0\) hence \(p_{i+l+1,j} = p_{ij}\). Then multiplier \((v - p_{i,j})\)  contracts. So, \(Z^{(i)}_{\underline{d}}  \supset\bigcup_{l=0}^{\infty} V_{i,\underline{d}}^{(l)}\).
        
        In the other direction, assume there exists \(p_{i,j}\in Z^{(i)}_{\underline{d}}\) such that \(p_{i,j}\not\in V_{i,\underline{d}}^{(l)}\) for all  \(l\ge 0\). Then for each \(l\) there exists \(j_l\le i+l+1\) such that \(p_{i+l+1,j_l} = p_{i,j}\). But by the first part of Lemma \ref{th:pijnotequal} it means that \(j_l=j\)  and \(d_{i,j}  =d_{i+l,j}\) for all \(l\ge 0\). Thus, \(d_{i,j} = 0\) and \(p_{i,j}\not\in Z^{(i)}_{\underline{d}}\).

        \textbf{Step \(3\).} Assume that \(\underline{d}\neq \underline{d}'\in\widehat{\mathtt{GT}}_{\text{perm}}(\Lambda_{\text{dom}})\) such that
        \begin{equation}
             a_{i,\underline{d}}(v)= a_{i,\underline{d}'}(v), \quad \quad \text{for any} \quad i\in\mathbb{Z}.
        \end{equation}
        Then \(Z_{\underline{d}}^{(i)} = Z_{\underline{d}'}^{(i)}\) for \(i\in\mathbb{Z}\). For fixed \(i\) show that equality \( Z^{(i)}_{\underline{d}} = Z^{(i)}_{\underline{d}'} \) implies \(p_{ij}(\underline{d}) = p_{ij}(\underline{d}')\) for all \(j\le i\). Suppose that it is not true, then without loss of generality there exists \(j\le i\) such that \(p_{ij}(\underline{d}) > p_{ij}(\underline{d}')\) and for any \(r\) such that \(j<r\le i\) we have \(p_{ir}(\underline{d}) = p_{ir}(\underline{d}')\). Since \(p_{ij}(\underline{d}) > p_{ij}(\underline{d}')\ge -\mathtt{y}_j\) we have \(d_{ij}(\underline{d})>0\) and \(p_{ij}(\underline{d})\in Z^{(i)}_{\underline{d}}\). So, since \( Z^{(i)}_{\underline{d}} = Z^{(i)}_{\underline{d}'} \)  there is \(l<j\) such that \(p_{il}(\underline{d}')=p_{ij}(\underline{d})\). Then \(\alpha_{l,j}\in\widehat{\Phi}_{+,\Lambda}\) and since \(\underline{d}'\in\widehat{\mathtt{GT}}_{\text{perm}}(\Lambda_{\text{dom}})\) by \eqref{eq:IneqOnPij} we get
        \begin{equation}
            p_{il}(\underline{d}')<p_{ij}(\underline{d}')<p_{ij}(\underline{d}).
        \end{equation}
        So, we have a contradiction.
    \end{proof}

    \begin{theorem}\label{th:PIrr}
        For  \(\kappa \notin \{0,n\}\)  the module \(\mathbb{L}_{\Lambda_{\text{dom}},u}\) is an irreducible \(\widehat{Y}_\circ(1,-\kappa)\)-module.    
    \end{theorem}
    To prove the theorem we need one more observation. Consider an undirected graph \(\Gamma(\Lambda_{\text{dom}})\) with vertices \(\underline{d}\in\widehat{\mathtt{GT}}_{\text{perm}}(\Lambda_{\text{dom}})\) and edges \((\underline{d},\underline{d}_{ij}^\pm)\), where $\underline{d},\underline{d}_{ij}^\pm\in \widehat{\mathtt{GT}}_{\text{perm}}(\Lambda_{\text{dom}})$ and $j\leq i$.
    \begin{lemma}\label{th:connectedgraph}
        For dominant weight \(\Lambda_{\text{dom}}\neq 0\), graph \(\Gamma(\Lambda_{\text{dom}})\) is connected.
    \end{lemma}
    \begin{proof}
    Let \( L(\underline{d}) \) be given by
        \begin{equation}\label{eq:lengthd}
            L(\underline{d}) = \max(i - j| d_{i,j}>0).
        \end{equation}
    It is the index of the last nonzero diagonal of the pattern \(\underline{d}\in\widehat{\mathtt{GT}}\). %\zh{Was "deepest diagonal" defined?} \s{is the word "last" better? I gave the formula right after.}      
     Let us prove that there exists \(j\in\mathbb{Z}\) such that \(\underline{d}_{L+j,j}^-\in \widehat{\mathtt{GT}}_{\text{perm}}(\Lambda_{\text{dom}})\) in \(3\) steps.
    
    \textbf{Step \(1\).} A number \(d_{L+j,j}\) participates in \(4\) types of inequalities
    \begin{align}\label{ineq:1}
                d_{L+j,j}&\le d_{L+j-1,j},\\ 
                \label{ineq:2} 
                d_{L+j,j}&\ge d_{L+j+1,j},\\
                \label{ineq:3}
                d_{L+j,j}&\le d_{L+j+1,r}+\mathtt{y}_j -\mathtt{y}_r-1, \quad \text{for} \quad \alpha_{j,r}\in\widehat{\Phi}_{+,\Lambda_{\text{dom}}},\\ 
                \label{ineq:4}
                d_{L+j,j}&\ge d_{L+j-1,l}-(\mathtt{y}_l -\mathtt{y}_j-1), \quad \text{for} \quad \alpha_{l,j}\in\widehat{\Phi}_{+,\Lambda_{\text{dom}}}.
    \end{align}
    Note that inequalities \eqref{ineq:1}, \eqref{ineq:3} hold for \(\underline{d}_{L+j,j}^-\) for any \(j\). Inequality \eqref{ineq:2} holds for \(d_{L+j,j}\) iff \(d_{L+j,j}>0\). It remains to analyze inequality \eqref{ineq:4}. It holds for \(\underline{d}_{L+j,j}^-\) if \(d_{L+j,j}>0\) and \(l<j-1\). Indeed, in this case \(d_{L+j-1,l}= 0\) by the definition of \(L\),  and \(\mathtt{y}_l -\mathtt{y}_j-1\ge 0\).

    \textbf{Step \(2\).} Consider any \(j\in\mathbb{Z}\) such that \(d_{L+j,j}>0\). If \(\alpha_{j-1,j}\not\in \widehat{\Phi}_{+,\Lambda_{\text{dom}}}\) or \begin{equation}
        d_{L+j,j} > d_{L+j-1,j-1}-(\mathtt{y}_{j-1} -\mathtt{y}_j-1)
    \end{equation}
    then we are done. Otherwise 
    \begin{equation}
        d_{L+j-1,j-1} = d_{L+j,j} +  (\mathtt{y}_{j-1} -\mathtt{y}_j-1) > 0.
    \end{equation} 
    In this case,  we replace \(j\) by \(j-1\) and repeat the argument.  Thus, after finitely many steps, either  we find \(j\) such that \(\underline{d}_{L+j,j}^-\in \widehat{\mathtt{GT}}_{\text{perm}}(\Lambda_{\text{dom}})\), or   \(\alpha_{j-1,j}\in\widehat{\Phi}_{+,\Lambda_{\text{dom}}}\) for any \(j\in\mathbb{Z}\).

    \textbf{Step \(3\).} If \(\alpha_{j-1,j}\in\widehat{\Phi}_{+,\Lambda_{\text{dom}}}\) for all \(j\in\mathbb{Z}\) then \(\Lambda_{\text{dom}} = \sum_{j=0}^{n-1} (\mathtt{y}_{j-1} -\mathtt{y}_j-1) \omega_j\) is integrable, and         
    \begin{equation}
            d_{L+j- n,j-n} = d_{L+j,j}+\sum_{s=j-n}^{j-1}(\mathtt{y}_{s} -\mathtt{y}_{s+1}-1) = d_{L+j,j} + \mathtt{y}_{j-n} -\mathtt{y}_{j}-n. 
    \end{equation}
    %Using the equality \(\mathtt{y}_{j-n} -\mathtt{y}_{j} = \kappa\) together with  Periodicity Condition \ref{def:dperidicity} for affine Gelfand-Tsetlin pattern, we obtain \(k = 0\). 
    On the other hand \( d_{L+j- n,j-n}  =  d_{L+j,j} \). Since \(\mathtt{y}_{s} -\mathtt{y}_{s+1}-1\ge 0\) we get  \(\Lambda_{\text{dom}} = 0\).%, in particular \(k = 0\). 
    %Hence  \( \widehat{\mathtt{GT}}_{\text{perm}}(\Lambda_{\text{dom}}) = \{\underline{0}\}\).
    \end{proof}

    \begin{proof}[Proof of Theorem \ref{th:PIrr}]
    Since \(\mathbb{L}_{\Lambda_{\text{dom}},u}\) is thin,  any nonzero submodule \(M\subset \mathbb{L}_{\Lambda_{\text{dom}},u}\) contains \(\xi_{\underline{d}}\) for some pattern \(\underline{d}\in\widehat{\mathtt{GT}}_{\text{perm}}(\Lambda_{\text{dom}})\). 

    We claim that for any other \(\underline{d}'\in\widehat{\mathtt{GT}}_{\text{perm}}(\Lambda_{\text{dom}})\), the vector \(\xi_{\underline{d}'}\) lies in \(M\). Indeed, by Lemma \ref{th:connectedgraph} there is a path \(\gamma\) in \(\Gamma(\Lambda_{\text{dom}})\) connecting \(\underline{d}\) and \(\underline{d}'\). By Lemma \ref{th:thin} and Corollary \ref{th:WellDefMatEl}, for each edge \((\underline{d},\underline{d}^{\pm}_{ij})\) there is an element \(h\in \widehat{Y}^0(1,-\kappa)\) such that for  \(X_{(\underline{d},\underline{d}^{\pm}_{ij})} = h \mathbf{x}^{\mp}_{i,0}\) we have
    \begin{equation}
        X_{(\underline{d},\underline{d}^{\pm}_{ij})} \xi_{\underline{d}}= \xi_{\underline{d}^{\pm}_{ij}}.
    \end{equation}
    Iterating along the path \(\gamma\), we obtain 
    \begin{equation}
        \prod_{e} X_{e} \xi_{\underline{d}}= \xi_{\underline{d}'},
    \end{equation}
   where the product is taken over the edges of \(\gamma\) in order.
    \end{proof}
    \begin{theorem}\label{th:Main}
        For  \(\kappa \neq n\) there is a structure of \(\widehat{Y}(1,-\kappa)\)-module on \(\mathbb{L}_{\Lambda_{\text{dom}},u}\) defined by~\eqref{eq:ActionInGTBasisYangian}. Moreover, there is an isomorphism of \(\widehat{Y}(1,-\kappa)\)-modules 
%        \begin{equation}
            \(\mathbb{L}_{\Lambda_{\text{dom}},u}\simeq \operatorname{ev}_{-}^{*}\mathcal{L}_{\Lambda,u}.\)
 %       \end{equation}
    \end{theorem}
    \begin{proof}
        By  Theorem \ref{th:PIrr} and  Corollary \ref{th:irrquad=irrwithserre} we have 
        \begin{equation}
            \mathbb{L}_{\Lambda_{\text{dom}},u}\simeq \mathtt{L}^\circ_{\vec{\mathbf{Q}}}\simeq \mathtt{L}_{\vec{\mathbf{Q}}}.
        \end{equation}
        So, Serre relations are satisfied. Moreover, since  the highest weights of \(\mathbb{L}_{\Lambda_{\text{dom}},u}\) and \(\operatorname{ev}_{-}^{*}\mathcal{L}_{\Lambda,u}\) coincide and they are both irreducible we obtain the desired isomorphism.
        %\(             \mathbb{L}_{\Lambda_{\text{dom}},u} =  \operatorname{ev}_{-}^{*}\mathcal{L}_{\Lambda,u}\).
    \end{proof}
%\zh{Here we go for the third time.}

%\zh{Main Theorem.   For  \(\kappa \neq n\) there is a structure %of \(\widehat{Y}(1,-\kappa)\)-module on $\Bbb{C}[x]$. Moreover, there is an isomorphism of \(\widehat{Y}(1,-\kappa)\)-modules 
%        \begin{equation}
%            \(\Bbb{C}[x]\simeq \operatorname{ev}_{-}^{*}\mathcal{L}_{\Lambda,u}.\)}
%\zh{\begin{proof}
%    Fix any isomorphism of vector spaces $\Bbb{C}[x]\to \operatorname{ev}_{-}^{*}\mathcal{L}_{\Lambda,u}$. Then the structure of Yangian module on $\Bbb{C}[x]$ is given by the pullback. The theorem is proved.
%\end{proof}
%}
    
    \begin{theorem}\label{th:PglhatIrr}
  There is a well-defined action of the algebra \(\mathtt{Heis}\) on \(\mathbb{L}_{\Lambda_{\text{dom}},u}\) commuting with \(\widehat{\mathfrak{sl}}_n\).
  This action is defined by formulas for \(a_m\) in \cite[Thm.4.18]{kodera2019braid}. The resulting \(\widehat{\mathfrak{gl}}_n\)-module is irreducible and isomorphic to \(\mathcal{L}_{\Lambda,u}\).
    \end{theorem}
    \begin{proof}
        By Theorem \ref{th:Main} we have an action of \(\widehat{\mathfrak{gl}}_n\) such that the Yangian action is obtained from it by evaluation homomorphism. So, by \eqref{eq:glnthroughsln} there is an action of \(\mathtt{Heis}\) commuting with \(\widehat{\mathfrak{sl}}_n\).

         The generators of \(\mathtt{Heis}\) are described in \cite[Thm.4.18]{kodera2019braid}  as images under the evaluation homomorphism of certain elements \(\widehat{Y}(1,-\kappa)\).
    \end{proof}
    \subsection{Characters of dominant representations}
    
        Let \(V\) be a representation of \(\widetilde{\mathfrak{sl}}_n\). Assume that  \(V\) has weight decomposition \(V = \bigoplus_{\nu\in\widetilde{\mathfrak{h}}^*} V_{\nu}\). For all \(h\in\widetilde{\mathfrak{h}}\) and \(v\in V_{\nu}\) we have \( h v = \nu(h) v\). Assume that \(\dim V_{\nu}<\infty\) for all \(\nu\). Then the character of \(V\) is 
    \begin{equation}
        \chi(V) = \sum_{\nu\in\widetilde{\mathfrak{h}}^*} \re^{\nu} \dim V_{\nu}. %\in \overline{\text{Span}}_{\mathbb{Z}}(\re^\nu)_{\nu\in\widetilde{\mathfrak{h}}^*},
    \end{equation}
    The character \(\chi(V)\) is a possibly infinite formal linear combinations of \(\re^{\nu}, \nu\in\widetilde{\mathfrak{h}}^*\). Note that 
    \begin{equation}
        \chi(\mathcal{M}_{\Lambda}) = \re^{\Lambda} \prod_{m=1}^{\infty}(1 - \re^{-m\delta})^{1-n}\prod_{\alpha\in\widehat{\Phi}^+_{\rm re}}(1 - \re^{- \alpha})^{-1}.        
    \end{equation}
    We also have
    \begin{align}\label{eq:Mlambdau}
         \chi(\mathcal{M}_{\Lambda,u}) &=  \chi(\mathcal{M}_\Lambda)\prod_{m=1}^{\infty}(1 - \re^{-m\delta})^{-1}.
        % \\
%         \label{eq:Llambdau}       \chi(\mathcal{L}_{\Lambda,u}) &=  \chi(\mathcal{L}_\Lambda)\prod_{m=1}^{\infty}(1 - \re^{-m\delta})^{-1}.       
    \end{align}
    For \(z\in\mathbb{C}\) we say that \(z> 0\) if either \(\Re(z) >0\) or \(\Re(z) =0\) and  \(\Im(z) =0\).
    \begin{theorem}[{\cite{kac1988modular}[Thm. 1]}]
        For a dominant weight \(\Lambda_{\text{dom}}\) such that \((\Lambda_{\text{dom}}+\rho,\delta)>0\) we have
        \begin{equation}\label{eq:characterofadmissible}
            \chi(\mathcal{L}_{\Lambda_{\text{dom}},u}) = \chi(\mathcal{M}_{0,u})\sum_{w \in \widehat{W}_{\Lambda_{\text{dom}}}} (-1)^{l(w)}\re^{w(\Lambda_{\text{dom}} + \rho) - \rho}.
        \end{equation}
        \qed
    \end{theorem}
    Consider the generating function of \(\widehat{\mathtt{GT}}_{\text{perm}}(\Lambda_{\text{dom}})\) defined by 
    \begin{equation}\label{eq:GenFunGTlambda}
        P(\Lambda_{\text{dom}}) = \re^{\Lambda_{\text{dom}}}\sum_{\underline{d}\in \widehat{\mathtt{GT}}_{\text{perm}}(\Lambda_{\text{dom}})}\re^{-\sum_{i=1}^n d_i(\underline{d})\alpha_i}.
    \end{equation}
    \begin{theorem}\label{th:char}
        For dominant weight \(\Lambda_{\text{dom}}\) such that \((\Lambda_{\text{dom}}+\rho,\delta)>0\) and \((\Lambda_{\text{dom}}+\rho,\delta)\neq n\) we have
            \begin{equation}
                  P(\Lambda_{\text{dom}})= \chi(\mathbb{L}_{\Lambda_{\text{dom}},u}) =\chi(\mathcal{M}_{0,u})\sum_{w \in \widehat{W}_{\Lambda_{\text{dom}}}} (-1)^{l(w)}\re^{w(\Lambda_{\text{dom}} + \rho) - \rho}.
            \end{equation}          
    \end{theorem}
    \begin{proof}
        The theorem follows from Theorem \ref{th:PglhatIrr}.
    \end{proof}
\subsection{Gelfand-Tsetlin patterns for \(\mathfrak{gl}_n\)}\label{sec:finite}
        Let \(\mu =\sum_{i=1}^n\mu_i \epsilon_i\) be a \(\mathfrak{gl}_n\)-weight. We denote by \(\mathbb{M}_{\mu}\) Verma module over \(\mathfrak{gl}_n\)  and by \(\mathbb{L}_{\mu}\) its irreducible quotient. Denote \(\mathtt{y}_{i} = \mu_i - i - \frac{1}{2}\) for \(i\in\{1,\dots,n\}\). We call \(\mu\) dominant \(\mathfrak{gl}_n\)-weight  if \(\mathtt{y}_l-\mathtt{y}_r\not\in \mathbb{Z}_{\le0}\) for any \(1\le l<r\le n\).

        Let a finite Gelfand-Tsetlin pattern  \(\underline{d}\) be \(\{d_{i,j}\in\mathbb{Z}_{\ge0}\}_{1\le j\le i\le n-1}\) satisfying inequality \(d_{i,j}\ge d_{i+1,j}\) for \(1\le j\le i \le n-2\) . Denote by \(\mathtt{GT}\) the set of finite Gelfand-Tsetlin patterns.  There is an injection \(\iota:\mathtt{GT}\rightarrow \widehat{\mathtt{GT}}\) 
        % defined by formula
        % \begin{equation}
        %     \iota(\underline{d})_{r,l} = 
        % \end{equation}        
        that does not change  \(d_{i,j}\) for \(1\le j\le i\le n-1\), sets \(d_{i,j}= 0\) for \(i\ge n, j\in\{1,\dots,n\}\) and then continues by periodicity to all \(j\le i\in\mathbb{Z}\). Abusing notation we denote \(\iota(\underline{d})\) by \(\underline{d}\).

        Similarly to affine case we call \(\underline{d}\in\mathtt{GT}\) permitted with respect to dominant \(\mathfrak{gl}_n\)-weight \(\mu_{\text{dom}}\) if for any \(1\le l<r\le n\) such that \(\mathtt{y}_l-\mathtt{y}_r\in \mathbb{Z}\) we have %inequality \ref{eq:DashedIneq}
        \begin{equation}
            d_{i,l}\le d_{i+1,r} + \mathtt{y}_l -\mathtt{y}_r-1. 
        \end{equation}
        Denote by \(\mathtt{GT}_{\text{perm}}(\mu_{\text{dom}})\) the set of all such patterns.

        The following known theorem follows from the results of Section \ref{sec:permitted}.
        \begin{theorem}[{{\cite[Prop.5.9]{futorny2019combinatorial}, \cite{Popov}}}]
         There is a basis \(\{\xi_{\underline{d}}\}_{\mathtt{GT}_{\text{perm}}(\mu_{\text{dom}})}\)  in \(\mathbb{L}_{\mu_{\text{dom}}}\)  with action of \(e_i,f_i,h_i\) for \(i\in\{1,\dots,n-1\}\) defined by formulas \eqref{eq:ActionInGTBasis}. 
        \end{theorem}
        \begin{proof}
            Consider \(\mathbb{L}_{\Lambda_{\text{dom}}, u}\) with  
            \begin{equation}
                \Lambda_{\text{dom}} = k\omega_0+ \sum_{i=1}^{n-1}(\mu_i-\mu_{i+1})(\omega_i-\omega_0) \in \omega_0^\perp, \quad \quad u = \sum_i \mu_i,
            \end{equation} 
            where \(k\not\in\mathbb{Q}\) . Note that \(\Lambda_{\text{dom}}\) is dominant %\(\widehat{\mathfrak{sl}}_n\)-
            weight. Consider an action of \(\mathfrak{sl}_n\subset \widehat{\mathfrak{sl}}_n\) generated by \(e_i,f_i,h_i\) for \(i\in\{1,\dots,n-1\}\) on \(\mathbb{L}_{\Lambda_{\text{dom}}, u}\). %By the argument similar to the proof of the Theorem \ref{th:PIrr}  we have
            Note that \(\operatorname{Span}_{\mathbb{C}}(\xi_{\underline{d}})_{\underline{d}\in\mathtt{GT}_{\text{perm}}(\mu_{\text{dom}})}\) is a subspace of \(\mathbb{L}_{\Lambda_{\text{dom}}, u}\) with  action of \(D\) by \(0\). 
            % \begin{equation}
            %     \mathcal{U}(\mathfrak{sl}_n)\cdot \xi_{\underline{0}} = \operatorname{Span}_{\mathbb{C}}(\xi_{\underline{d}})_{\underline{d}\in\mathtt{GT}_{\text{perm}}(\mu_{\text{dom}})}.
            % \end{equation}
            By Theorem \ref{th:PglhatIrr} it is clear that \(\operatorname{Span}_{\mathbb{C}}(\xi_{\underline{d}})_{\underline{d}\in\mathtt{GT}_{\text{perm}}(\mu_{\text{dom}})}\) is an irreducible \(\mathfrak{sl}_n\)-module with highest weight \(\mu\).
        \end{proof}

\section{Admissible representations of \(\widehat{\mathfrak{sl}}_n\)} \label{sec:admissible}

    \subsection{Admissible weights of \(\widehat{\mathfrak{sl}}_n\).}

    Let \( \widehat{Q} = \mathrm{Span}_{\mathbb{Z}}\{\alpha_0,\alpha_1,\dots,\alpha_{n-1}\}\)
   be the \textit{affine root lattice}. For a dominant weight \(\Lambda\), denote \( \widehat{Q}_{\Lambda} =\operatorname{Span}_{\mathbb{Z}} \widehat{\Phi}_{\Lambda} \).
    Following \cite{kac1989classification}, we say that a dominant \(\Lambda\) is \textit{admissible} if \(\widehat{Q}_\Lambda\) is a full-rank sublattice of \(\widehat{Q}\).

    \begin{theorem}[{\cite{kac1989classification}[Thm 2.1]}]
        Suppose \(\Lambda\) is admissible. Then the corresponding \(\kappa\) can be written as \(\kappa = \frac{p}{p'}\), where \(p\) and \(p'\) are coprime positive integers such that \(p \geq n\).\qed
    \end{theorem}
    
    Fix \(p, p'\) as above. Define the operator \(A_{p'}\) acting on \(\tilde{\mathfrak{h}}^* = \mathbb{C}\omega_0 \oplus \mathfrak{h}^* \oplus \mathbb{C}\delta\) diagonally by \((\frac{1}{p'}, \mathrm{Id}_{\mathfrak{h}^*}, p')\). Denote by 
    \begin{equation}
        \alpha_i' =  A_{p'}\alpha_i \quad \text{for} \quad i\in\mathbb{Z}
    \end{equation} 
    and
    \begin{equation}\label{eq:alpha'}
        \widehat{\Delta}_{p'} = A_{p'} ~ \widehat{\Delta} = \{\alpha_0', \dots,  \alpha_{n-1}' \}.
    \end{equation}
    \begin{theorem}[{\cite{kac1989classification}[Thm 2.1]}]\label{th:KacAdm}
        For fixed \(p, p'\), all admissible weights of level \(k =  \frac{p}{p'} -n \) are parametrized by pairs \((\eta, \gamma)\), where \(\eta\) is a weight
        \begin{equation}\label{eq:Defai}
                \eta = \sum_{i=0}^{n-1} a_i \omega_i,
        \end{equation}
        with \(a_i \in \mathbb{Z}_{\ge 0}\) such that \(\sum_{i}(a_i+1) = p\), and \(\gamma \in \widehat{W}^e\) satisfying \(\gamma ~\widehat{\Delta}_{p'} \subset \widehat{\Phi}_+\).  
        Then the weight \(\Lambda\) is given by
        \begin{equation}\label{eq:LambdaThroughLambda0}
            \Lambda = \gamma \left(\eta - \omega_0 (p'-1)\kappa + \rho\right) - \rho - \mathtt{const}_1 \cdot \delta,
        \end{equation}
        where \(\mathtt{const}_1\in\mathbb{Q}\) is uniquely determined by the condition \(\Lambda \in \omega_0^{\perp}\).\qed
    \end{theorem}
    \begin{remark}[{\cite[Prop. 2.1]{kac1989classification}}]\label{th:GammaEtaUnique}
        A pair \((\gamma,\eta)\) is not uniquely determined by the weight \(\Lambda\). In fact, two pairs \((\gamma_1,\eta_1)\) and \((\gamma_2,\eta_2)\) correspond to the same \(\Lambda\) if and only if there exists \(\mathtt{r}\in\widehat{W}^e\) such that \( \mathtt{r}(\widehat{\Delta}_{p'}) = \widehat{\Delta}_{p'}\) and
        \begin{equation}
        \begin{aligned}
            \gamma_2 &= \gamma_1\circ \mathtt{r} ,
            \\
            \eta_2+\rho- \omega_0 (p'-1)\kappa  &= \mathtt{r}^{-1}(\eta_1 +\rho- \omega_0(p'-1)\kappa).
        \end{aligned}
        \end{equation}
        Furthermore, it follows from \( \mathtt{r}(\widehat{\Delta}_{p'}) = \widehat{\Delta}_{p'}\) that \(\mathtt{r}\) should be equal to \(\mathtt{r}_s\) for some  \(s\in\{0,\dots n-1\}\)  where 
        \begin{equation}
            \mathtt{r}_s(\alpha'_i) = \alpha'_{i+s}, \quad \text{for} \quad i\in\{0,\dots\zh{,} n-1\}.
        \end{equation}
    \end{remark}

    Let \(A_{p',\gamma}= \gamma A_{p'}\) and 
    \begin{equation}\label{eq:Deltap'gamma}
        \widehat{\Delta}_{p',\gamma} = \gamma \widehat{\Delta}_{p'}.
    \end{equation}
    %Denote by 
    % \begin{equation}\label{eq:alpha''}
    %     \alpha''_i = \gamma\alpha'_i = A_{p',\gamma}\alpha_i \quad \text{for}\quad i\in\{0,\dots,n-1\}.
    % \end{equation}
    
    \begin{proposition}\label{th:WLambdaThroughW}
    We have
    \begin{equation}
        \widehat{W}_\Lambda = A_{p',\gamma} \widehat{W} A_{p',\gamma}^{-1}.
    \end{equation}
    \end{proposition}
    
    \begin{proof}
    First, we see that
    \begin{equation}
        \widehat{\Phi}_\Lambda = \gamma \cdot \widehat{\Phi}_{ - \omega_0 (p'-1)\kappa},      
    \end{equation}
    and since \(s_{\gamma(\alpha)} = \gamma s_{\alpha} \gamma^{-1}\), we obtain
    \begin{equation}
        \widehat{W}_\Lambda = \gamma \cdot \widehat{W}_{ - \omega_0 (p'-1)\kappa} \cdot \gamma^{-1}.
    \end{equation}
     For any \(\alpha \in \Phi,\, m \in \mathbb{Z}\), we have the scalar product \(
        (\alpha + m\delta,\,  \omega_0 (p'-1)\kappa ) = (p'-1)m \kappa \). This number is integer if and only if \(p'|m\). Thus, 
    \begin{equation}
        \widehat{\Phi}_{ - \omega_0 (p'-1)\kappa} =
        \bigl\{\alpha + mp'\delta \mid \alpha \in \Phi,\, m \in \mathbb{Z}\bigr\}
        \cup
        \bigl\{mp'\delta \mid m \in \mathbb{Z} \setminus \{0\}\bigr\}.
    \end{equation}

    The operator \(A_{p'}\) maps \(\widehat\Phi\) bijectively to \(\widehat\Phi_{\eta - \omega_0 (p'-1)\kappa}\). Since \(s_{A_{p'}\alpha} = A_{p'} s_{\alpha} (A_{p'})^{-1}\), we conclude that
    \begin{equation}
        \widehat{W}_{ - \omega_0 (p'-1)\kappa} = A_{p'} \widehat{W} (A_{p'})^{-1}.        
    \end{equation}
    \end{proof}
    \begin{corollary}
            Reflections  with respect to \(\widehat{\Delta}_{p',\gamma} = \gamma\widehat{\Delta}_{p'}\) generate \(\widehat{W}_\Lambda\).
    \end{corollary}

    \subsection{Parametrization of \(\gamma\)}
    Consider a strictly increasing sequence \(\vec{g} =(g_0,\dots,g_{n-1},g_{n})\in\mathbb{Z}^{n+1}\) such that \(g_{n} - g_{0} = p' n\) and there exists permutation \(\sigma_{\vec{g}}\in\mathfrak{S}_n\) such that
    \begin{equation}\label{eq:DefOfSigmag}
        g_i \equiv \sigma_{\vec{g}}(i) \quad \operatorname{mod} n, \quad \text{for} \quad i\in\{0,\dots, n-1\}.
    \end{equation}
    Let us continue sequence \(\{g_i\}\) for all \(i\in\mathbb{Z}\) by the rule \(g_{i+n} = g_i+p'n\).       
    \begin{proposition}
        There is a unique  \(\gamma_{\vec{g}}\in\widehat{W}^e\) such that 
        \begin{equation}\label{eq:Gammag}
            \gamma_{\vec{g}}(\alpha_i') = \alpha_{g_i,g_{i+1}}
            , \quad i \in \{0,1,\dots,n-1\},
        \end{equation}
        where \(\alpha_i'\) are defined by \eqref{eq:alpha'}.
    \end{proposition}

    \begin{proof}
    \textbf{Existence.}    
        Let us show that
        \begin{equation}
            \sigma_{\vec{g}} \circ  t_{V}(\alpha_i') = \alpha_{g_i,g_{i+1}}, \quad i \in \{0,1,\dots,n-1\},
        \end{equation}
        where
        \begin{equation}
            V = \sum_{i=1}^{n-1} \left(\Big\lceil \frac{g_{i+1}}{n} \Big\rceil -\Big\lceil \frac{g_{i}}{n} \Big\rceil\right)(\omega_i - \omega_0).
        \end{equation}
        By \eqref{eq:Paction} and \eqref{eq:gammag-delta} we have
        \begin{equation}
            \sigma_{\vec{g}} \circ  t_{V}(\alpha_i) = \epsilon_{\sigma_{\vec{g}}(i)}- \epsilon_{\sigma_{\vec{g}}(i+1)}+ \left(\Big\lceil \frac{g_{i+1}}{n} \Big\rceil -\Big\lceil \frac{g_{i}}{n} \Big\rceil\right)\delta = \alpha_{g_i,g_{i+1}}, \quad \text{for } i\in\{1,\dots, n-1\}.
        \end{equation} 
        Note that
            \begin{equation}
                 \sigma_{\vec{g}} \circ  t_{V}\Big(\sum_{i=0}^{n-1}\alpha'_i\Big) = \sigma_{\vec{g}} \circ  t_{V}(p'\delta) =p'\delta=\sum_{j=g_0}^{g_n}\alpha_j = \sum_{i=0}^{n-1}\alpha_{g_i,g_{i+1}} .
            \end{equation}
        So, we have
        \begin{equation}
            \sigma_{\vec{g}} \circ  t_{V}(\alpha_0') = \alpha_{g_0,g_1}.
        \end{equation}
             \textbf{Uniqueness.} For two such transformations \(\gamma,\gamma'\) it is easy to check that \(\gamma^{-1}\circ \gamma'\) is identity.\!
    \end{proof}

    Denote by \(\tau\in \operatorname{Aut}(\Phi)\)  an involution defined by %formulas
    \begin{equation}\label{eq:tau}
          \tau(\delta) =\delta, \quad \tau(\omega_0) = \omega_0, \quad \tau(\alpha_i) = \alpha_{n-i}.
    \end{equation}
    Note that \(\tau\not \in \widehat{W}^e\) for \(n\ge 3\) because \(\tau\) performs an external automorphism of \(\Phi\).
    \begin{proposition}
        For any element \(\gamma \in \widehat{W}^e\) satisfying \(\gamma \widehat{\Delta}_{p'} \subset \widehat{\Phi}_+\) there exists \(\vec{g}\) such that
        \begin{equation}
            \gamma = \gamma_{\vec{g}}.
        \end{equation}     
    \end{proposition}
    \begin{proof}
         Since  \(\gamma \widehat{\Delta}_{p'} \subset \widehat{\Phi}_+\) by the Proposition \ref{th:RepOfRoot}  we have
         \begin{equation}
            \gamma(\alpha_i') = \alpha_{l_i,r_i} % \sum_{s=l_i}^{r_i-1}\alpha_s,
        \end{equation} 
        for some \(l_i< r_i\in\mathbb{Z}\) for 
        \(i\in\{0,\dots n-1\}\). Extend \(l_i, r_i\)  by periodicity: \(l_i = l_{i+n}, r_i = r_{i+n}\) to all \(i \in \mathbb{Z}\). Since \(\gamma(\delta) = \delta\) we have
        \begin{equation}
            \sum_{i=1}^{n} (r_{i}-l_i) = p'n.
        \end{equation}
        It is clear that 
        \begin{equation}
           (\alpha_{l_i,r_i},\alpha_{l_j,r_j})%\Big(\sum_{s=l_i}^{r_i-1}\alpha_s,\sum_{t=l_j}^{r_j-1}\alpha_t\Big) 
            = \delta_{l_i,l_j}^n +\delta_{r_i,r_j}^n - \delta_{r_i,l_j}^n - \delta_{l_i,r_j}^n.
        \end{equation}
        On the other hand  \((\gamma(\alpha_i'),\gamma(\alpha_j'))=C_{ij}\) since \(\gamma\in \operatorname{Aut}(\widehat{\Phi})\). % As before we can think about \(\mathtt{P}_{\mathfrak{h}^*}\gamma(\alpha_i')\)  as about arrow between vertices from \(l_i\!\mod{n}\) to \(r_i\!\mod{n}\). 
        Then
        \begin{enumerate}
            \item  \(
                \{l_i\operatorname{mod} n \}_{i=0}^{n-1}\equiv\{0,\dots,n-1\} \equiv \{r_i\operatorname{mod} n \}_{i=0}^{n-1}\),
            \item   one of the following conditions holds    
                \begin{enumerate}
                \item\label{case1}
                        \(r_i \equiv l_{i+1} \operatorname{mod} n\) for \(i\in\{0,\dots, n-1\}\) ,
                \item\label{case2}                     
                        \(l_i \equiv_{n} r_{i+1}\operatorname{mod} n\) for \(i\in\{0,\dots, n-1\}\).   
            \end{enumerate}
        \end{enumerate} 
        In the case \ref{case1} we choose \(r_i,l_i\) by such a way that  \(r_i = l_{i+1}\) and we  denote \(g_{i} = l_{i}\) for \(i\in\{0,\dots,n-1\}\) and \(g_{n} = r_n\). So, 
        \begin{equation}
            \gamma = \gamma_{\vec{g}}.
        \end{equation}
        For \(n = 2\) conditions \ref{case1} and \ref{case2} are equivalent. Let us prove, that the case \ref{case2} cannot be realized if \(n>2\). Let us show that \(\gamma\circ\tau\in \widehat{W}^e\). Indeed,
        \begin{equation}
            \gamma\circ\tau(\alpha_i') = \alpha_{l'_i,r'_i}%\sum_{s=l'_{i}}^{r_{i}'-1}\alpha_s 
            , \quad \text{where}  \quad l_{i}' = l_{n-i}\quad \text{   and   } \quad r'_{i} = r_{n-i}.
        \end{equation} 
        So, \(r_i' \equiv_{n} l_{i+1}'\) for \(i\in\{0,\dots, n-1\}\). Define \(g_{i}' = l_{i}'\) for \(i\in\{0,\dots,n-1\}\) and \(g_{n}'= r_n'\). By case \ref{case1} 
        \begin{equation}
            \gamma\circ \tau= \gamma_{\vec{g}'}\in\widehat{W}^e,
        \end{equation}
        Since \(\tau\not\in \widehat{W}^e\) we see that \(\gamma\not\in \widehat{W}^e\).
    \end{proof}
    Introduce notation
    \begin{equation}
        b_i = g_{i+1} - g_i - 1, ~\text{for}~ i\in\{0,\dots n-1\}.
    \end{equation}
    On the other hand for given vector \(\vec{b}=(b_0,\dots,b_{n-1})\) let us denote sequence \(\{g_i(\vec{b})\}_{i\in\mathbb{Z}}\) satisfying the following conditions
    \begin{equation}
        \begin{cases}
            g_{i+1} - g_{i}-1 = b_{i \operatorname{mod} n}, \quad \text{for} \quad i\in\mathbb{Z},\\
            g_0 = 0.
        \end{cases}
    \end{equation}
    Let us reformulate  Theorem  \ref{th:KacAdm} in more explicit terms.
    \begin{theorem}\label{th:KacAdmRefined}
        The set of admissible weights on admissible level \(k =  \frac{p}{p'} -n\) is in a bijection with pairs of sequences \(\{a_i\}_{i=0}^{n-1}, \{b_i\}_{i=0}^{n-1}\)  of nonnegative integers such that 
        \begin{equation}
            \begin{cases}
                \sum_{i=0}^{n-1} a_i = p - n,\\
                \sum_{i=0}^{n-1} b_i = p'n - n,
            \end{cases}
        \end{equation}
        and there exists \(\sigma\in \mathfrak{S}_n\) such that 
        \begin{equation}
            \sigma(i) = g_i(\vec{b}) ~\operatorname{mod} ~ n.
        \end{equation}
        Moreover, any admissible weight has the form
        \begin{equation}
            \Lambda(\vec{a},\vec{b}) = \gamma_{\vec{g}(\vec{b})}\left(\sum_{i=0}^{n-1} a_i\omega_i-\omega_0\kappa(p'-1)+\rho\right)-\rho - \mathtt{const}_1 \cdot \delta 
        \end{equation}
        where \(\mathtt{const}_1\in\mathbb{Q}\) is uniquely determined by the condition \(\Lambda \in \omega_0^{\perp}\). 
    \end{theorem}
    \begin{proof}
        The theorem follows from  Theorem  \ref{th:KacAdm} and Remark \ref{th:GammaEtaUnique}. 
    \end{proof}
    For an admissible weight \(\Lambda\),  we encode the sequences \(\vec{a},\vec{b}\) defined in Theorem \ref{th:KacAdmRefined} as weights
    \begin{equation}\label{eq:etaxi}
        \eta = \sum_{i=0}^{n-1} a_{i}\omega_i, \quad \xi = \sum_{i=0}^{n-1} b_{i}\omega_i.       
    \end{equation}

    \subsection{Gelfand-Tsetlin Patterns for  Admissible Modules}

    \begin{proposition}\label{pr:admissible}
    Let \(\Lambda \neq 0\) be an admissible weight. Let \(\vec{a}, \vec{b}\) be as defined in Theorem~\ref{th:KacAdmRefined}, and let \(g_s = g_s(\vec{b})\). Let \(\underline{d} \in \widehat{\mathtt{GT}}\). Then \(\underline{d} \in \widehat{\mathtt{GT}}_{\mathrm{perm}}(\Lambda)\) if and only if
    \begin{equation}\label{eq:DashedIneqForPi}
        d_{i,g_s}\,{\color{blue}\xleftarrow{a_s}}\,d_{i+1,g_{s+1}},
        \quad \text{for } s \in \{0,\dots,n-1\}.
    \end{equation}
\end{proposition}

\begin{proof}
    First let us show that \eqref{eq:DashedIneqForPi} holds for any \(\underline{d} \in \widehat{\mathtt{GT}}_{\mathrm{perm}}(\Lambda)\). Note that
    \begin{equation}
        a_s = \mathtt{y}_{g_s} - \mathtt{y}_{g_{s+1}} - 1.
    \end{equation}
    Indeed, by  \eqref{eq:Gammag} and \eqref{eq:yi}  we have
    \begin{equation}
        (\gamma_{\vec{g}}(\alpha_s'), \Lambda + \rho) =  (\alpha_{g_s,g_{s+1}}, \Lambda + \rho)=\mathtt{y}_{g_s} - \mathtt{y}_{g_{s+1}}.
    \end{equation}
    On the other hand,
    \begin{equation}
        (\gamma_{\vec{g}}(\alpha_s'), \Lambda + \rho)
        = (\alpha_s', \gamma_{\vec{g}}^{-1}(\Lambda + \rho))
        = (\alpha_s', \eta - \omega_0 (p'-1)\kappa + \rho) =  a_s +1.
    \end{equation}
     In the last equality we use that \(\alpha_s' = \alpha_s\) for \(s \in \{1,\dots,n-1\}\) and \(\alpha_0' = \alpha_0+ (p'-1)\delta\). Applying inequalities \eqref{eq:DashedIneq} to the roots \(\alpha_{g_s,g_{s+1}}\in \widehat{\Phi}_{\Lambda}^+\) we obtain \eqref{eq:DashedIneqForPi}.

    Now assume \eqref{eq:DashedIneqForPi} for \(\underline{d}\), and prove that \(\underline{d}\in\widehat{\mathtt{GT}}_{\mathrm{perm}}(\Lambda)\). By \cite[formula (1.7b)]{kac1989classification} we have
    \begin{equation}
        \widehat{\Phi}_{\Lambda}^+ = \widehat{\Phi}_+ \cap \operatorname{Span}_{\mathbb{Z}}(\widehat{\Delta}_{p',\gamma}).        
    \end{equation}
    By \eqref{eq:Deltap'gamma} and \eqref{eq:Gammag} we  have 
    \begin{equation}\label{eq:Deltap'gammagi}
        \widehat{\Delta}_{p',\gamma} = \{\alpha_{g_i,g_{i+1}}\}_{i=0}^{n-1}.
    \end{equation}
    So, for any \(\beta \in \widehat{\Phi}_{\Lambda}^+\), we can write
    \begin{equation}
        \beta = \sum_{s=l}^{r-1} \alpha_{g_s,g_{s+1}} = \alpha_{g_l,g_r},
    \end{equation}
    where \(l < r\). From \eqref{eq:DashedIneqForPi} we obtain the chain of inequalities
    \begin{equation}
        p_{i,g_l} < p_{i+1,g_{l+1}} < \cdots < p_{i+r-l,g_r}.
    \end{equation}
    Since \(r - l > 0\), it follows that
    \begin{equation}
        p_{i,g_l} < p_{i+1,g_r}.
    \end{equation}
    This is exactly inequality \eqref{eq:DashedIneq} corresponding to the root \(\beta\).
\end{proof}

    Consider the re-indexed $n$-tuple of partitions 
    \((\tilde{\lambda}^{(1)}, \dots, \tilde{\lambda}^{(n)})\) defined by
    \begin{equation}
        \tilde{\lambda}^{(a)}_j 
        = \lambda^{(g_{n-a+1})}_j 
        = d_{g_{n-a+1}+j-1,\, g_{n-a+1}},
    \end{equation}
    where the partitions $\lambda^{(a)}$ are defined by \eqref{eq:nYoungDiagrams}. Note that inequalities \eqref{eq:DashedIneqForPi} take the form
    \begin{equation}
        \tilde{\lambda}^{(i)}_j 
        \ge 
        \tilde{\lambda}^{(i+1)}_{j + b_{n-i}} - a_{n-i}.
    \end{equation}
    Set $\tilde{a}_i = a_{n-i}$ and $\tilde{b}_i = b_{n-i}$. Then
    \begin{equation}\label{eq:tauetaxi}
        \tau \eta = \sum_{i=1}^n \tilde{a}_i \omega_i,
        \qquad
        \tau \xi = \sum_{i=1}^n \tilde{b}_i \omega_i.
    \end{equation}

    We call an \(n\)-tuple of Young diagrams \(\{\tilde{\lambda}^{(i)}\}_{i=1}^{n}\) a \textit{shifted cylindric plane partition} with parameters \(\tilde{a}_i, \tilde{b}_i\), \(i \in \{1,\dots,n\}\), if
    \begin{equation}\label{eq:abconditions}
        \tilde{\lambda}^{(i)}_j 
        \ge 
        \tilde{\lambda}^{(i+1)}_{j+\tilde{b}_i} - \tilde{a}_i,
        \quad \text{for all } i \in \{1,\dots,n\} \text{ and } j \ge 1,
    \end{equation}
    where the index \(i\) is taken modulo \(n\). These conditions appeared in work \cite{feigin2011quantum}. See also \cite{burge1993restricted} for the \(n=2\) case and \cite{bershtein2014agt,alkalaev2014conformal,belavin2015agt} for the appearance of these conditions in the AGT relation.
        
   % We call an $n$-tuple of Young diagrams \(\{\tilde{\lambda}^{(i)}\}_{i=1}^n\)  a \textit{shifted cylindric plane partition} with parameters \(\tilde{a}_i, \tilde{b}_i\), \(i \in \{0,\dots,n-1\}\), if
   %  \begin{equation}\label{eq:abconditions}
   %      \tilde{\lambda}^{(i)}_j 
   %      \ge 
   %      \tilde{\lambda}^{(i+1)}_{j+\tilde{b}_i} - \tilde{a}_i \quad \text{for} \quad i\in \{0,\dots,n-1\},
   %  \end{equation}
   %  where \(i\) is considered modulo \(n\). 

%    See examples of such partitions for \(\underline{0}\in \widehat{\mathtt{GT}}_{\text{perm}}(\Lambda)\) for different admissible weights \(\Lambda\) on the Figures \ref{fig:example}, \ref{fig:example2} and \ref{fig:example3}. The rows of Young diagrams \(\tilde \lambda^{(i)}\) grow up and columns grow right. Action of \(e_i\) (correspondingly \(f_i\)) gives a linear combination of all possible shifted cylindric plane partitions without one of the boxes (correspondingly with one added box)  of \(i\)-th color (together with all periodic to it).

    \begin{remark}
            The case of  \(\tilde{b}_i = 0\) for \(i\in\{0,\dots,n-1\}\) corresponds to integrable representation. In this case the  shifted cylindric %\zh{al}   \s{google thinks that cylindric is correct}
            plane partitions are cylindric 
            %\zh{al}
            plane partitions 
            introduced in \cite{gessel1997cylindric}. %according to definition from  \cite[p.15]{tingley2008three}. {\color{teal} references?} 
            % See the example on \ref{fig:example3}, note that directions of axes are chosen differently from \cite[p.18]{tingley2008three} .
    \end{remark}

    \begin{remark}
       The case $a_i=0$ corresponds to {\rm boundary} admissible modules of Kac-Wakimoto, studied in \cite{KacWakimoto2017Boundary}. Similarly to that of integrable modules, the principal specialization of characters of admissible modules is always expressed in terms of certain theta-functions (see the next section). For boundary admissible modules, this specialization has a factorizable form, as a product of $\vartheta_{11}$. 
    \end{remark}

%\zh{Maybe it is good to say that $a_i=0$ corresponds to boundary admissible modules which among other things have factorizable characters? }
    % \begin{remark}
    %     In \cite[Chapter 5]{feigin2013representations} the case corresponding to limit \(%p,p'
    %     a_0,b_0\rightarrow\infty \) with fixed \(a_i,b_i\) for \(i\in\{1,\dots, n-1\}\) is considered. 
    % \end{remark}

    \subsection{Principal specialization}    
        The main result of this section is Theorem \ref{th:prspec} that gives a formula for principal specialization of character of admissible module. Recall that principal specialization of monomial \(\re^{\nu}\) is \(q^{-(\nu,\rho)}\). We denote it by overline, hence we have 
        \begin{equation}
            \overline{\sum_{\nu} m_{\nu} \re^{\nu}} = \sum_{\nu} m_{\nu}q^{-(\nu,\rho)}.        
        \end{equation}
        \begin{theorem}\label{th:prspec}
            For \(\Lambda\) admissible and \(\eta,\xi\) defined \eqref{eq:etaxi}, we have
        \begin{equation}\label{eq:prspec}
            \overline{\chi(\mathcal{L}_{\Lambda, u})} = \frac{q^{-(\Lambda,\rho)}}{(q)_\infty^n} \sum_{w\in \widehat{W}} (-1)^{l(w)} q^{(\xi + \rho - w(\xi+\rho),\eta + \rho)},
        \end{equation}
        where \((q)_\infty = \prod_{m=1}^\infty (1 - q^m)\).
        \qed
        \end{theorem}
         %Recall that \(\mathcal{L}_{\Lambda, u}\) is defined by formula \eqref{eq:Llambdau}.
          %{\color{teal} Remark that if we replace \(\eta,\xi\) by \(\eta' =\tau(\eta)=\eta\) and \(\xi' =\tau(\xi) \) in the theorem the statement still will be true. The advantage of our choice of \(\eta,\xi\) will be clear in comparison with language of plane partitions in next chapter.}

        \begin{remark}
                Note that   if \((n,p) = 1\), then the right-hand side of \eqref{eq:prspec} coincides with the character of \(\mathcal{W}(\widehat{\mathfrak{gl}}_n)\)-module from the minimal \((p'n,p)\) theory (see \cite[Formula (4.7)]{feigin2011quantum},\cite[Theorem 3.2]{frenkel1992characters}, \cite{arakawa2004quantized}). 
        \end{remark}
        To prove the theorem we need a technical lemma.
        \begin{lemma}\label{th:AinvLambda}
            We have
            \begin{align}
                A_{p',\gamma}^{-1}(\Lambda + \rho) &= \eta + \rho-\frac{\mathtt{const}_1}{p'}\delta, \label{eq:AinvLambda1}\\ 
                A_{p',\gamma}^{-1}(\rho) &= \xi + \rho + \frac{\mathtt{const}_2}{p'} \delta, \label{eq:AinvLambda2}
            \end{align}  
            where \(\mathtt{const}_1\) is from \eqref{eq:LambdaThroughLambda0} and  \(\mathtt{const}_2\in\mathbb{Q}\).
        \end{lemma}
        \begin{proof}[Proof of Lemma \ref{th:AinvLambda}]
            \begin{enumerate}
                \item  By \eqref{eq:LambdaThroughLambda0} we have
                \begin{equation}
                    A_{p',\gamma}^{-1}(\Lambda+\rho) = (A_{p'})^{-1}  \left(\eta -  (p'-1)\kappa \omega_0 + \rho\right)  -\frac{\mathtt{const}_1}{p'} \delta.
                \end{equation}
                Note that 
                \begin{equation}
                    \eta -  (p'-1)\kappa \omega_0 + \rho \in \kappa \omega_0 + \mathfrak{h}^*, 
                \end{equation}
                 then
                 \begin{multline}
                    (A_{p'})^{-1}  \left(\eta - (p'-1)\kappa \omega_0  + \rho\right)
                    = \eta-(p'-1)\kappa \omega_0 + \rho + (p'-1)\kappa \omega_0 = \eta + \rho .                 
                 \end{multline}

                \item First, let us calculate \((\gamma_{\vec{g}})^{-1}\rho\). We have
                \begin{equation}
                    ((\gamma_{\vec{g}})^{-1}\rho,\alpha_i) = (\rho,\gamma_{\vec{g}} \alpha_i)   =   g_{i+1} - g_i = b_i+1,            
                \end{equation}
                for \(i\in\{1,\dots,n-1\}\) and  
                \begin{multline}                ((\gamma_{\vec{g}})^{-1}\rho,\alpha_0 ) = (\rho,\gamma_{\vec{g}} (\alpha'_0-(p'-1)\delta))  
                    \\
                    = 
                    g_{n+1} - g_n - n(p'-1) = b_n+1- n(p'-1).                 
                \end{multline}
                So, 
                \begin{equation}
                    (\gamma_{\vec{g}})^{-1}\rho = \xi + \rho - n(p'-1)\omega_0+\mathtt{const}_2\cdot\delta,
                \end{equation}
                for some \(\mathtt{const}_2\in \mathbb{Q}\). Remark that \(\rho \in n\omega_0 + \mathfrak{h}^*\). Hence
                \begin{equation}
                    (\gamma_{\vec{g}})^{-1}\rho \in n\omega_0  + \mathtt{const}_2\cdot\delta + \mathfrak{h}^*.
                \end{equation} 
                Then we see that
                \begin{multline}
                    (A_{p'})^{-1}(\gamma_{\vec{g}})^{-1}\rho = \xi + \rho -n (p'-1)\omega_0 + \mathtt{const}_2\cdot\delta 
                    \\
                    +n (p'-1)\omega_0 - \frac{p'-1}{p'}\mathtt{const}_2\cdot\delta = \xi + \rho + \frac{\mathtt{const}_2}{p'} \delta.                
                \end{multline}    
                    %(\gamma_{\vec{g}})^{-1}\rho + n(p'-1)\omega_0 - \frac{p'-1}{p'}\mathtt{const}_2\cdot\delta = \tau(\xi + \rho) + \frac{\mathtt{const}_2}{p'} \delta.    
            \end{enumerate}   
        \end{proof}               
        %What happens if \((p',p)=1\) but \((n,p)\neq 1\)? In \cite{frenkel1992characters} they do not cover such representations.     
        \begin{proof}[Proof of Theorem \ref{th:prspec}]
        We will compute principal specialization of \eqref{eq:characterofadmissible} in \(2\) steps.
        
        \textbf{Step 1.}
        Let us see that
        \begin{equation}
            \overline{\chi (\mathcal{M}_{0,u})} = \frac{1}{(q)_{\infty}^n}.
        \end{equation}
        Indeed,
        \begin{equation}
            \chi (\mathcal{M}_{0,u}) = \prod_{m>0}(1 - \re^{-m\delta})^{-n}  \prod_{1\le i<j\le n}\prod_{m>0}(1 - \re^{-m\delta+\epsilon_i-\epsilon_j})^{-1}(1 - \re^{(1-m)\delta-\epsilon_i+\epsilon_j})^{-1}.
        \end{equation}
        
        Since \((\delta,\rho) = n\) and \( (\epsilon_i - \epsilon_j,\rho) = j - i\), we obtain
        \begin{multline}
            \overline{\chi(\mathcal{M}_{0,u})} = \prod_{m>0}(1 - q^{mn})^{-n}  \prod_{1\le i<j\le n}\prod_{m>0}(1 - q^{mn + i - j})^{-1}(1 - q^{(m-1)n - i + j})^{-1} 
            \\
            = \prod_{s=0}^{n-1}\prod_{m>0}(1 - q^{mn+s})^{-n} = \frac{1}{(q)_{\infty}^n}.        
        \end{multline}
    
        \textbf{Step 2.}
        Consider
            \begin{equation}
                S = \overline{\sum_{w \in \widehat{W}_\Lambda} (-1)^{l(w)} \re^{w(\Lambda + \rho) - \rho}}.
            \end{equation}   
        By Proposition \ref{th:WLambdaThroughW} we have
        \begin{equation}
            S = \sum_{w \in \widehat{W}} (-1)^{l(w)} q^{(A_{p',\gamma} w A_{p',\gamma}^{-1}(\Lambda + \rho) - \rho,-\rho)}.
        \end{equation}
        Let us calculate 
        \begin{multline}
            (A_{p',\gamma} w A_{p',\gamma}^{-1}(\Lambda + \rho) - \rho,-\rho) 
            = -(\Lambda,\rho) + (-A_{p',\gamma} w A_{p',\gamma}^{-1}(\Lambda + \rho) + \Lambda+\rho, \rho) 
            \\
            \overset{(1)}{=}-(\Lambda,\rho) + (-w A_{p',\gamma}^{-1}(\Lambda + \rho) + A_{p',\gamma}^{-1}(\Lambda+\rho), A_{p',\gamma}^{-1}\rho) 
            \\
            \overset{(2)}{=}  -(\Lambda,\rho) + (-w (\eta + \rho)+ \eta + \rho,A_{p',\gamma}^{-1}\rho)
            \\
            \overset{(3)}{=} -(\Lambda,\rho) + (-w (\eta + \rho)+ \eta + \rho, \xi + \rho +  \frac{\mathtt{const}_2}{p'}\delta)
            \\
            \overset{(4)}{=} -(\Lambda,\rho) + (-w (\eta + \rho)+ \eta + \rho,\xi + \rho ).
        \end{multline}
        We have equality \((1)\) since \(A_{p',\gamma} \)  is isometry. Equality \((2)\) follows from  \eqref{eq:AinvLambda1} and \(\widehat{W}\)-invariance of \(\delta\).
        By \eqref{eq:AinvLambda2} we obtain equality \((3)\).
        For any \( v\in \widetilde{\mathfrak{h}}^*, w\in\widehat{W}\) we have \((-w v + v,\delta) = 0\), it implies equality \((4)\). 
        
        Let \(u = w^{-1} \),  then we have  
        \begin{equation}
            S = q^{-(\Lambda,\rho)} \sum_{u \in \widehat{W}}(-1)^{l(u)} q^{(-u^{-1}(\eta + \rho)+ \eta + \rho,\xi + \rho )} = q^{-(\Lambda,\rho)} \sum_{u \in \widehat{W}}(-1)^{l(u)} q^{(\eta + \rho,-u(\xi + \rho)+\xi + \rho )}.       
        \end{equation}

        \end{proof}

    In the principal specialization, the generating function of \(\widehat{\mathtt{GT}}_{\text{perm}}(\Lambda)\) has the form
    \begin{equation}\label{eq:GenFunGTlambdaPrSpec}
        P_{\Lambda}(q) = q^{-(\Lambda,\rho)}\sum_{\underline{d}\in \widehat{\mathtt{GT}}_{\text{perm}}(\Lambda)}q^{\sum_{i=1}^n d_{i}(\underline{d})}.
    \end{equation}

    \begin{proposition}
        For admissible weight \(\Lambda\) we have
        \begin{equation}\label{eq:PrSpecVSWalgCharacter}
            P_{\Lambda}(q) = \frac{q^{-(\Lambda,\rho)}}{(q)_\infty^n} \sum_{w\in \widehat{W}} (-1)^{l(w)} q^{(\xi + \rho - w(\xi+\rho),\eta + \rho)}.
        \end{equation}
    \end{proposition}
    \begin{proof}
        The statement is trivial for \(\Lambda=0\). The case \(\Lambda\neq 0\) follows from Theorems~\ref{th:prspec} and~\ref{th:PglhatIrr}.
    \end{proof}
    It is clear that     \begin{equation}\label{eq:PrSpecVSWalgCharacterTilde}
        P_{\Lambda}(q) = \frac{q^{-(\Lambda,\rho)}}{(q)_\infty^n} \sum_{w\in  \widehat{W}} (-1)^{l(w)} q^{(\tau\xi + \rho - w(\tau\xi+\rho),\tau\eta + \rho)}
    \end{equation}
    and that \(P_{\Lambda}(q)\) is a generating function of \(n\)-tuples of Young diagrams \(\tilde\lambda^{(i)}\) with conditions \eqref{eq:abconditions}. Thus, it gives another proof of \eqref{eq:PrSpecVSWalgCharacterTilde} originally proven  in \cite[Section 4]{feigin2011quantum} by recursion.

       %\subsection{Geometry of Laumon spaces? Compact components? Maybe here scalar product}

\section{\(q\)-deformed case}\label{sec:qcase}
In this section, we generalize the main results of the paper to the q-deformed setting. We fix \(q\in\mathbb{C}^{\times}\) to be not a root of unity.

    Let \(U_q\widehat{\mathfrak{sl}}_n\) be an associative algebra generated by \(e_i,f_i,q^{\pm h_i}\) for \(i\in\{0,\dots,n-1\}\) with relations 
    \begin{align}
         q^{h_i}q^{h_j} &=q^{h_j}q^{h_i},\\
        [e_i,f_j] &= \delta_{i,j} \frac{q^{h_i}-q^{-h_i}}{q-q^{-1}},\\
        q^{h_i} e_j q^{-h_i} &= e_j q^{C_{ij}}, \quad q^{h_i} f_j q^{-h_i} = f_j q^{-C_{ij}},\\ 
        [e_i, e_j] &=[f_i, f_j]= 0, \quad \text{if}\quad  C_{ij} = 0,\\
        [e_i,[e_i,e_j]_q]_{q^{-1}} &= [f_i,[f_i,f_j]_q]_{q^{-1}} = 0, \quad \text{if}\quad  C_{ij} = -1,
    \end{align}
    where \([a,b]_q = ab - q ba\). Denote by \(\mathcal{K} = q^{h_0}q^{h_1}\dots q^{h_{n-1}}\). 
    
    Denote by \(\mathtt{Heis}_q\) an algebra generated by \(Z_r\) for \(r\in\mathbb{Z}\setminus\{0\}\), \(\mathcal{K}', q^{\pm Z_0}\) with relations
    \begin{equation}
        [q^{Z_0}, \mathcal{K}'] = [q^{Z_0}, Z_r] = [\mathcal{K}',Z_r] = 0,\quad
        [Z_r,Z_s] =-\delta_{r+s,0}\frac{1 }{r}\frac{q^{nr+1}-q^{-nr-1}}{q-q^{-1}}\frac{\mathcal{K}'-(\mathcal{K}')^{-1}}{q-q^{-1}}.        
    \end{equation}
    Let \(U_q\widehat{\mathfrak{gl}}_n = (U_q\widehat{\mathfrak{sl}}_n\otimes \mathtt{Heis}_q)/(\mathcal{K}-\mathcal{K}')\).
    Similarly to the classical case there are Verma  \( U_{q}\widehat{\mathfrak{gl}}_n\)-modules \(\mathcal{M}_{\Lambda,u}\) and their unique simple quotients \(\mathcal{L}_{\Lambda,u}\) with action on highest weight vector\(~\zeta\)
    \begin{equation}
        q^{Z_0} \zeta = q^{u} \zeta, \quad q^{h_i} \zeta = q^{(\Lambda, h_i)} \zeta, \quad  Z_{r>0}\zeta = 0, \quad e_i \zeta = 0.
    \end{equation}
%    \s{a little problem with notation. \(u\) is overloaded, should I get rid of \(u,v\) in the definition toroidal algebra? it will make it less similar to notations in \cite{tsymbaliuk2010quantum}. Or it is better to use another style, for example         \(\mathfrak{v},\mathfrak{u}\)?}
    \begin{define}
        For \(d\in\mathbb{C}^{\times}\) an associative algebra  \(U_{q,d}(\ddot{\mathfrak{gl}}_n)\) over $\mathbb{C}$ is defined as algebra generated by $x^{\pm}_{i,r}$, $q^{\pm h_i}$, $h_{i,m}$
        $(1\leq i\leq n, r \in \mathbb{Z},  m\in \mathbb{Z}\setminus \{0\})$ with the
        following defining relations: 
        \begin{equation}
        \label{1} \psi_i^{s}(z)\psi_j^{s'}(w)=\psi_j^{s'}(w)\psi_i^{s}(z),
        \end{equation}
        \begin{equation}
        \label{3} [x_i^{+}(z), x_j^{-}(w)]=\frac{\delta_{ij}}{q-q^{-1}}
        \{\delta(w/z)\psi_i^{+}(w)-\delta(z/w)\psi_i^{-}(z) \},
        \end{equation}
        % \begin{equation}
        % \label{4} (z-q^{\pm
        % 2}w)x_k^{\pm}(z)x_k^{\pm}(w)=x_k^{\pm}(w)x_k^{\pm}(z)(q^{\pm 2}z-w)
        % \end{equation}
        \begin{equation}
        \label{5} (z-q^{\pm
         C_{i,j}}w)x_i^{\pm}(z)x_j^{\pm}(w)=x_j^{\pm}(w)x_i^{\pm}(z)(q^{\pm
         C_{i,j}}z-w),\quad %k\neq l,~
         (i,j)\neq (1,n),(n,1),
        \end{equation}
        \begin{equation}
        \label{2} (z-q^{\pm C_{i,j}}w)\psi_j^s(z)x_i^{\pm}(w)=x_i^{\pm}(w)
        \psi_j^s(z)(q^{\pm C_{i,j}}z-w), \quad  (i,j)\neq (1,n),(n,1), 
        \end{equation}
        \begin{equation}
        \label{tor4} ~'\!x_n^{\pm}(z)x_1^{\pm}(w)(z-q^{\mp 1}w)=(q^{\mp
        1}z-w)x_1^{\pm}(w)~'\!x_n^{\pm}(z),
        \end{equation}
        
        \begin{equation}
        \label{tor2.1} '\psi_n^s(z)x_1^{\pm}(w)(z-q^{\mp
        1}w)=x_1^{\pm}(w) '\psi_n^s(z)(q^{\mp 1}z-w),
        \end{equation}
        
        \begin{equation}
        \label{tor2.2} \psi_1^s(z)~'\!x_n^{\pm}(w)(z-q^{\mp
        1}w)=~'\!x_n^{\pm}(w) \psi_1^s(z)(q^{\mp 1}z-w),
        \end{equation}
        \begin{multline}
        \label{6} \{x_i^{s}(z_1)x_i^{s}(z_2)x_{i\pm
        1}^{s}(w)-(q+q^{-1})x_i^{s}(z_1)x_{i\pm
        1}^{s}(w)x_i^{s}(z_2)+
        \\
        x_{i\pm
        1}^{s}(w)x_i^{s}(z_1)x_i^{s}(z_2)\}+\{z_1\leftrightarrow z_2
        \}=0.   
        \end{multline}
        Here $~'\!x_n^{\pm}(z):=x_n^{\pm}(z d^n),\
        '\psi_n^{\pm}(z)=\psi_n^{\pm}(z d^n)$. Here $s,s'\in\{+,-\}$.  Generating functions   $\delta (z), x_k^{\pm}(z), \psi_k^{\pm}(z)$
        are defined as following
        $$\delta(z):=\sum_{r=-\infty}^\infty z^r,\
        x_k^{\pm}(z):=\sum_{r=-\infty}^\infty x^{\pm}_{k,r}z^{-r},$$
        $$\psi_k^{\pm}(z):=q^{\pm h_k} \exp \left(\pm(q-q^{-1})\sum_{m=1}^\infty
        h_{k,\pm m}z^{\mp m}\right).$$
    \end{define}
    %\zh{Is  (5.10) just a special case of (5.12) when $k=l$?}
    \begin{remark}  
        %For $v = q$ and $u=d^{\frac{n}{2}}q^{-\frac{n}{2}}$  d = q^{-1-\frac{2k}{n}}, 
        There is an isomorphism of algebras 
        \begin{equation}
            \Phi:U_{q,d}%_{v,u}
            (\ddot{\mathfrak{gl}}_n)\rightarrow \mathcal{E}_n(q,d),
        \end{equation}
where \(\mathcal{E}_n(q,d)\) is the quantum toroidal algebra in \cite{feigin2013representations} . The map \(\Phi\) is defined by 
        \begin{equation}
            \Phi(x_i^{+}(z))=E_{i}(d^{-i}z),\quad \Phi(x_i^{-}(z))=F_{i}(d^{-i}z),\quad \  \Phi(\psi_i^{\pm}(z))=K_{i}^{\pm}(d^{-i}z).
        \end{equation}    
    \end{remark}
    We denote \(\mathtt{Y}_j = q^{\mathtt{y}_j}\) and \(P_{i,j}:=\mathtt{Y}_{j}^2 q^{-2d_{ij}}= q^{2(\mathtt{y}_j-d_{ij})} = q^{-2p_{ij}}\). Recall that \(\mathbb{M}_{\Lambda,u}, \mathbb{L}_{\Lambda,u}\) are defined by formulas \eqref{eq:GLu},\eqref{eq:PLu}.

%\begin{equation}
%    \mathtt{Y}_{j+n} = \mathtt{Y}_j u
%\end{equation}

%\begin{equation}
%    \kappa_i = \frac{\mathtt{Y}_i}{\mathtt{Y}_{i+1}v},\quad \prod_{i=0}^{n-1}\kappa_i = u^{-1}q^{-n}
%\end{equation}

\begin{theorem}[{\cite[Prop.4.15]{tsymbaliuk2010quantum}}]
For generic \(\Lambda\) there is an action of  \(U_{q,q^{-1-2k/n}}(\ddot{\mathfrak{gl}}_n)\) %\mathcal{U}_{q,q^{-\kappa/n}}(\ddot{\mathfrak{sl}}_n)
on  \(\mathbb{M}_{\Lambda,u}\) defined by %formulas
    \begin{subequations}
        \begin{align}\label{eq:qanalog}
                x^-_{i,r}\xi_{\underline{d}}&=-\sum_{j\le i}\xi_{\underline{d}^{+}_{i,j}}
                \mathtt{Y}_i^{-1}%u^{-\delta_{i,n}}
                q^{d_i-d_{i-1}+i} P_{i,j}(P_{i,j}q^i)^r
                (1-q^2)^{-1}\frac{
                \prod_{s\leq i-1}(1-P_{i,j}P_{i-1,s}^{-1})}{\prod_{j\ne s\leq i}(1-P_{i,j}P_{i,s}^{-1})},\\
                x^+_{i,r}\xi_{\underline{d}}&=\sum_{j\le i}\xi_{\underline{d}^{-}_{i,j}}
                \mathtt{Y}_{i+1}^{-1}q^{d_{i+1}-d_i+1-i}(P_{i,j}q^{i+2})^r(1-q^2)^{-1}\frac{ \prod_{s\leq
                i+1}(1-P_{i+1,s}P_{i,j}^{-1})}{\prod_{j\ne
                s\leq i}(1-P_{i,s}P_{i,j}^{-1})},\\
                \psi_i^{\pm}(z)\xi_{\underline{d}} &= \xi_{\underline{d}} \frac{\mathtt{Y}_i q^{d_{i+1}-2d_{i}+d_{i-1}-1}}{\mathtt{Y}_{i+1}}\frac{\prod_{j\le i+1}(1-z^{-1}q^{i+2}P_{i+1,j})\prod_{j\le i-1}(1-z^{-1}q^iP_{i-1,j})}{\prod_{j\le i}(1-z^{-1}q^{i+2}P_{i,j})(1 -z^{-1}q^iP_{i,j})},
        \end{align}            
    \end{subequations}
    where the last expression is expanded in $z^{\mp1}$ for $ \psi_i^{\pm}$.       \qed
\end{theorem}
\begin{remark}
    The module \(\mathbb{M}_{\Lambda,u}\) is a highest weight \(U_{q,q^{-1-2k/n}}(\ddot{\mathfrak{gl}}_n)\)-module with  highest weight vector \(\xi_{\underline{0}}\) and weights
    \begin{equation}\label{eq:DrinfToroidal}
        Q_i(z) = \frac{\mathtt{Y}_i}{\mathtt{Y}_{i+1} q}\frac{1 - z^{-1} q^{i+2}\mathtt{Y}_{i+1}^2}{1 - z^{-1} q^{i}\mathtt{Y}_{i}^2}
    \end{equation}
   % \zh{some confusion in the weight here, maybe $v\to q$ and $Y_{i+1},Y_i\to Y_{i+1}^2,Y_i^2$ in the second factor?}
    i.e.
    \begin{equation}
        \psi_i^{\pm}(z) \xi_{\underline{0}} = Q_i(z) \xi_{\underline{0}}, \quad\quad    x_i^{+}(z) \xi_{\underline{0}} = 0. 
    \end{equation}
\end{remark}
\begin{theorem}
    \begin{enumerate}
        \item({\cite[Thm.3.1]{Miki1999}}) There is a homomorphism 
    \begin{equation}
        \operatorname{ev}^{(1)}:U_{q,q^{-1-2k/n}}(\ddot{\mathfrak{gl}}_n)\rightarrow (U_{q}\widehat{\mathfrak{gl}}_n)_{\text{comp}}/(\mathcal{K} - q^{k}), 
    \end{equation}
    where \( (U_{q}\widehat{\mathfrak{gl}}_n)_{\text{comp}}\) is a properly defined completion of algebra \( U_{q}\widehat{\mathfrak{gl}}_n\). 
    \item({\cite[Thm.5.1]{feigin2020evaluation}}) The module \((\operatorname{ev}^{(1)})^*\mathcal{L}_{\Lambda,u}\) has highest weights \eqref{eq:DrinfToroidal}.        
    \end{enumerate}
        \qed
\end{theorem}
Note that in \cite{feigin2020evaluation} there is also another homomorphism \(\operatorname{ev}^{(3)}\).
\begin{theorem}\label{th:qcase}
    For a dominant weight \(\Lambda\)  with \((\Lambda,\delta)=k\neq 0\) formulas \eqref{eq:qanalog} define an action \(U_{q, q^{-1-2k/n}}(\ddot{\mathfrak{gl}}_n)\) on \(\mathbb{L}_{\Lambda,u}\) and \(\mathbb{L}_{\Lambda,u}\simeq (\operatorname{ev}^{(1)})^*\mathcal{L}_{\Lambda,u} \) as \(U_{q, q^{-1-2k/n}}(\ddot{\mathfrak{gl}}_n)\)-module.%\zh{Change sl to gl}  
\end{theorem}
\begin{proof}
    The proof is analogous to the proof of Theorem \ref{th:Main}.
\end{proof}

\begin{remark}
    There are isomorphisms \(\iota_{\text{tor}}:\mathcal{E}_n(q,d)\rightarrow \mathcal{E}_n(q,d^{-1})\)  mapping \(E_i(z), F_i(z), K^{\pm}_i(z)\) to \(E_{n-i}(z), F_{n-i}(z), K^{\pm}_{n-i}(z)\) correspondingly and  \(\iota_{\text{aff}}:U_{q}\widehat{\mathfrak{gl}}_n\rightarrow U_{q}\widehat{\mathfrak{gl}}_n\) mapping \(e_i, f_i, q^{\pm h_i}\) to \(e_{n-i}, f_{n-i}, q^{\pm h_{n-i}}\) and preserving  \(\mathtt{Heis}_q\) . Then 
    \begin{equation}
        \operatorname{ev}^{(1)} \circ\iota_{\text{tor}}  = \iota_{\text{aff}}\circ \operatorname{ev}^{(3)}.
    \end{equation}
    Note that on the level of roots \(\iota_{\text{aff}}\) corresponds to  \(\tau\) defined by formula \eqref{eq:tau}. So, renumerating generators in formulas \eqref{eq:qanalog} or \eqref{eq:ActionInGTBasisYangian}, one can get rid of \(\tau\) in formulas \eqref{eq:tauetaxi}  and corresponding tilde over \(a_i,b_i\).
\end{remark}
    % \begin{remark}
    %     In notations of \cite{feigin2013representations}  admissible module over \(\mathcal{U}_{\mathfrak{q}}(\widehat{\mathfrak{gl}}_n)\) could be considered as a subquotient of MacMahon module \( \mathcal{M}^{(n-\nu^t_1)}_{\emptyset,\mu^t,\nu^t}(u,K) \) where \(K^2 =  \mathfrak{q}_3^{n}, ~~ \mathfrak{q}_2^{p}\mathfrak{q}_1^{p'n} = 1\) over toroidal algebra  \(\mathcal{U}_{q,d}(\ddot{\mathfrak{sl}}_n)\). Here we use notation \(\mathfrak{q}_1 = d/q,\mathfrak{q}_2=q^2,\mathfrak{q}_3=1/dq\)
    %     and \(\mu,\nu\) are Young diagrams defined by formulas 
    %     \begin{align}
    %         \mu_i -\mu_{i+1} &= \tilde 
    %         a_i  \quad \text{for}\quad i \in \{1,\dots n-1\},\\
    %         \mu_i&=0,\quad  \text{for}\quad i\ge n-1,\\
    %         \nu_i -\nu_{i+1} &= \tilde 
    %         b_i  \quad \text{for}\quad i \in \{1,\dots n-1\},\\
    %         \nu_i&=0,\quad  \text{for}\quad i\ge n-1.         
    %     \end{align}
    %     Note that existence of \(\sigma\in \mathfrak{S}_n\) such that 
    %     \begin{equation}
    %         \sigma(i) = g_i(\vec{b}) ~\operatorname{mod} ~ n
    %     \end{equation}
    %     is equivalent to \(\nu\) be colorless partition.
    %     Up to  \(\iota_{\text{tor}}\) and limit \(a_0,b_0\rightarrow \infty\) these modules coincide with modules in \cite{feigin2013representations}[Sec. 5]. 
    % \end{remark}

    \begin{remark}
        One of particular cases of Theorem \ref{th:qcase} is considered \cite{feigin2013representations}[Sec. 5]. It is the limiting case of admissible \(U_{\mathfrak{q}}\widehat{\mathfrak{gl}}_n\)-modules with \(a_0,b_0\rightarrow \infty\) with fixed \(a_i,b_i\) for \(i\in\{1,\dots, n-1\}\). In particular, the character formula in \cite{feigin2013representations}[Prop.5.5] is a special case of Theorem \ref{th:char}.
    \end{remark}

%\section{Discussion}\label{sec:Discussion}

% \begin{enumerate}
%     \item Affine Laumon space is a fixed point set of Gieseker with respect to \(\mathbb{Z}/n\mathbb{Z}\) action (see \cite{finkelberg2010quantization}[Sec.2.2],
%     \item Fixed point set with respect to \(\mathbb{T}=(\mathbb{C}^\times)^{n+2}\) on Gieseker  is enumerated by  \(\widehat{\mathtt{GT}}\)
%     \item Character of torus on Gieseker (NY) + Restriction on Laumon.
%     \item Atiyah-Bott, Poincare pairing
%     \item Specialization of \(\Lambda\) correspond to restriction of torus \(\mathbb{T}\) to smaller subtorus \(\mathbb{T}'\)
%     \item Under this restriction Poincare pairing becomes degenerate. We still have Atiyah-Bott theorem, but the components of invariant points becomes more complicated. There are \(3\) types of components:
%     \begin{itemize}
%         \item isolated points
%         \item compact components of dimension \(>0\)
%         \item noncompact components.
%     \end{itemize}
   
%     \end{enumerate}
\section{Geometry}\label{sec:geometry}
In this section we recall the geometric origin of the modules considered above. One has the affine Yangian action on equivariant homology of affine Laumon spaces and the fixed-point basis is labeled by
affine Gelfand--Tsetlin patterns (see \cite{feigin2011yangians}). We use it to give a sketch of an another proof of Theorem \ref{th:Main} and describe a natural inner product on \(\mathbb{L}_{\Lambda_{\text{dom},u}}\).

Let \(\mathcal{P}_{\vec{d}}\) be the affine Laumon space considered in \cite{feigin2011yangians}, along with the action of the torus
\(
\mathbb T=(\mathbb C^\times)^{n+2}
\)
on \(\mathcal{P}_{\vec{d}}\). 
The \(\mathbb T\)-fixed points in \(\mathcal{P}_{\vec{d}}\) are isolated and naturally indexed by patterns
\(
\underline d\in\widehat{\mathtt{GT}}
\)
with \(\vec d=(d_1(\underline d),\dots,d_n(\underline d))\)
\begin{equation}
    \underline d
    \longleftrightarrow
    z_{\underline d}\in \mathcal{P}_{\vec{d}}.
\end{equation}
        
The equivariant homology space
\begin{equation}
H_*^{\mathbb T}\!\left(
\bigsqcup\nolimits_{\vec d\in\mathbb Z_{\ge0}^n}
\mathcal{P}_{\vec{d}}
\right)
\end{equation} is naturally a module over the $\mathbb{T}$-equivariant cohomology ring of a point, that is, the polynomial ring in $n+2$ equivariant parameters \((x_1,\dots,x_n,\hbar,\hbar')\).
The equivariant Poincaré duality gives a $\mathbb C[x_1,\dots,x_n,\hbar,\hbar']$-module morphism
\begin{equation}
H_*^{\mathbb T}
\longrightarrow
H^*_{\mathbb T},
\end{equation}
which together with the natural pairing
\begin{equation}
H^*_{\mathbb T}\!\left(
\bigsqcup\nolimits_{\vec d}\mathcal{P}_{\vec{d}}
\right)
\otimes
H_*^{\mathbb T}\!\left(
\bigsqcup\nolimits_{\vec d}\mathcal{P}_{\vec{d}}
\right)
\longrightarrow
\mathbb C[x_1,\dots,x_n,\hbar,\hbar']
\end{equation}
defines a bilinear form
\(
\langle\cdot,\cdot\rangle
\)
on equivariant homology.

We use the Atiyah--Bott localization theorem \cite{atiyah1984moment} in the following form.

\begin{theorem}
One has
\begin{equation}
H_*^{\mathbb T}\!\left(
\bigsqcup\nolimits_{\vec d}\mathcal{P}_{\vec{d}}
\right)\otimes_{\mathbb C[x_1,\dots,x_n,\hbar,\hbar']}\mathbb C(x_1,\dots,x_n,\hbar,\hbar')
=\bigoplus_{\underline d\in\widehat{\mathtt{GT}}}
\mathbb C(x_1,\dots,x_n,\hbar,\hbar')
\,\xi_{\underline d},
\end{equation}
where \(\xi_{\underline d}\) is the class of the fixed point \(z_{\underline d}\). The basis \(\{\xi_{\underline d}\}\) is orthogonal and
\begin{equation}
\left\langle
\xi_{\underline d},
\xi_{\underline d}
\right\rangle
=
\mathtt{e}_{\mathbb{T}}(T_{z_{\underline d}}\mathcal{P}_{\vec{d}}).
\end{equation}
\end{theorem}

Fix generic \(\Lambda,u\) and set
\begin{equation}
(x_1,\dots,x_n,\hbar,\hbar')
=
(\mathtt y_1\hbar,\dots,\mathtt y_n\hbar,\hbar,-\kappa\hbar).
\end{equation}
Then there exist polynomials
\(
\mathcal N_{\underline d}(\vec{\mathtt y},\kappa)
\)
such that
\begin{equation}
\label{eq:fixedpointpairing}
\mathtt{e}_{\mathbb{T}}(T_{z_{\underline d}}\mathcal{P}_{\vec{d}})
=
\hbar^{2\sum_{i=1}^n d_i}
\mathcal N_{\underline d}(\vec{\mathtt y},\kappa).
\end{equation}
        
Recall that every pattern \(\underline d\in\widehat{\mathtt{GT}}\) corresponds to an \(n\)-tuple of Young diagrams
\(\{\lambda^{(a)}\}_{a=1}^n\) via \eqref{eq:nYoungDiagrams}.
For a box \(s=(i,j)\), denote by \(a_\lambda(s)\) and \(l_\lambda(s)\) its arm and leg lengths in the Young diagram~\(\lambda\).

\begin{theorem}
One has
\begin{equation}\label{eq:NY}
\begin{aligned}
\mathcal N_{\underline d}(\vec{\mathtt y},\kappa)
=
\prod_{i,j=1}^n
&
\prod_{\substack{
s\in\lambda^{(i)}\\
j-i+l_{\lambda^{(j)}}(s)+1\equiv0\!\!\!\pmod n
}}
\left(
\mathtt y_i-\mathtt y_j
-\frac{j-i+l_{\lambda^{(j)}}(s)+1}{n}\kappa
-a_{\lambda^{(i)}}(s)
\right)
\\
\times\,
&
\prod_{\substack{
s\in\lambda^{(j)}\\
j-i-l_{\lambda^{(i)}}(s)\equiv0\!\!\!\pmod n
}}
\left(
\mathtt y_i-\mathtt y_j
-\frac{j-i-l_{\lambda^{(i)}}(s)}{n}\kappa
+a_{\lambda^{(j)}}(s)+1
\right).
\end{aligned}
\end{equation}
\end{theorem}

\begin{proof}
This follows from the computation of tangent characters for the Gieseker moduli space \(\mathfrak M_{n,d}\) in \cite{nakajima2005instanton} together with the realization of affine Laumon spaces as \(\mathbb Z/n\mathbb Z\)-fixed loci \cite[Sec.~2.2]{finkelberg2010quantization}.
\end{proof}

Since all expressions are homogeneous in \(\hbar\), we set \(\hbar=1\). Then for generic \(\Lambda\),
\begin{equation}
    \mathbb{M}_{\Lambda,u}
    \simeq
    \Bigg(H_*^{\mathbb T}\big(
    \bigsqcup\nolimits_{\vec d}\mathcal{P}_{\vec{d}}
    \big)\Bigg)_{\hbar = 1}.
\end{equation}
Under this identification, the Yangian generators \(\mathbf x^{\pm}_{i,r}\) from Theorem~\ref{th:YactsonGTVerma}
are realized geometrically by correspondences \cite[Sect.~2.3]{feigin2011yangians} inside
% \begin{equation}
\(
\mathcal{P}_{(d_1,\dots,d_i,\dots,d_n)}
\times
\mathcal{P}_{(d_1,\dots,d_i\mp1,\dots,d_n)}.
\)
% \end{equation}
The transposed correspondences define adjoint operators with respect to \(\langle\cdot,\cdot\rangle\). Hence
\begin{equation}\label{eq:conjugation}
    (\mathbf x^{\pm}_{i,r})^*
    =
    \mathbf x^{\mp}_{i,r}.
\end{equation}

\bigskip

Specializing \(\Lambda\) to \(\Lambda_{\text{dom}}\) corresponds to the choice of one-parameter subgroup in \(\mathbb{T}\). 
By \(\mathbb T_{\text{dom}}\subset \mathbb{T}\) we denote its Zarisky closure. 
The \(\mathbb T_{\text{dom}}\)-fixed locus is no longer necessarily discrete: its irreducible components may be isolated points, compact positive-dimensional varieties, or noncompact varieties. In particular, the Poincaré duality map
\begin{equation}
H_*^{\mathbb T_{\text{dom}}}(\mathcal{P}_{\vec{d}})
\longrightarrow
H^*_{\mathbb T_{\text{dom}}}(\mathcal{P}_{\vec{d}})
\end{equation}
is in general neither injective nor surjective. 
Note that any irreducible component of \(\mathbb T_{\text{dom}}\)-fixed locus has to contain at least one of the points \(z_{\underline{d}}\) for \(\underline{d}\in\widehat{\mathtt{GT}}\).

Both \(\bigoplus\limits_{\vec{d}}H_*^{\mathbb T_{\text{dom}}}(\mathcal{P}_{\vec{d}})\) and \(\bigoplus\limits_{\vec{d}}H^*_{\mathbb T_{\text{dom}}}(\mathcal{P}_{\vec{d}})\) still carry affine Yangian actions.

\begin{claim}\label{conj1}
    One has
    \begin{equation}
    \bigoplus\limits_{\vec{d}}H_*^{\mathbb T_{\text{dom}}}(\mathcal{P}_{\vec{d}})
    \simeq
    \mathrm{ev}_{-}^*\mathcal M_{\Lambda,u},
    \qquad
    \bigoplus\limits_{\vec{d}}H^*_{\mathbb T_{\text{dom}}}(\mathcal{P}_{\vec{d}})
    \simeq
    \mathrm{ev}_{-}^*\mathcal M^\vee_{\Lambda,u},
    \end{equation}
    where \(\mathcal M^\vee_{\Lambda,u}\) is the contragredient dual Verma module.      
\end{claim}
\begin{remark} A
\(K\)-theoretic analogue of the second half of Claim \ref{conj1} is proven in \cite{shen2024affine}.     
\end{remark}

%\s{ Should we add something about sum of units in cohomologies as Whittaker vector?}

The Poincaré duality map is therefore the unique homomorphism
\begin{equation}
\phi:
\mathrm{ev}_{-}^*\mathcal M_{\Lambda,u}
\longrightarrow
\mathrm{ev}_{-}^*\mathcal M^\vee_{\Lambda,u}.
\end{equation}
Hence its image is
\(
\mathrm{ev}_{-}^*\mathcal L_{\Lambda,u}
\).
By Atiyah--Bott localization,
\begin{equation}
\mathrm{im}\,\phi
\simeq
\bigoplus_{Z_{\mathrm c}}
H^*_{\mathbb T_{\text{dom}}}(Z_{\mathrm c}),
\end{equation}
where \(Z_{\mathrm c}\) runs over compact fixed components. Similarly,
\begin{equation}
\ker\phi
=
\bigoplus_{Z_{\mathrm{nc}}}
H_*^{\mathbb T_{\text{dom}}}(Z_{\mathrm{nc}}),
\end{equation}
where \(Z_{\mathrm{nc}}\) runs over noncompact fixed components.
\begin{proposition}\label{prop:isolated points = permited}
Let \(\Lambda_{\mathrm{dom}}\) be dominant with \(k\neq0\). Then
\begin{equation}
\mathcal N_{\underline d}(\vec{\mathtt y},\kappa)\neq0
\quad\Longleftrightarrow\quad
\underline d\in
\widehat{\mathtt{GT}}_{\mathrm{perm}}
(\Lambda_{\mathrm{dom}}).
\end{equation}
\end{proposition}
\begin{proof}
This follows by direct inspection of formula \eqref{eq:NY}.
\end{proof}
\begin{corollary}
    Any noncompact \(Z_{\text{nc}}\subset \mathcal{P}_{\vec{d}}^{\mathbb T_{\text{dom}}}\) does not contain \(z_{\underline{d}}\) for \(\underline{d}\in\widehat{\mathtt{GT}}_{\text{perm}}(\Lambda_{\text{dom}})\). 
\end{corollary}
\begin{claim}\label{conj2}
    For dominant \(\Lambda_{\mathrm{dom}}\) with \(k\neq 0\) any irreducible component \(Z\subset \mathcal{P}_{\vec{d}}^{\mathbb T_{\text{dom}}}\) containing  \(z_{\underline{d}}\) for \(\underline{d}\not\in\widehat{\mathtt{GT}}_{\mathrm{perm}}(\Lambda_{\text{dom}})\) is noncompact.     
\end{claim}
\begin{proof}[Proof modulo Claim \ref{conj1}]
Since \(\operatorname{im}\phi=\mathcal{L}_{\Lambda,u}\), the  character
\(\chi(\operatorname{im}\phi)\) is the generating function of all compact components.
On the other hand, by Theorem~\ref{th:Main} and Proposition~\ref{prop:isolated points = permited}, \(\chi(\mathcal{L}_{\Lambda,u})\) is the
generating function of isolated fixed points. Therefore, there are no compact
components of positive dimension.
\end{proof}
% \begin{proof}[Sketch of the Proof]
%    To prove this claim recall that Affine Laumon Space is eqiuivariant resolution  of singularity \(\pi: L_{\vec{d}}\rightarrow\mathcal{Z}_{\vec{d}} \) where  \(\mathcal{Z}_{\vec{d}}\) is Affine Zastava space \s{(reference?)}. \(\mathcal{Z}_{\vec{d}}\) is quasiaffine. The resolution map is proper. 0 is a unique invariant with repect to \(\mathbb{T}\). So, any irreducible component of fixed point set in \(\mathcal{Z}_{\vec{d}}\) with respect to \(\mathbb{T}_{\text{dom}}\) contains \(0\). Thus, any compact irreducible component of \(L^{\mathbb{T}_{\text{dom}}}\) lies in \(\pi^{-1}(0)\).

%    Moreover, there is a stratification of \(\pi^{-1}(0)\) by affine spaces \(\mathcal{A}_{\underline{d}}\) enumerated by \(\underline{d}\in\widehat{\mathtt{GT}}\). The action of \(\mathbb T\) respects stratification and linear on \(\mathcal{A}_{\underline{d}}\). Note that \(\mathcal{A}_{\underline{d}}^{\mathbb{T}}=\{z_{\underline{d}}\}\). If \(Z\cap\!\mathcal{A}_{\underline{d}}\neq\varnothing\) then \(z_{\underline{d}}\in Z\). So, one has to show that for \(\underline{d}\not\in\widehat{\mathtt{GT}}_{\text{perm}}(\Lambda_{\text{dom}})\) we have \((T_{z_{\underline{d}}}L_{\vec{d}})^{\mathbb{T}_{\text{dom}}}\not\subset~T_{z_{\underline{d}}} \mathcal{A}_{\underline{d}}\). 
% \end{proof}
% Now, it is clear that that \(\mathrm{ev}_{-}^*\mathcal L_{\Lambda,u}\) is spanned by classes of isolated fixed points. These are precisely the points \(z_{\underline d}\) such that \(\mathcal N_{\underline d}(\vec{\mathtt y},\kappa)\neq0\).
\begin{remark}
    Claims~\ref{conj1} and~\ref{conj2}, together with Proposition~\ref{prop:isolated points = permited}, provide an alternative proof of Theorems~\ref{th:Main} and~\ref{th:PglhatIrr} using the geometry of the Laumon space. Moreover, they endow \(\mathbb{L}_{\Lambda_{\mathrm{dom}},u}\) with a nondegenerate bilinear form.
\end{remark}

\begin{theorem}
The pairing \(\langle\cdot,\cdot\rangle\) on
\(
\mathbb{L}_{\Lambda_{\mathrm{dom}},u}
\)
coincides with the Shapovalov form of \(\widehat{\mathfrak{gl}}_n\).
\end{theorem}

\begin{proof}
By \eqref{eq:conjugation}, the generators \(\mathbf x^+_{i,r}\) and \(\mathbf x^-_{i,r}\) are adjoint with respect to the pairing. It remains to check the Heisenberg generators.

Let \(a_m\) be the Heisenberg operators introduced in \cite[Theorem~4.18]{kodera2019braid}. Using their realization in terms of Yangian generators, one checks that
\begin{equation}
(a_m)^*=a_{-m}.
\end{equation}
Hence the pairing satisfies the defining %adjointness
properties of the Shapovalov form.
\end{proof}

\printbibliography[
heading=bibintoc,
title={References}
]

\noindent \textsc{Scuola Internazionale Superiore di Studi Avanzati (SISSA) Via Bonomea 265, Trieste, Italy,
}

\emph{E-mail}:\,\,\textbf{mbersht@sissa.it}\\

\noindent \textsc{Department of Mathematical Sciences, Indiana University Indianapolis, 402 N. Blackford
St., LD 270, Indianapolis, IN 46202, USA,}

\emph{E-mail}:\,\,\textbf{emukhin@iu.edu}\\

\noindent\textsc{Department of Mathematics and Statistics, University of Montreal, Montreal QC, Canada
	}

\emph{E-mail}:\,\,\textbf{leonid.rybnikov@umontreal.ca}\\

\noindent\textsc{Department of Mathematics and Statistics, University of Montreal, Montreal QC, Canada
	}

\emph{E-mail}:\,\,\textbf{trufaleks2022@gmail.com}\\

\end{document}